%% file: HFTSymmetries.tex
\documentclass[11pt,a4paper,twoside]{amsart}
\usepackage[top=1in, bottom=1.25in, marginparwidth=1in, inner=1in, outer=1in]{geometry}

\usepackage{lmodern}
\usepackage[T1]{fontenc}
\usepackage[utf8]{inputenc}
\usepackage{inputenc}
\usepackage[UKenglish]{babel}
\usepackage{yfonts}
\usepackage{bm}
\usepackage{booktabs}

\usepackage{calc}
\usepackage{float}
\usepackage{graphicx}
\usepackage{xcolor}
\colorlet{myblue}{black}
\colorlet{lightred}{red!20!white}
\colorlet{lightblue}{blue!30!white}
\colorlet{darkgreen}{green!70!black}
\colorlet{lightgreen}{green!50!white}
\colorlet{gold}{yellow!90!black!70!red}

\usepackage{rotating}
\usepackage{tikz}
\usepackage{tikz-cd}
\usetikzlibrary{shapes,decorations}

\usepackage{caption}
\usepackage[labelfont=up,margin=5pt]{subcaption}
\usepackage{wrapfig}
\newcommand{\myfixwrapfig}{~\vspace*{-\baselineskip}\vspace*{0pt}}

\usepackage[section]{placeins}

\usepackage[bookmarks=true,pdfencoding=auto]{hyperref}
\usepackage{bookmark}
\hypersetup{colorlinks=true,citecolor=black,urlcolor=myblue,linkcolor=black}

\usepackage{setspace} 
\usepackage{amsmath}
\usepackage{amsfonts}
\usepackage{amssymb}
\usepackage{amsthm}
\usepackage{mathtools}
\usepackage{mathrsfs}
\mathtoolsset{mathic=true}

\def\co{\colon\thinspace\relax}

\newtheorem{theorem}{Theorem}[section]
\newtheorem*{theorem*}{Theorem}
\newtheorem{lemma}[theorem]{Lemma}

\newtheorem{question}[theorem]{Question}

\newtheorem{proposition}[theorem]{Proposition}
\theoremstyle{definition}
\newtheorem{definition}[theorem]{Definition}
\newtheorem{example}[theorem]{Example}

\newtheorem{observation}[theorem]{Observation}
\newtheorem{Remark}[theorem]{Remark}

\usepackage[makeroom]{cancel}
\usepackage{enumerate}
\usepackage[normalem]{ulem} 

\def\co{\colon\thinspace}

\input{sections/0_notation.tex}
\input{sections/0_pics.tex}

\begin{document}

\title[On symmetries of peculiar modules]{On symmetries of peculiar modules\\
\(\text{\textnormal{or}}\)\\
\(\delta\)-graded link Floer homology is mutation invariant}
\author{Claudius Bodo Zibrowius}
\address{Department of Mathematics,
	The University of British Columbia,
	1984 Mathematics Road,
	Vancouver, BC,
	Canada, V6T 1Z2}
\email{claudius.zibrowius@posteo.net}
\urladdr{\url{https://cbz20.raspberryip.com/}}

\begin{abstract}
	We investigate symmetry properties of peculiar modules, a Heegaard Floer invariant of 4-ended tangles which the author introduced in~\cite{PQMod}. 
	In particular, we give an almost complete answer to the geography problem for components of peculiar modules of tangles. 
	As a main application, we show that Conway mutation preserves the hat flavour of the relatively $\delta$-graded Heegaard Floer theory of links.
\end{abstract}
\maketitle

\renewcommand{\contentsname}{Table of contents}
\setcounter{tocdepth}{1}
\setcounter{section}{-1}
\tableofcontents

\input{sections/Intro}
\input{sections/Background}
\input{sections/MCG}
\input{sections/Linearity}

\input{sections/Stabilization}
\input{sections/HalfIdentityBimodule}
\input{sections/ConjugationBimodule}

\newcommand*{\arxiv}[1]{\href{http://arxiv.org/abs/#1}{ArXiv:\ #1}}
\newcommand*{\arxivPreprint}[1]{\href{http://arxiv.org/abs/#1}{ArXiv preprint #1}}
\newcommand*{\arxivANC}{available as an ancillary file to this preprint on \href{https://arxiv.org/src/1909.04267/anc}{arXiv:\ 1909.04267}}
\bibliographystyle{alpha}
\bibliography{HFTSymmetries}

\end{document}

%% file: sections/0_notation.tex
\newcommand{\pp}[2]{(\textcolor{blue}{ p_{#1}}\vert\textcolor{red}{p_{#2}})}

\newcommand{\qq}[2]{(\textcolor{blue}{ q_{#1}}\vert\textcolor{red}{q_{#2}})}
\newcommand{\pn}[1]{(\textcolor{blue}{ p_{#1}}\vert\textcolor{red}{1})}
\newcommand{\qn}[1]{(\textcolor{blue}{ q_{#1}}\vert\textcolor{red}{1})}
\newcommand{\np}[1]{(\textcolor{blue}{ -}\vert\textcolor{red}{p_{#1}})}
\newcommand{\nq}[1]{(\textcolor{blue}{ -}\vert\textcolor{red}{q_{#1}})}

\newcommand{\Alex}[4]{\prescript{#1}{#2}{t}^{#4}_{#3}}
\DeclareMathOperator{\AlexGr}{\mathfrak{A}}

\newcommand{\ob}{\operatorname{ob}}
\DeclareMathOperator{\Mor}{Mor}
\DeclareMathOperator{\Mod}{Mod}

\DeclareMathOperator{\pqMod}{pqMod}

\newcommand{\QGrad}[1]{{\textcolor{violet}{#1}}}

\DeclareMathOperator{\CFTd}{CFT^\partial}

\DeclareMathOperator{\HFT}{HFT}
\DeclareMathOperator{\HFTdelta}{HFT^\delta}
\newcommand{\CFTminus}{\operatorname{CFT}^-}
\newcommand{\Rat}{\mathfrak{r}}
\newcommand{\Irr}{\mathfrak{i}}

\DeclareMathOperator{\rr}{r}
\DeclareMathOperator{\mut}{mut}
\DeclareMathOperator{\mr}{mr}

\DeclareMathOperator{\HFL}{\widehat{HFL}}

\DeclareMathOperator{\BSD}{\widehat{BSD}}

\DeclareMathOperator{\BSAD}{\widehat{BSAD}}

\newcommand{\typeA}[2]{\prescript{}{#1}{#2}}

\DeclareMathOperator{\LagrangianFH}{HF}

\DeclareMathOperator{\Spinc}{Spin^c}
\DeclareMathOperator{\ConjBimod}{\typeA{31}{CB^{13}}}

\DeclareMathOperator{\Homology}{H_\ast}

\newcommand{\Ad}{\operatorname{\mathcal{A}}^\partial}

\newcommand{\Apre}{\operatorname{\mathcal{A}}^\text{\normalfont pre}}
\newcommand{\Rpre}{\operatorname{R}^\text{pre}}
\newcommand{\R}{\operatorname{R}}
\newcommand{\Aminus}{\operatorname{\mathcal{A}}^-}
\newcommand{\Id}{\operatorname{\mathcal{I}}^\partial}

\newcommand{\Braidgroup}{B_3}

\newcommand{\field}{\mathbb{F}_2}

\DeclareMathOperator{\id}{id}

\DeclareMathOperator{\GL}{GL}
\DeclareMathOperator{\PSL}{PSL}
\DeclareMathOperator{\SL}{SL}
\DeclareMathOperator{\im}{im}

\newcommand{\FourPuncturedSphere}{S^2\smallsetminus\textnormal{4}}
\newcommand{\FourPuncturedTorus}{T^2\smallsetminus\textnormal{4}}
\newcommand{\PuncturedPlane}{\mathbb{R}^2\smallsetminus \mathbb{Z}^2}
\newcommand{\Plane}{\mathbb{R}^2}
\newcommand{\Lattice}{\mathbb{Z}^2}
\newcommand{\Torus}{T^2}
\newcommand{\Sphere}{S^2}
\newcommand{\QPI}{\operatorname{\mathbb{Q}P}^1}
\newcommand{\slope}{\tfrac{p}{q}}
\newcommand{\slopeZero}{\tfrac{0}{1}}

\newcommand{\delV}{\delta_{\!\rotatebox[origin=c]{90}{$-$}\!}}
\newcommand{\delH}{\delta_{-}}

%% file: sections/0_pics.tex
\newcommand{\medsmallpic}[1]{\!\!\raisebox{-14.5pt}{#1}\!\!}%
\newcommand{\largepic}[1]{\!\!\raisebox{-48pt}{#1}\!\!}%

\newcommand{\MutationTangle}{\medsmallpic{
	\includegraphics[page=1]{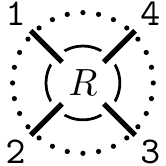}
}}
\newcommand{\MutationTangleX}{\medsmallpic{
	\includegraphics[page=2]{HFTSymmetries-pics-pics.pdf}
}}
\newcommand{\MutationTangleY}{\medsmallpic{
	\includegraphics[page=3]{HFTSymmetries-pics-pics.pdf}
}}
\newcommand{\MutationTangleZ}{\medsmallpic{
	\includegraphics[page=4]{HFTSymmetries-pics-pics.pdf}
}}
\newcommand{\pretzeltangle}{\largepic{
	\includegraphics[page=5]{HFTSymmetries-pics-pics.pdf}
}}
\newcommand{\DehnTwistArcs}{\largepic{
	\includegraphics[page=6]{HFTSymmetries-pics-pics.pdf}
}}

\newcommand{\MCGActionOnInterestingSegmentsA}{\largepic{
		\includegraphics[page=7]{HFTSymmetries-pics-pics.pdf}
	}}
\newcommand{\MCGActionOnInterestingSegmentsB}{\largepic{
		\includegraphics[page=8]{HFTSymmetries-pics-pics.pdf}
	}}
\newcommand{\MCGActionOnInterestingSegmentsC}{\largepic{
		\includegraphics[page=9]{HFTSymmetries-pics-pics.pdf}
	}}
\newcommand{\MCGActionOnInterestingSegmentsD}{\largepic{
		\includegraphics[page=10]{HFTSymmetries-pics-pics.pdf}
	}}
\newcommand{\MCGActionOnInterestingSegmentsE}{\largepic{
		\includegraphics[page=11]{HFTSymmetries-pics-pics.pdf}
	}}
\newcommand{\MCGActionOnInterestingSegmentsF}{\largepic{
		\includegraphics[page=12]{HFTSymmetries-pics-pics.pdf}
	}}

\newcommand{\MCGActionOnBoringSegmentsA}{\largepic{
		\includegraphics[page=13]{HFTSymmetries-pics-pics.pdf}
	}}
\newcommand{\MCGActionOnBoringSegmentsB}{\largepic{
		\includegraphics[page=14]{HFTSymmetries-pics-pics.pdf}
	}}
\newcommand{\MCGActionOnBoringSegmentsC}{\largepic{
		\includegraphics[page=15]{HFTSymmetries-pics-pics.pdf}
	}}
\newcommand{\MCGActionOnBoringSegmentsD}{\largepic{
		\includegraphics[page=16]{HFTSymmetries-pics-pics.pdf}
	}}
\newcommand{\MCGActionOnBoringSegmentsE}{\largepic{
		\includegraphics[page=17]{HFTSymmetries-pics-pics.pdf}
	}}
\newcommand{\MCGActionOnBoringSegmentsF}{\largepic{
		\includegraphics[page=18]{HFTSymmetries-pics-pics.pdf}
	}}
\newcommand{\AddingASingleCrossing}{\largepic{
		\includegraphics[page=19]{HFTSymmetries-pics-pics.pdf}
	}}

\newcommand{\GlueingTangleT}{\largepic{
		\includegraphics[page=20]{HFTSymmetries-pics-pics.pdf}
	}}
\newcommand{\GlueingTangleTP}{\largepic{
		\includegraphics[page=21]{HFTSymmetries-pics-pics.pdf}
	}}
\newcommand{\DehnTwistAD}{\largepic{
		\includegraphics[page=22]{HFTSymmetries-pics-pics.pdf}
	}}
\newcommand{\DehnTwistHDs}{\largepic{
		\includegraphics[page=23]{HFTSymmetries-pics-pics.pdf}
	}}

\newcommand{\wrappingi}{\largepic{
		\includegraphics[page=24]{HFTSymmetries-pics-pics.pdf}
	}}

\newcommand{\stabilizationfingerBefore}{\largepic{
		\includegraphics[page=28]{HFTSymmetries-pics-pics.pdf}
	}}
\newcommand{\stabilizationfinger}{\largepic{
		\includegraphics[page=29]{HFTSymmetries-pics-pics.pdf}
	}}
\newcommand{\pretzeltangleDUAL}{\largepic{
		\includegraphics[page=30]{HFTSymmetries-pics-pics.pdf}
	}}
\newcommand{\HalfIdHD}{\largepic{
		\includegraphics[page=31]{HFTSymmetries-pics-pics.pdf}
	}}
\newcommand{\HalfIdbefore}{\largepic{
		\includegraphics[page=32]{HFTSymmetries-pics-pics.pdf}
	}}
\newcommand{\HalfIdWhile}{\largepic{
		\includegraphics[page=33]{HFTSymmetries-pics-pics.pdf}
	}}
\newcommand{\HalfIdAfter}{\largepic{
		\includegraphics[page=34]{HFTSymmetries-pics-pics.pdf}
	}}
\newcommand{\twobasiccurves}{\largepic{
		\includegraphics[page=35]{HFTSymmetries-pics-pics.pdf}
	}}
\newcommand{\ConjugationHD}{\largepic{
		\includegraphics[page=36]{HFTSymmetries-pics-pics.pdf}
	}}

\newcommand{\tanglepairingI}{\largepic{
		\includegraphics[page=50]{HFTSymmetries-pics-pics.pdf}
	}}

\newcommand{\nutshell}{\largepic{
		\includegraphics[page=52]{HFTSymmetries-pics-pics.pdf}
	}}

\newcommand{\ptcurve}{\largepic{
		\includegraphics[page=54]{HFTSymmetries-pics-pics.pdf}
	}}
\newcommand{\pretzeltangleINTRO}{\largepic{
		\includegraphics[page=55]{HFTSymmetries-pics-pics.pdf}
	}}

\newcommand{\rigidcurve}{
	\includegraphics[page=73]{HFTSymmetries-pics-pics.pdf}
}
\newcommand{\LinearityOverview}{
	\includegraphics[page=74]{HFTSymmetries-pics-pics.pdf}
}
\newcommand{\LinearityCloseup}{
	\includegraphics[page=75]{HFTSymmetries-pics-pics.pdf}
}
\newcommand{\TwistPSLBottom}{
	\includegraphics[page=76]{HFTSymmetries-pics-pics.pdf}
}
\newcommand{\TwistPSLRight}{
	\includegraphics[page=77]{HFTSymmetries-pics-pics.pdf}
}
\newcommand{\rigidcurveIntroB}{
	\includegraphics[page=78]{HFTSymmetries-pics-pics.pdf}
}
\newcommand{\loosecurveIntro}{
	\includegraphics[page=79]{HFTSymmetries-pics-pics.pdf}
}
\newcommand{\rigidcurveIntroD}{
	\includegraphics[page=80]{HFTSymmetries-pics-pics.pdf}
}
\newcommand{\LiftOfParametrizationINTRO}{
	\includegraphics[page=81]{HFTSymmetries-pics-pics.pdf}
}
\newcommand{\mutationA}{
	\includegraphics[page=82]{HFTSymmetries-pics-pics.pdf}
}
\newcommand{\mutationB}{
	\includegraphics[page=83]{HFTSymmetries-pics-pics.pdf}
}

%% file: sections/Intro.tex
\section{Introduction}\label{sec:Intro}

\begin{figure}
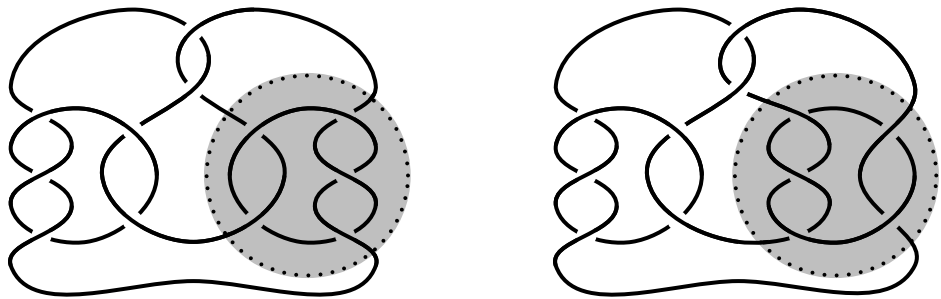

	\centering
	\centering
	$\mutationA$
	\qquad\qquad
	$\mutationB$
	\caption{The Kinoshita-Terasaka knot (left) and its Conway mutant (right). These two knot diagrams agree outside the grey disc bounded by the dotted circle; the tangle diagrams within this disc agree up to a rotation by $\pi$. }\label{fig:KTCcounterexample}
\end{figure}

Conway mutation is an operation which splits a given knot or link into two 4-ended tangles and glues them back together in a particular way. The most famous example of a mutant knot pair is the Kinoshita-Terasaka and the Conway knot, which is shown in Figure~\ref{fig:KTCcounterexample}. 
The effect of Conway mutation on link invariants has been studied for many decades. Interestingly, there is a large number of invariants that fail to distinguish mutant knots and links. Examples include the Seifert matrix up to S-equivalence~\cite{Cooper}, the HOMFLY polynomial~\cite{mutation_HOMFLY}, volumes of hyperbolic knots~\cite{mutation_volume}, Khovanov and Bar-Natan homology over \(\mathbb{Z}/2\)~\cite{Wehrli,Bloom,KWZ} and Rasmussen's $s$-invariant~\cite{KWZ}. 
Link Floer homology \(\HFL(L)\) is, in general, not invariant under mutation. This was first observed by Ozsváth and Szabó for the mutant pair from Figure~\ref{fig:KTCcounterexample}~\cite{OSmutation}; it follows more generally from the fact that \(\HFL(L)\) detects the Seifert genus~\cite{OSHFLThurston}, which is different for these two knots~\cite{Gabai_genus_mutation}.
However, based on computations from~\cite{BGcomps}, Baldwin and Levine conjectured that \(\HFL(L)\) is mutation invariant if the bigrading is collapsed to a single grading known as the $\delta$-grading~\cite[Conjecture~1.5]{BaldwinLevine}.  
This conjecture was confirmed for a number of families of mutant links by Lambert-Cole~\cite{LambertCole1,LambertCole2} and the author~\cite{PQMod}. 
In this paper, we prove the conjecture in general:

\begin{theorem}\label{thm:MutInvHFL}
	Let \(L\) be a link in a \(\mathbb{Z}\)-homology 3-sphere. Suppose \(L'\) is obtained from \(L\) by Conway mutation. Then \(\HFL(L)\) and \(\HFL(L')\) agree as relatively \(\delta\)-graded invariants. 
\end{theorem}

This result is an immediate consequence of certain symmetries of peculiar modules \(\CFTd(T)\), a tangle invariant defined in~\cite{PQMod}. Leading up to the proof of these symmetries, we establish a number of additional remarkable properties of \(\CFTd(T)\), which might be of independent interest. 

\begin{figure}[b]
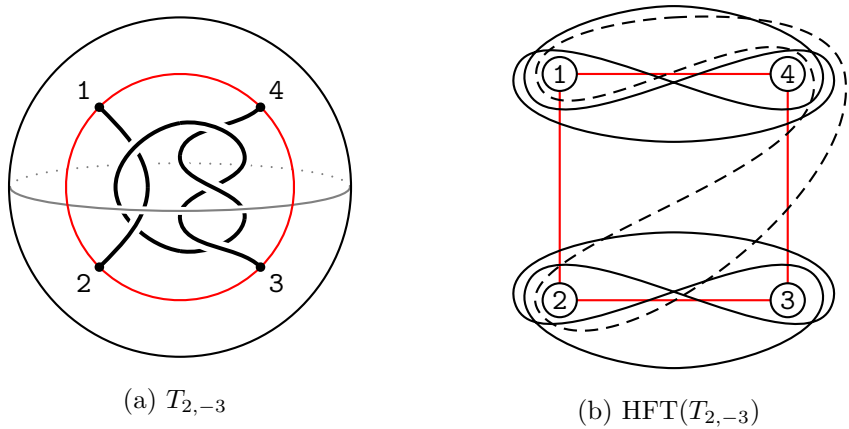

	\centering
	\begin{subfigure}{0.4\textwidth}
		\centering
		\pretzeltangleINTRO
		\caption{$T_{2,-3}$}\label{fig:intro:2m3pt:tangle}
	\end{subfigure}
	\begin{subfigure}{0.4\textwidth}
		\centering
		\ptcurve
		\caption{$\HFT(T_{2,-3})$}\label{fig:intro:2m3pt:curve}
	\end{subfigure}
	\caption{The $(2,-3)$-pretzel tangle (a) and its immersed curve invariant (b). The latter consists of three components, an embedded one (dashed) and two immersed components, each of which carries a (trivial) 1-dimensional local system.}\label{fig:intro:2m3pt}
\end{figure}

\subsection*{Immersed curve invariants of unparametrized tangles} 
The invariant
$$T\mapsto\CFTd(T)$$
associates with any oriented 4-ended tangle $T$ in a $\mathbb{Z}$-homology 3-ball $M$ the homotopy class of a relatively bigraded curved type D structure $\CFTd(T)$ over an algebra $\Ad$.  We call $\CFTd(T)$ the peculiar module of $T$ and $\Ad$ the peculiar algebra. The homotopy class of $\CFTd(T)$ depends on a parametrization of $\partial M=S^2$ in the form of an embedded oriented circle $\textcolor{red}{S^1}\hookrightarrow \partial M$ connecting the tangle ends, which are ordered compatibly. In~\cite{PQMod}, we chose to include this data---for convenience---in the definition of tangles, which is illustrated in Figure~\ref{fig:intro:2m3pt:tangle} for the $(2,-3)$-pretzel tangle $T_{2,-3}$ in the 3-ball. One of the central results from~\cite{PQMod} is that the category of peculiar modules is equivalent to the compact part of the Fukaya category of the 4-punctured sphere $\FourPuncturedSphere$. In particular, peculiar modules are classified by multicurves, ie homotopy types of immersed curves with local systems on $\FourPuncturedSphere$ (see Definition~\ref{def:loops} and Theorem~\ref{thm:classificationPecMod}). 
Therefore, we can interpret $\CFTd(T)$ as an invariant
$$T\mapsto \HFT(T)$$
which associates with $T$ such a multicurve.  Figure~\ref{fig:intro:2m3pt:curve} shows this invariant for the $(2,-3)$-pretzel tangle $T_{2,-3}$. Here, the 4-punctured sphere comes with a natural parametrization by four arcs connecting the punctures, which is used to record the relative bigrading of $\HFT(T)$. So one might wonder if it is legitimate to identify this 4-punctured sphere with $\partial M$ minus the four tangle ends. In other words, does twisting tangle ends correspond to twisting immersed curves? The first result that we prove in this paper answers this question in the affirmative:

\begin{theorem}[{\ref{thm:MCGaction}}]
	\(\HFT\) commutes with the action of the mapping class group of~\(\FourPuncturedSphere\).
\end{theorem}

The proof boils down to a computation of a bordered sutured type AD bimodule. 

\begin{figure}[b]
	\centering
	\LiftOfParametrizationINTRO
	\caption{The covering space $\eta\co\PuncturedPlane\longrightarrow\FourPuncturedSphere$}\label{fig:coveringmapINTRO}
\end{figure} 

\subsection*{The geography problem for components of peculiar modules of tangles}

Let us consider the map 
$\mathbb{R}^2\rightarrow \Torus\rightarrow S^2$
which is defined as the composition of the universal cover of a torus $\Torus$ with the double branched cover $\Torus\rightarrow\Sphere$, branched at four marked points of $\Sphere$. If we restrict this map to the preimage of the complement of the marked points, we obtain a covering space
$$
\eta\co\PuncturedPlane\rightarrow\FourPuncturedSphere.
$$
This is illustrated in Figure~\ref{fig:coveringmapINTRO}, where the standard parametrization of $\FourPuncturedSphere$ has been lifted to $\PuncturedPlane$ and the front face and its preimage under $\eta$ are shaded grey. 

It is useful to consider the lifts of the underlying curves of $\HFT(T)$ to $\PuncturedPlane$. Remarkably, there are only two relevant classes of curves that matter to $\HFT$, namely rational and irrational ones:

\begin{definition}\label{def:intro:curves}
	We call a closed curve in $\FourPuncturedSphere$ \textbf{primitive}, if it defines a primitive element of $\pi_1(\FourPuncturedSphere)$. 
	For a given slope $\slope\in\QPI$, let $\Rat(\slope)$ be an embedded, primitive curve in $\FourPuncturedSphere$ which under $\eta$ lifts to a straight line of slope $\slope\in\QPI$. We write  $\Rat_X(\slope)$ for this curve equipped with a local system $X\in\GL_n(\field)$ and call it the \textbf{rational curve of slope $\slope$ with local system $X$}. Usually, we will omit the 1-dimensional (trivial) local system $X$ from the notation. 
	
	The second family of curves is constructed as follows. Suppose $i_1$ and $i_2$ are two distinct tangle ends which lie on a straight line of slope $\slope\in\QPI$. The lattice points divide this line into intervals of equal length. For a fixed integer $n>0$, let us mark every $(2n)^\text{th}$ interval of the line. Then consider a small push-off of this line such that it intersects only the marked intervals and each of them exactly once. Finally, let $\Irr_n(\slope;i_1,i_2)$ be the immersed, primitive curve in $\FourPuncturedSphere$ which under $\eta$ lifts to this push-off. We call $\Irr_n(\slope;i_1,i_2)$ the \textbf{irrational curve of slope $\slope$ through the punctures $i_1$ and $i_2$}. Note that for each slope $\slope$ and fixed $n>0$, there are exactly two choices for the pair of punctures $(i_1,i_2)$. 
\end{definition}

\begin{example}
	For the slope $\slope=\frac{0}{1}$, the lifts of $\Rat(\slope)$ and $\Irr_n(\slope;i_1,i_2)$ for $(i_1,i_2)=(1,4)$ and  $(i_1,i_2)=(2,3)$ are shown in Figure~\ref{fig:intro:curves}. 
	Rational curves are, as the name suggests, the underlying curves of the immersed curve invariants of rational tangles. But they also arise as the underlying curves of components of $\HFT$ of non-rational tangles. For example, the dashed component of Figure~\ref{fig:intro:2m3pt:curve} is equal to $\Rat(\frac{1}{2})$. The underlying curves of the other two components of $\HFT(T_{2,-3})$ are equal to $\Irr_1(\frac{0}{1};4,1)$ and $\Irr_1(\frac{0}{1};2,3)$, respectively. Irrational components for $n>1$ show up in the invariants of 2-stranded pretzel tangles with more twists, see~\cite[Theorem~6.9]{PQMod}. 
\end{example}

\begin{figure}[t]
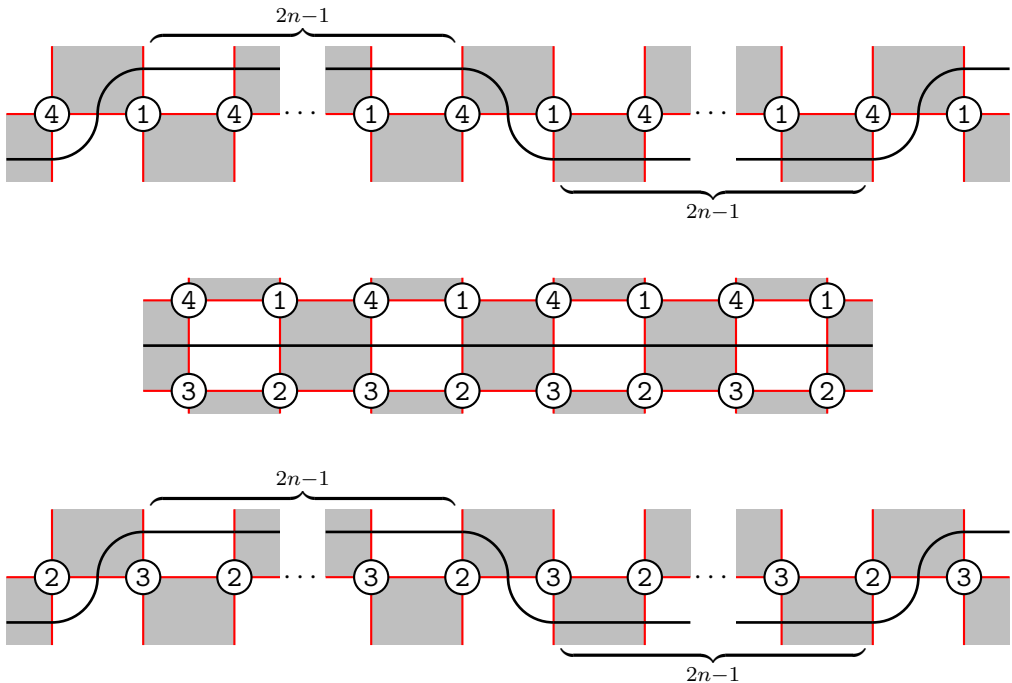

	\centering
		\centering
		\rigidcurveIntroD
		\\
		\loosecurveIntro
		\\
		\rigidcurveIntroB
	\caption{The lifts of the curves $\Irr_n(\frac{0}{1};4,1)$ (top), $\Rat(\frac{0}{1})$ (middle) and $\Irr_n(\frac{0}{1};2,3)$ (bottom), illustrating Definition~\ref{def:intro:curves}.}\label{fig:intro:curves}
\end{figure}

\begin{theorem}[\ref{thm:linearCurves}, \ref{thm:rigid_curves}]
	For a 4-ended tangle \(T\) in a $\mathbb{Z}$-homology 3-ball, the underlying curve of each component of \(\HFT(T)\) is either rational or irrational. Moreover, if it is irrational, its local system is equal to an identity matrix. 
\end{theorem}


This result completely answers the geography problem for components of peculiar modules, except for local systems on rational components; see Question~\ref{que:local_systems_for_rationals}. 
The proof of this theorem relies on the existence of an extension of peculiar modules of tangles to curved type D structures $\CFTminus$ over a larger algebra together with the fact that this extendibility property is preserved under basic chain homotopies (see Subsection~\ref{subsec:background:GeneralizedPQM}). 

\subsection*{Stabilization}
In~\cite[Proposition 9.2]{SFH}, Juhász gave an interpretation of link Floer homology in terms of sutured Floer homology of the link complement equipped with a pair of oppositely oriented meridional sutures on each link component. Moreover, as a simple application of his surface decomposition formula~\cite[Proposition 8.6]{SurfaceDecomposition}, one can also show that adding additional such pairs of sutures corresponds to tensoring the sutured Floer homology by a two-dimensional vector space, an operation which is known as stabilization. 
$\HFT(T)$ is essentially defined in terms of a multi-pointed Heegaard Floer theory and, as such, it is closely related to sutured Floer homology. Therefore, it is natural to generalize $\HFT$ to stabilized tangles $T(n_1,\dots,n_{|T|})$, that is tangles $T$ with $n_i$ (additional) pairs of meridional sutures on the $i^\text{th}$ component of $T$ for $i=1,\dots,|T|$, where $|T|$ is the number of components of $T$. As one might expect, the relationship between the immersed curve invariants of stabilized and ordinary tangles is as follows:

\begin{theorem}[\ref{thm:stabilization}]\label{thm:intro:stabilization}
	$\HFT$ commutes with stabilization. More explicitly, for any stabilized tangle $T(n_1,\dots,n_{|T|})$, 
	$$\HFT(T(n_1,\dots,n_{|T|}))=\bigotimes_{k=1}^{|T|} V_{t_k}^{\otimes n_k}\otimes\HFT(T),$$
	where \(V_t\) is a 2-dimensional vector space supported in a single relative \(\delta\)-grading and two consecutive relative Alexander gradings \(t^{+1}\) and \(t^{-1}\). 
\end{theorem}

\subsection*{Relationship between peculiar modules and bordered sutured Heegaard Floer theory}
In~\cite{PQMod}, the author proved a Glueing Theorem which says that if we decompose a link $L$ into two 4-ended tangles $T_1$ and $T_2$, we can compute the link Floer homology $\HFL(L)$ in terms of the Lagrangian intersection Floer homology of (the mirror of) $\HFT(T_1)$ and $\HFT(T_2)$.
The proof of this result relies on Zarev's bordered sutured Heegaard Floer theory \cite{ZarevThesis} and uses a particular bordered sutured structure on the tangle complement $M_T$. The bordered sutured type D structure $\BSD(M_T)$ of $M_T$ is defined over an algebra which is much larger and more complicated than the peculiar algebra $\Ad$. But for the proof of the Glueing Theorem, it is sufficient to identify the image of the peculiar module $\CFTd(T)$ of a tangle $T$ under a certain quotient functor with the image of $\BSD(M_T)$ under another quotient functor. 

In this article, we offer an alternative definition of peculiar modules in terms of a different bordered sutured structure on the tangle complement, which was first suggested to the author by Andy Manion. This bordered sutured structure is in some sense dual to the one chosen previously and has the advantage that the corresponding bordered sutured algebra can be interpreted as a very simple quotient of the peculiar algebra $\Ad$. Moreover, we show in Theorem~\ref{thm:HalfIdBimodAction} that the image of $\CFTd(T)$ under the induced quotient functor agrees with the bordered sutured type D structure of the tangle complement with the new bordered sutured structure. The proof of this identification relies on a certain bimodule, which is sometimes called the half-identity bimodule of a bordered sutured structure. At the risk of being rather imprecise, let us summarize the proof of Theorem~\ref{thm:HalfIdBimodAction} as follows:

\begin{theorem}[{\ref{thm:HalfIdBimodAction}}]
	The half-identity bimodule of the 4-punctured sphere acts like the identity.
\end{theorem}

\subsection*{Conjugation symmetry}

Like link Floer homology, peculiar modules come with a relative bigrading. More precisely, if $T$ is an oriented 4-ended tangle in a $\mathbb{Z}$-homology 3-ball, the generators $\CFTd(T)$ carry a homological grading $h$ which is a relative $\mathbb{Z}$-grading as well as a multivariate Alexander grading $A$, which is a relative $\mathbb{Z}^{|T|}$-grading, where $|T|$ is the number of components of~$T$. If we compose $A$ with the map $\mathbb{Z}^{|T|}\rightarrow\mathbb{Z}$ which adds all components together, we obtain the univariate Alexander grading $\overline{A}$. There is also a third grading, the $\delta$-grading, which is a relative $\frac{1}{2}\mathbb{Z}$-grading defined in terms of the other two gradings by
$$\delta=\tfrac{1}{2}\overline{A}-h.$$
In this paper, we will use the $\delta$-grading instead of the homological grading, since the former is, unlike the latter, independent of the orientation of $T$. 
We write $\HFTdelta(T)$ for the image of $\HFT(T)$ under the forgetful functor which drops the homological and Alexander grading. 

\begin{definition}
	Let us fix, once and for all, an absolute bigrading of $\Rat_X(\slope)$ and $\Irr_n(\slope;i,j)$ for every slope $\slope\in\QPI$ and any choice of $X$, $n$, $i$ and $j$. (In Section~\ref{sec:ConjugationBimodule}, we will impose some conditions on how we do this, see Theorem~\ref{thm:Conjugation:Horizontals} and Definition~\ref{def:choice_of_delta_grading}.) Then, if $\gamma$ is one of these absolutely bigraded curves with local systems, let us denote by $\delta^mt^A\gamma$ the bigraded immersed curve with local system obtained from $\gamma$ by shifting the $\delta$-grading by $m$ and the Alexander grading by $A$.
	Given an absolutely bigraded multicurve $L$ and  $\gamma\in\{\Rat_X(\slope),\Irr_n(\slope;i,j)\}$, let us define
	$$
	\gamma(L)=\sum n_{m,A}\delta^{m} t^{A},
	$$
	where $n_{m,A}$ is the number of curves $\delta^{m} t^{A}\gamma$ in $L$ over all $m\in\frac{1}{2}\mathbb{Z}$ and Alexander gradings~$A$. We call $\gamma(L)$
	the \textbf{Poincaré polynomial of $L$ at $\gamma$}.
	Let us write $\gamma^\delta(L)\in\mathbb{Z}[\delta^{\pm\frac{1}{2}}]$ for the polynomial $\gamma(L)$ evaluated at $t=1$.  
	Given an oriented 4-ended tangle $T$, let us write $\gamma(T)\coloneqq \gamma(\HFT(T))$
	 and similarly $\gamma^\delta(T)\coloneqq \gamma^\delta(\HFT(T))$. 
	Finally, let $\rr$ be the operation which inverts all Alexander gradings.
\end{definition}

\begin{theorem}[Bigraded conjugation symmetry for horizontal components; \ref{thm:Conjugation:Horizontals}]\label{thm:Conjugation:Horizontals:INTRO}
	Given an oriented 4-ended tangle \(T\), let us fix an absolute lift of the \(\delta\)-grading of \(\HFT(T)\). Suppose, \(\HFT(T)\) contains a component of slope \(\tfrac{0}{1}\). Then, there exists a unique lift of the Alexander grading on \(\HFT(T)\) such that for all local systems \(X\) and positive integers \(n\),
	$$\rr\Big(\Rat_X(\slopeZero)(T)\Big)= \Rat_{X^{-1}}(\slopeZero)(T)$$
	and:
	$$\rr\Big(\Irr_n(\slopeZero;2,3)(T)\Big)=\Irr_n(\slopeZero;4,1)(T)$$
	if the tangle ends \(\texttt{2}\) and \(\texttt{3}\) belong to different open tangle components and
	$$\rr\Big(\Irr_n(\slopeZero;2,3)(T)\Big)\cdot(t_{2,3}+t_{2,3}^{-1})= \Irr_n(\slopeZero;4,1)(T)\cdot(t_{4,1}+t_{4,1}^{-1})$$
	otherwise, where \(t_{2,3}\) and \(t_{4,1}\) are the colours of the tangle components at the tangle ends (\(\texttt{2}\) and \(\texttt{3}\)) and (\(\texttt{4}\) and \(\texttt{1}\)), respectively. 
\end{theorem}

The author expects a similar statement to hold for all slopes. A proof of this would---amongst other things---require a better understanding of the Alexander grading on rational and irrational curves of \emph{arbitrary} slope, which is beyond the scope of this paper. However, if we content ourselves with the $\delta$-grading, we can easily deduce the following result:

\begin{theorem}[$\delta$-graded conjugation symmetry]\label{thm:Conjugation:delta}
	Given an oriented 4-ended tangle \(T\), let us fix an absolute lift of the \(\delta\)-grading of \(\HFTdelta(T)\). Then for any choice of \(\slope\), \(X\), \(n\) and \((i_1,i_2,i_3,i_4)\),
	$$
	\Rat^\delta_X(\slope)(T)= \Rat^\delta_{X^{-1}}(\slope)(T)
	\quad\text{ and }\quad
	\Irr^\delta_n(\slope;i_1,i_2)(T)=\Irr^\delta_n(\slope;i_3,i_4)(T).
	$$ 
\end{theorem}

As their names suggest, these two results should be regarded as the analogues of conjugation symmetry of Heegaard Floer theories of closed 3-manifolds or knots and links therein. The classical proof of conjugation symmetry goes along the following lines: first, we fix a Heegaard diagram from which we compute the invariant in question. Then we construct a new Heegaard diagram from the first by reversing the orientation of its underlying surface and simultaneously exchanging the role of $\alpha$- and $\beta$-curves. The algebraic structures computed from the two diagrams are the same, up to reversing the $\Spinc$-grading, which corresponds to the Alexander grading. Moreover, for closed 3-manifolds and knots or links therein, one can argue that the new diagram also represents the same geometric object, so the two algebraic structures need to be the same up to chain homotopy.

For bordered and bordered sutured 3-manifolds, this argument is more complicated, since Heegaard diagrams also involve $\alpha$- and $\beta$-arcs that parametrize the 3-manifold boundaries. So while conjugation leaves the underlying 3-manifolds unchanged, $\alpha$-arcs become $\beta$-arcs and vice versa. One might hope to fix this problem by glueing in a bordered (sutured) manifold such that the diffeomorphism type of the underlying manifold does not change, but $\alpha$- and $\beta$-arcs are interchanged once more. For example, Hanselman, J.\,Rasmussen and Watson use this strategy to show that their immersed curve invariant for 3-manifolds with torus boundary is preserved by the elliptic involution up to reversing the $\Spinc$-grading~\cite{HRW2}. In their case, the bordered Heegaard diagrams have two $\alpha$-arcs and no $\beta$-arcs, so one of the key ingredients is to find a bordered Heegaard diagram of the once-punctured torus which turns a parametrization by $\beta$-arcs into one by $\alpha$-arcs.

In the case of tangles, the situation is even more complicated than in the torus boundary case: here, there is no bordered sutured structure which only interchanges the parametrization by $\alpha$- and $\beta$-arcs. Any such bordered sutured manifold also needs to introduce at least two extra pairs of sutures on the open tangle components. (This can be easily seen by studying how the underlying sutured structure has to change.) In Section~\ref{sec:ConjugationBimodule}, we describe a bordered sutured manifold that interchanges $\alpha$- and $\beta$-arcs at the expense of two additional pairs of meridional sutures. As we have seen in Theorem~\ref{thm:intro:stabilization}, adding meridional sutures is conceptually not a big problem. However, it makes computations much more involved, as the corresponding type AD bimodule with its 32 generators from Figure~\ref{fig:ConjugationBimodule} illustrates. 

\subsection*{Conway mutation}


\begin{figure}
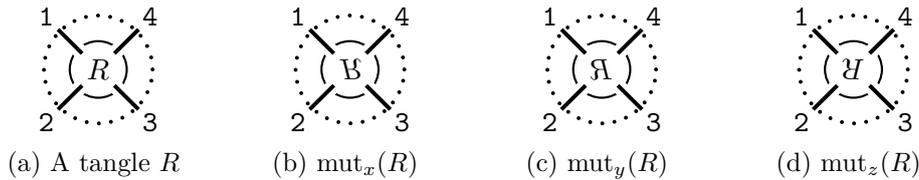

	\centering
	\begin{subfigure}{0.2\textwidth}
		\centering
		\MutationTangle
		\caption{A tangle $R$}
	\end{subfigure}
	\begin{subfigure}{0.2\textwidth}
		\centering
		\MutationTangleX
		\caption{$\mut_x(R)$}
	\end{subfigure}
	\begin{subfigure}{0.2\textwidth}
		\centering
		\MutationTangleY
		\caption{$\mut_y(R)$}
	\end{subfigure}
	\begin{subfigure}{0.2\textwidth}
		\centering
		\MutationTangleZ
		\caption{$\mut_z(R)$}
	\end{subfigure}
	\caption{Conway mutation. The relabelling of the tangle ends is illustrated here in terms of rotations of the tangle.}\label{fig:mutation}
\end{figure}

Given a tangle $R$ in a 3-manifold $M$ with spherical boundary, let $\mut_x(R)$, $\mut_y(R)$ and $\mut_z(R)$ be the tangles obtained from $R$ by relabelling the tangle ends according to the permutations $(12)(34)$, $(14)(23)$ and $(13)(24)$, respectively. 
This is illustrated in Figure~\ref{fig:mutation}. 
We say that  $\mut_x(R)$, $\mut_y(R)$ and $\mut_z(R)$ are obtained from $R$ by mutation. 
If $R$ is oriented, we orient these three tangles such that the orientations agree at the tangle ends with the same labels. 
If this means that we need to reverse the orientation of the two open components of~$R$, then we also reverse the orientation of all other components; otherwise we do not change any orientation. 
We say two links are \textbf{Conway mutants} if they agree outside a 3-manifold with spherical boundary which intersects the links transversely in four points and the tangles in its complement are related by mutation. 

As a consequence of $\delta$-graded conjugation symmetry of $\HFT$, we can now show:

\begin{theorem}\label{thm:MutInvHFT}
	Let \(T\) be a 4-ended tangle in a \(\mathbb{Z}\)-homology 3-ball and \(T'\) obtained from \(T\) by mutation. Then \(\HFTdelta(T)\) and \(\HFTdelta(T')\) agree as relatively \(\delta\)-graded invariants.
\end{theorem}

Together with the aforementioned Glueing Theorem for $\HFT$, this proves Theorem~\ref{thm:MutInvHFL}.

\subsection*{Acknowledgements}
I would like to thank Jonathan Hanselman, Jake Rasmussen and Liam Watson for many helpful discussions about their immersed curve invariant for 3-manifolds with torus boundary and their encouragement to keep thinking about my tangle invariant. In particular, I am indebted to Liam for suggesting to think about the covering space from Section~\ref{sec:Linearity}, which lead to a breakthrough in this project. I would also like to thank Andy Manion for sharing his views on bordered sutured algebras and suggesting the bordered sutured structure from Section~\ref{sec:HalfId} to me. Finally, I would like to thank Artem Kotelskiy, Jake Rasmussen and Liam Watson for comments on a first draft of this paper.

%% file: sections/Background.tex
\section{Review: peculiar modules}\label{sec:background}

In this section, we review some background on peculiar modules from~\cite{PQMod}.

\begin{wrapfigure}{r}{0.3333\textwidth}
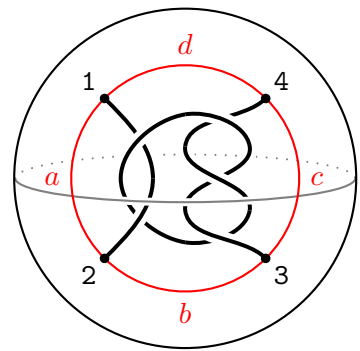

	\centering
	\bigskip
	\pretzeltangle
	\caption{The $(2,-3)$-pretzel tangle in $B^3$}\label{fig:2m3pt}
	\bigskip
\end{wrapfigure}
\subsection{Definition of peculiar modules}

\begin{definition}\label{def:tangle}
	A \textbf{4-ended tangle} $T$ in a $\mathbb{Z}$-homology 3-ball~$M$ with spherical boundary is an embedding
	\[T\co\left(I \sqcup I \sqcup \coprod S^1,\partial\right)\hookrightarrow \left(M,\textcolor{red}{S^1}\subset S^2=\partial M\right),\hspace{0.3333\textwidth}\]
	such that the endpoints of the two intervals lie on a fixed oriented circle $\textcolor{red}{S^1}$ on the boundary of~$M$, together with a choice of distinguished tangle end. Starting at this distinguished ($=$first) tangle end and following the orientation of the fixed circle $\textcolor{red}{S^1}$, we number the tangle ends and label the arcs $\textcolor{red}{S^1}\smallsetminus \im(T)$ by $a$, $b$, $c$ and $d$, in that order. Sometimes, we will find it more convenient to use the labels $s_1$ for $a$, $s_2$ for $b$, $s_3$ for $c$ and $s_4$ for $d$ instead. We call a choice of a single arc a \textbf{site} of the tangle~$T$. We call $\partial M$ the \textbf{boundary of the tangle $T$} and those four arcs a \textbf{parametrization} thereof. 

	We consider tangles up to ambient isotopy which fixes the distinguished tangle end and the orientation of $\textcolor{red}{S^1}$ (and thus preserves the labelling of the tangle ends and arcs). The images of the two intervals are called the \textbf{open components}, the images of any circles are called the \textbf{closed components} of the tangle. We label these tangle components by variables $t_1$ and $t_2$ for the open components and $t'_1,t'_2,\dots$ for the closed components. We call those variables the \textbf{colours} of~$T$.
	An orientation of a tangle is a choice of orientation of the two intervals and the circles. 
	\\\indent
	Note that the orientation of $\textcolor{red}{S^1}$ enables us to distinguish between the two components of $\partial M\smallsetminus\textcolor{red}{S^1}$. The \textbf{back} component of $\partial M\smallsetminus\textcolor{red}{S^1}$ is the one whose boundary orientation agrees with the orientation of $\textcolor{red}{S^1}$, using the right-hand rule and a normal vector field pointing into $M$. We call the other one the \textbf{front} component. 
\end{definition}

\begin{wrapfigure}{r}{0.2\textwidth}
	\centering
	{	$
		\begin{tikzcd}[row sep=0.5cm, column sep=0.5cm]
		&
		\overset{\iota_4}{\bullet}
		\arrow[bend left=10,leftarrow]{dr}[inner sep=1pt]{q_4}
		\arrow[bend right=10,swap]{dr}[inner sep=1pt]{p_4}
		&
		\\
		\!\!\!\raisebox{7pt}{$\underset{\iota_1}{~}$}\,\bullet
		\arrow[bend left=10,leftarrow]{ur}[inner sep=1pt]{q_1}
		\arrow[bend right=10,swap]{ur}[inner sep=1pt]{p_1}
		&
		&
		\bullet\,\raisebox{7pt}{$\underset{\iota_3}{~}$}\!\!\!
		\arrow[bend left=10,leftarrow]{dl}[inner sep=1pt]{q_3}
		\arrow[bend right=10,swap]{dl}[inner sep=1pt]{p_3}
		\\
		&
		\underset{\iota_2}{\bullet}
		\arrow[bend left=10,leftarrow]{ul}[inner sep=1pt]{q_2}
		\arrow[bend right=10,swap]{ul}[inner sep=1pt]{p_2}
		\end{tikzcd}
		\vspace{-5pt}
		$
	}
	\caption{}\label{fig:quiverForAd}
\end{wrapfigure}
\myfixwrapfig

\begin{definition}\label{def:matching}
	A \textbf{matching} $P$ is a partition $\{\{i_1,o_1\},\{i_2,o_2\}\}$ of $\{1,2,3,4\}$ into pairs. An \textbf{ordered matching} is a matching in which the pairs are ordered. A 4-ended tangle $T$ gives rise to a matching $P_T$ as follows: the first pair consists of the two endpoints of the open component with colour $t_1$, the second consists of the two endpoints of the second open component of $T$, the one labelled $t_2$. Given an orientation of the two open components of $T$, we order each pair of points such that the inward pointing end comes first, the outward pointing end second.
\end{definition}

\begin{definition}\label{def:RecallCFT}
	Let $\Ad$ be the path algebra of the quiver in Figure~\ref{fig:quiverForAd} 
	modulo the relations $$p_iq_i=0=q_ip_i.$$ We call $\Ad$ the \textbf{peculiar algebra}. We write $$\Id\coloneqq \field\langle\iota_1,\iota_2,\iota_3,\iota_4\rangle\hspace*{0.2\textwidth}$$
	for the corresponding ring of idempotents.
	For convenience, we sometimes write the elements $p_ip_{i+1}\dots p_{j-1}p_{j}$ as $p_{i(i+1)\cdots (j-1)j}$ and $q_iq_{i-1}\dots q_{j+1}q_{j}$ as $q_{i(i-1)\cdots (j+1)j}$, where again, we take the indices modulo 4 with an offset of 1. Furthermore, to simplify notation, we set 
	$$
	p=p_1+p_2+p_3+p_4\in\Ad
	\quad\text{ and }\quad
	q=q_1+q_2+q_3+q_4\in\Ad,
	$$
	so we can write for example $p^4=p_{1234}+p_{2341}+p_{3412}+p_{4123}$.
	
	We define a $\tfrac{1}{2}\mathbb{Z}$-grading on $\Ad$, called the \textbf{$\delta$-grading}, by setting $$\delta(p_i)=\delta(q_i)\coloneqq\tfrac{1}{2},$$
	where $i=1,2,3,4$, and then extending linearly to all of $\Ad$. Similarly, given an ordered matching $P=\{\{i_1,o_1\},\{i_2,o_2\}\}$, we define relative $\mathbb{Z}$-gradings $A_{t_k}$ for $k=1,2$, called \textbf{Alexander gradings}, by
	$$A_{t_k}(\iota_s)\coloneqq 0,\quad A_{t_k}(p_{i_k})=A_{t_k}(q_{i_k})\coloneqq 1\quad\text{ and }\quad A_{t_k}(p_{o_k})=A_{t_k}(q_{o_k})\coloneqq -1,$$
	and then extend linearly to $\Ad$. We call $(\delta,A)$ the \textbf{bigrading} on $\Ad$. 
	We denote the sum of the two Alexander gradings by $\overline{A}$ and call it the \textbf{reduced Alexander grading}.
	There is also a third grading, the \textbf{homological grading}, which is defined in terms of the other two by 
	$$h=\tfrac{1}{2}\overline{A}-\delta.$$	
\end{definition}

\begin{definition}\label{def:pqMod}
	Given an (ordered) matching $P=\{\{i_1,o_1\},\{i_2,o_2\}\}$, let $\pqMod\coloneqq \pqMod_P$ be the category of $\delta$-graded (and Alexander graded) curved type D structures over \(\Ad\) with curvature 
	\[p^4+q^4.\]
	More explicitly, an object of this category is a pair \((C,\partial)\), where \(C\) is a relatively $\delta$-graded (and Alexander graded) right $\Id$-module and 
	$\partial\co C\rightarrow C\otimes_{\Id} \Ad$ is a right \(\Id\)-module homomorphism which increases the $\delta$-grading by 1, preserves the Alexander grading (if defined), and satisfies
	\[(1_C\otimes \mu)\circ(\partial\otimes 1_{\Ad})\circ\partial=1_C\otimes (p^4+q^4),\]
	where $\mu$ denotes composition in $\Ad$.
	As in~\cite{PQMod}, we call objects of this category \textbf{peculiar modules} and usually consider the underlying $\Id$-modules with a preferred choice of basis.
	
	A morphism between two peculiar modules $(C,\partial)$ and $(C',\partial')$ is an $\Id$-module homomorphism $C\rightarrow C'\otimes_{\Id} \Ad$. The composition of two morphisms $f$ and $g$ is defined as 
	\[(g\circ f)=(1\otimes\mu)\circ(g\otimes 1_{\Ad})\circ f.\]
	We endow the space of morphisms $\Mor((C,\partial),(C',\partial'))$ with a differential~$D$ defined by
	\[D(f)= \partial'\circ f+f\circ\partial.\]
	This turns $\pqMod$ into an enriched category over the category of ordinary chain complexes over $\field$. The underlying ordinary category is obtained by restricting the morphism spaces to degree 0 elements in the kernel of $D$, giving us the usual notions of chain map and chain homotopy. For more details, see~\cite[Section~1.1]{PQMod}.	
\end{definition}

\begin{definition}\label{def:tanglepairing}
	Given two 4-ended tangles $T_1$ and $T_2$, let $L(T_1,T_2)$ be the link obtained by glueing $T_1$ to $T_2$ as shown in Figure~\ref{fig:glueing2tangles1}. 
\end{definition}

\begin{wrapfigure}{r}{0.3333\textwidth}
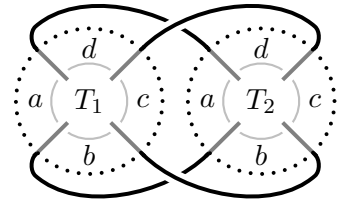

	\centering
	\tanglepairingI
	\caption{The link $L(T_1,T_2)$}\label{fig:glueing2tangles1}\bigskip
\end{wrapfigure}
\myfixwrapfig

\begin{theorem}[{\cite[Theorem~0.1 and Corollary~0.7]{PQMod}}]\label{thm:PecMod}
	For any oriented 4-ended tangle \(T\) in a \(\mathbb{Z}\)-homology 3-ball, one can define a peculiar module \(\CFTd(T)\in\ob(\pqMod_{P_T})\), where \(P_T\) is the ordered matching induced by \(T\).  Furthermore, its bigraded chain homotopy type is an invariant of the tangle. Moreover, let \(L=L(T_1,T_2)\) be the result of glueing two oriented 4-ended tangles \(T_1\) and \(T_2\) together such that their orientations match. Also, let \(V_t\) be as in Theorem~\ref{thm:intro:stabilization} and let \(\mr(T_1)\) be the mirror image of \(T_1\) with the orientation of all components reversed. Then
	\[
	\HFL(L)\otimes V_t
	\cong
	\Homology\left(\Mor\left(\CFTd(\mr(T_1)),\CFTd(T_2)\right)\right)
	\]
	if the four open components of the tangles become identified to the same component of colour $t$ and 
	\[
	\HFL(L)
	\cong
	\Homology\left(\Mor\left(\CFTd(\mr(T_1)),\CFTd(T_2)\right)\right)
	\]
	otherwise. 
\end{theorem}

$\CFTd$ is defined as a multi-pointed Heegaard Floer homology using a suitable notion of Heegaard diagrams for 4-ended tangles in which the parametrization $\textcolor{red}{S^1}$ of the boundary of the tangle plays the role of an $\alpha$-circle. All the results of this paper can be understood without any knowledge of how this construction works. Therefore, we will only give more background on this when it seems necessary and point the interested reader to the relevant sections in~\cite{PQMod} for more details. 

\begin{definition}\label{def:generalizedAlexgrading}
	We sometimes want to use an Alexander grading without specifying a particular ordered matching ie orientation of the tangle. So we define a grading on $\Ad$ and $\CFTd(T)$ which takes values in 
	$$
	\AlexGr\coloneqq\mathbb{Z}^4/\im(\mathbb{Z}\rightarrow\mathbb{Z}^4,n\mapsto(n,n,n,n))
	$$
	such that the grading of $p_i,q_i\in\Ad$ is 1 in the $i^\text{th}$ component and 0 in the other three components. 
	We usually write an element $(a,b,c,d)$ in this group as $\Alex{a}{b}{c}{d}$. Note that this grading extends to $\pqMod$, since the gradings of $p^4$ and $q^4$ are both $\Alex{1}{1}{1}{1}=\Alex{0}{0}{0}{0}$. We call $\AlexGr$ the \textbf{generalized Alexander grading}.
\end{definition}

\begin{Remark}
	Let us also note that each closed component $t'$ of a tangle $T$ gives rise to an additional $\mathbb{Z}$-grading $A_{t'}$ which is preserved along differentials of $\CFTd(T)$ and vanishes on $\Ad$. As already noted in~\cite[Observation~2.31]{PQMod}, this implies that $\CFTd(T)$ decomposes into the direct sum over the Alexander
	gradings of closed tangle components. This is why we will usually suppress these gradings from our notation. 
\end{Remark}

\begin{wrapfigure}{r}{0.36\textwidth}
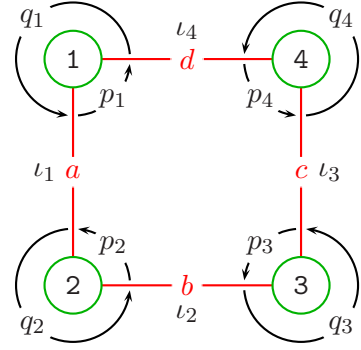

	\centering
	\nutshell
	\caption{A geometric interpretation of the quiver from Figure~\ref{fig:quiverForAd} in terms of the parametrization of the 4-punctured sphere}\label{fig:nutshell}
\end{wrapfigure}

\subsection{Classification results for peculiar modules}

In \cite{PQMod}, the author classified peculiar modules in terms of immersed curves on the 4-punctured sphere, using an algorithm from~\cite{HRW} due to Hanselman, J.\,Rasmussen and Watson. 

\begin{definition}\label{def:loops}
	A \textbf{loop} or \textbf{immersed curve with local system} on the 4-punctured sphere $\FourPuncturedSphere$ is a pair $(\gamma, X)$ where $\gamma$ is an immersion of an oriented circle into $\FourPuncturedSphere$ representing a non-trivial primitive element of $\pi_1(\FourPuncturedSphere)$ and $X\in \GL_n(\field)$ for some integer $n$. We consider $\gamma$ up to isotopy, $X$ up to matrix similarity and $(\gamma,X)$ up to simultaneous orientation reversal of $\gamma$ and inversion of the matrix $X$. 
	Given an immersed curve $(\gamma, X)$, we call $\gamma$ the \textbf{underlying curve} of the loop and $X$ its \textbf{local system}. We sometimes write $\gamma_X$ for $(\gamma,X)$.
	
	In~\cite[Definition~4.28]{PQMod}, the author defined a \textbf{collection of immersed curves with local systems} as a set of loops \(\{(\gamma_i,X_i)\}_{i\in I}\) such that the immersed curves $\gamma_i$ are pairwise non-homotopic as unoriented curves.
	Then, two collections of loops \(\{(\gamma_i,X_i)\}_{i\in I}\) and \(\{(\gamma'_j,X'_j)\}_{j\in J}\) were considered equivalent iff there was a bijection \(\iota\co I\rightarrow J\) such that \(\gamma_i\)~was homotopic to~\(\gamma'_{\iota(i)}\) and \(A_i\)~was similar to~\(A'_{\iota(i)}\) for all $i\in I$.
	
	In this paper, we will use a slightly different way to bundle loops together: a \textbf{multicurve} is an unordered set of loops each of whose local system is the companion matrix of a polynomial in $\field[x]$, subject to the condition that each set of polynomials whose corresponding curves are homotopic to each other can be ordered such that one polynomial divides the next. Given a collection $L$ of immersed curves with local systems, we can obtain a multicurve as follows: first, we put each local system $X$ of a loop $\gamma_X$ in $L$ into Frobenius normal form as in \cite[Section~4.7]{PQMod}. Then, we split $\gamma_X$ into the direct summands of curves $\gamma_{X_i}$, where $X_i$ are the diagonal blocks of the Frobenius normal form of $X$.  Since the set of diagonal block matrices in the Frobenius normal form is unique, two collections of loops are considered equivalent iff their corresponding multicurves are actually the same. Conversely, we can go from a multicurve to a collection of immersed curves by taking the direct sum of all local systems whose underlying curves are homotopic as unoriented curves. 
\end{definition}

\begin{definition}[{cp \cite[Definition~4.30]{PQMod}}]\label{def:loopsTOpqMod}
	Given a loop $\gamma_X$ and $s\in\{a,b,c,d\}$, let $C.\iota_{s}$ be the right $\Id$-module over $\field$ generated by the intersection points of $\gamma$ with the arc $s$, counted with multiplicity $\dim X$. Let us assume that each of the elementary curve segments of $\gamma$, ie the immersed intervals obtained by cutting $\gamma$ along the parametrizing arcs of the 4-punctured sphere, connects distinct arcs. Then $C.\iota_{s}$ is an invariant of the homotopy type of $\gamma$. We define the $\Id$-module $C$ as the direct sum of $C.\iota_{s}$ over $s=a,b,c,d$. 
	
	Next, we want to define a differential on $C$. For this, let us first consider an elementary curve segment between two intersection points, say $x$ and $y$. We can orient it in two ways, namely $x\rightarrow y$ and $y\rightarrow x$. The oriented elementary curve segment $x\rightarrow y$ is homotopic to an oriented path on the union of the parametrizing arcs and the boundary paths $p_i$ and $q_i$, $i=1,2,3,4$, in Figure~\ref{fig:nutshell}. Moreover, if we want it to avoid at least one boundary component of the 4-punctured sphere and if we also want its orientation to agree with the orientations of the boundary paths $p_i$ and $q_i$, then this path is unique up to reparametrization. Let $\{a_1,\dots, a_j\}$ be the ordered list of boundary paths $p_i$ and $q_i$ that we meet along this path. (Note that by construction, $1\leq j\leq 3$.) By interpreting the boundary paths $p_i$ and $q_i$ as the generators of $\Ad$, we can define $a(x\rightarrow y)$ as the product $a_j\cdots a_1$. Similarly, we define $a(y\rightarrow x)$. (Note that $a(y\rightarrow x)\cdot a(x\rightarrow y)$ is equal to a product of all $p_i$ or all $q_i$, depending on whether the elementary curve segment sits on the front or back of the parametrized 4-punctured sphere.)
	
	Now we are ready to define the differential. If $\dim X=1$, let $\partial\co C\rightarrow C\otimes_{\Id}\Ad$ be the homomorphism given on generators $x$ by 
	$$\partial(x)\coloneqq y\otimes a(x\rightarrow y)+z\otimes a(x\rightarrow z),$$
	where $x\rightarrow y$ and $x\rightarrow z$ are the two elementary curve segments starting at $x$. If the local system $X$ is the $n\times n$ identity matrix and $v\in\field^n$, we define
	$$\partial(x\otimes v)\coloneqq(y\otimes v) \otimes a(x\rightarrow y)+(z\otimes v)\otimes a(x\rightarrow z).$$
	Finally, if $X$ is a general local system with $\dim X=n$, we define $\partial$ as before, except that we replace the contribution of a single elementary curve segment between two intersection points $x$ and $y$ by
	$$
	(y\otimes X(v)) \otimes a(x\rightarrow y)
	\quad\text{ and }\quad
	(x\otimes X^{-1}(v)) \otimes a(y\rightarrow x),
	$$
	where we are assuming that the orientations of $\gamma$ and $x\rightarrow y$ agree. 
	We define $\Pi(\gamma_X)\coloneqq (C,\partial)$. By construction, $\Pi(\gamma_X)$ is an ungraded peculiar module. One can also define bigradings on loops, see~\cite[Definitions~4.28 and~5.1]{PQMod}, but since they are essentially defined as bigradings on $\Pi(\gamma_X)$, these definitions are rather unilluminating. 
\end{definition}

As noted in~\cite[Definition~4.30]{PQMod}, the chain homotopy type of the image of a (bi)graded loop under $\Pi$ is an invariant of the equivalence class of a loop. The following result is more remarkable: 

\begin{theorem}[{\cite[Theorem~0.4, Definition~4.30]{PQMod}}]\label{thm:classificationPecMod}
	Every peculiar module \((C,\partial)\) is chain homotopic to \(\Pi(L)\) for some multicurve \(L\) which is unique up to equivalence. Thus, we can associate with \((C,\partial)\) a multicurve \(L(C,\partial)=\{(\gamma_i,X_i)\}_{i\in I}\) such that if \((C',\partial')\) is another  peculiar module with \(L(C',\partial')=\{(\gamma'_{i'},X'_{i'})\}_{i'\in I'}\), \((C,\partial)\) and \((C',\partial')\) are homotopic iff the multicurves are the same. Moreover, the homology of the space of morphisms between two peculiar modules is chain homotopic to the Lagrangian intersection Floer theory between their associated multicurves:
	\[\Homology(\Mor((C,\partial),(C',\partial')))\cong\LagrangianFH(L(C,\partial),L(C',\partial')).\]
\end{theorem}

\begin{definition}
	Given an oriented 4-ended tangle \(T\) in a \(\mathbb{Z}\)-homology 3-ball, we define $\HFT(T)$ as the multicurve associated with $\CFTd(T)$. 
\end{definition}

\begin{theorem}[{\cite[Theorem~5.9]{PQMod}}]\label{thm:GlueingTheorem:HFT}
	With the same notation as in Theorems~\ref{thm:PecMod} and~\ref{thm:classificationPecMod}, 
	\[
	\HFL(L)\otimes V_t
	\cong
	\LagrangianFH\left(\HFT(\mr(T_1)),\HFT(T_2)\right)
	\]
	if the four open components of the tangles become identified to the same component of colour $t$ and 
	\[
	\HFL(L)
	\cong
 	\LagrangianFH\left(\HFT(\mr(T_1)),\HFT(T_2)\right)
	\]
	otherwise. 
\end{theorem}

\begin{question}[geography question for \(\HFT\)]
	Which loops arise as the invariants of 4-ended tangles?
\end{question}

Many results in this paper can be regarded as attempts to answer this question. In~\cite[Observation~6.1]{PQMod}, the author already observed the following ``global'' obstruction.

\begin{observation}\label{obs:AlexGradingOfCFTdLoops}
	For any 4-ended tangle \(T\) in a $\mathbb{Z}$-homology 3-ball \(M\), the underlying curve of each component of \(\HFT(T)\) lies in the kernel of 
	\[\pi_1(\partial M\smallsetminus \partial T)\rightarrow\pi_1(M\smallsetminus \nu(T))\rightarrow H_1(M\smallsetminus \nu(T)),\]
	where the first map is induced by the inclusion and the second is the Abelianization map. In particular, \(\HFT(T)\) does not contain any component whose underlying curve is a meridional curve around a tangle end. 
	A similar obstruction comes from the $\delta$-grading.
\end{observation}

\subsection{Quotient functors for peculiar modules}\label{subsec:BasepointZero}
Theorem~\ref{thm:classificationPecMod} can be generalized to a whole class of categories of curved complexes associated with marked surfaces with arc systems. In fact, this is the perspective taken in~\cite[Section~4]{PQMod}. For the purpose of the present paper, it is not necessary to recall any details about marked surfaces and arc systems; we only need to discuss one of the observations that follow from taking this broader perspective. 

Recall that the differential $\CFTd(T)$ is defined in terms of counting holomorphic curves and recording how often they cover certain basepoints which correspond to the elementary algebra elements $p_i$ and $q_i$. The curvature terms in $\CFTd(T)$ originate from boundary degenerations which cover all the basepoints $p_i$ or all the basepoints $q_i$ exactly once. One way to avoid these curvature terms is to simply ignore those holomorphic curves covering the basepoints $p_i$ and $q_j$ for a certain choice of $i,j\in\{1,2,3,4\}$. In terms of the algebra $\Ad$, this corresponds to setting these two generators equal to 0: 

\begin{definition}
	For \(i,j\in\{1,2,3,4\}\), the quotient homomorphism 
	$$\pi_{ij}\co \Ad\rightarrow \Ad_{ij}\coloneqq \Ad/(p_i=0=q_j)$$
	induces a functor 
	\[\mathcal{F}_{ij}\co\pqMod\rightarrow\pqMod_{ij},\]
	where \(\pqMod_{ij}=\Mod^{\Ad_{ij}}\) is the category of (bi)graded type D structures over $\Ad_{ij}$.
\end{definition}

In terms of marked surfaces with arc systems, the homomorphism $\pi_{ij}$ corresponds to adding a marker to each of the two boundary paths $p_i$ and $q_j$ in Figure~\ref{fig:nutshell}. The classification result for $\pqMod_{ij}$ is analogous to Theorem~\ref{thm:classificationPecMod}, except that we need to enlarge the set of immersed curves to include non-compact ones. Moreover, the functor $\mathcal{F}_{ij}$ is essentially the identity on immersed curves. This implies the following result:

\begin{theorem}[{\cite[Corollary~4.48, see also the proof of Corollary~5.7]{PQMod}}]\label{thm:BasepointZero}
	Let \(i,j\in\{1,2,3,4\}\). Then two peculiar modules are chain homotopic iff the same is true for their images under \(\mathcal{F}_{ij}\).  
\end{theorem}

\subsection{Generalized peculiar modules}\label{subsec:background:GeneralizedPQM}

In~\cite[Section~2]{PQMod}, the author also defined a tangle invariant $\CFTminus(T)$, which generalizes $\CFTd(T)$. Like $\CFTd(T)$, the invariant $\CFTminus(T)$ is a curved type D structure, but it is defined over a larger algebra $\Aminus_n$, where $n$ is equal to the number of closed components of the tangle $T$. For any $n>0$, there is an epimorphism $\Aminus_n\rightarrow\Aminus_0\eqqcolon\Aminus$, so we can always regard $\CFTminus(T)$ as a curved type D structure over $\Aminus$. This is what we will do throughout this paper, as it will be sufficient for  our purposes. 

\begin{definition}[{\cite[Definitions~2.10 and~2.11]{PQMod}}]
	Let $\Rpre$ be the free polynomial ring generated by the variables $p_i$ and $q_i$, where $i=1,2,3,4$. Let $\Apre$ be the $\Id$-$\Id$-algebra whose underlying $\Id$-$\Id$-bimodule structure is given by  $\iota_{s'}\Apre.\iota_{s}\coloneqq \Rpre$ for pairs $(s,s')$ of sites
	and whose algebra multiplication is defined by the unique $\Id$-$\Id$-bimodule homomorphism $\Apre\otimes_{\Id}\Apre\rightarrow \Apre$ which, for all triples $(s,s',s'')$ of sites, restricts to the multiplication map in $\Rpre$:
	\[\underbrace{\iota_{s''}.\Apre.\iota_{s'}}_{\Rpre}\otimes_{\Id}\underbrace{\iota_{s'}.\Apre.\iota_{s}}_{\Rpre}\rightarrow \underbrace{\iota_{s''}.\Apre.\iota_{s}}_{\Rpre}.\]
	Let $\R$ be the free polynomial ring in the variables $U_i$ for $i=1,2,3,4$. Via the inclusion
	$$ \R\hookrightarrow \Rpre,\quad U_i\mapsto p_iq_i,$$
	we can regard $\Apre$ as an $\R$-algebra. Let $\Aminus$ be the subalgebra of $\Apre$ generated as an $\R$-algebra by the idempotents in $\Id$ and 
	\[p_i\coloneqq \iota_{i-1}.p_i.\iota_i,\quad\text{ and }\quad q_i\coloneqq \iota_{i}.q_i.\iota_{i-1},\]
	where we take the indices $i=1,2,3,4$ modulo 4 with an offset of 1. Note that $p_iq_i=\iota_{i-1}.U_i$ and $q_ip_i=\iota_i.U_i$ as elements in $\Aminus$. Thus, any element in $\Aminus$ can be written uniquely as a sum of elements of the form 
	\[\iota_i.r, \quad p_ip_{i+1}\dots p_{k-1}p_{k}.r\quad\text{ and }\quad q_iq_{i-1}\dots q_{k+1}q_{k}.r,\]
	where $r\in \R$ is a monomial. This is the standard basis on $\Aminus$ as a vector space over $\field$. We call $\Aminus$ the \textbf{generalised peculiar algebra}. By setting $U_i=0$ for all $i=1,2,3,4$, we obtain the algebra $\Ad$. Note that we can use the corresponding 
	epimorphism
	$$\Aminus\longrightarrow\Ad$$
	to extend the $\delta$- and Alexander gradings of $\Ad$ to $\Aminus$ in a unique way.
\end{definition}

\begin{definition}[{\cite[Definition~2.16]{PQMod}}]\label{def:CFTminus}
	Given an oriented 4-ended tangle \(T\), $\CFTminus(T)$ is a curved type D structure over $\Aminus$ with curvature
	$$p^4 + q^4 + U_{i_1} U_{o_1} + U_{i_2} U_{o_2},$$
	where $\{\{i_1,o_1\},\{i_2,o_2\}\}$ is the ordered matching induced by $T$. 
	Like $\CFTd(T)$, $\CFTminus(T)$ is defined as a multi-pointed Heegaard Floer theory; we simply count more holomorphic curves in its differential. The epimorphism $\Aminus\rightarrow\Ad$ gives rise to a functor $\mathcal{E}$ between the categories of curved type D structures and $\CFTd(T)$ is the image of $\CFTminus(T)$ under this functor. 
\end{definition}

By~\cite[Theorem~2.17]{PQMod}, the relatively bigraded chain homotopy type of \(\CFTminus(T)\) is a tangle invariant. However, for the purpose of this paper, we only need to know that $\CFTd(T)$ lies in the image of $\mathcal{E}$. Of course, such a functor preserves chain homotopy types. However, it is not clear if, given two chain homotopic complexes over $\Ad$, the fact that one complex lies in the image of $\mathcal{E}$ implies that the same is true for the other complex. Nonetheless, we have the following result. 

\begin{proposition}\label{prop:arrow-pushing-for-CFTminus}
	Let \((C^-,\partial^-)\) be a bigraded curved complex over \(\Aminus\). Then \((C^-,\partial^-)\) is bigraded chain homotopic to a complex whose image under \(\mathcal{E}\) lies in the image of \(\Pi\). In particular, for any oriented 4-ended tangle \(T\), \(\Pi(\HFT(T))\) lies in the image of~\(\mathcal{E}\).
\end{proposition}
\begin{proof}
	The strategy for the proof is to lift every chain homotopy for \((C,\partial)=\mathcal{E}(C^-,\partial^-)\) in the simplification algorithm from~\cite[Section~4]{PQMod} to a chain homotopy for \((C^-,\partial^-)\). This algorithm uses two tools for changing complexes while preserving their chain homotopy types, namely the Cancellation Lemma and the Clean-Up Lemma~\cite[Lemmas~1.22 and~1.24]{PQMod}, which we always apply to single components of the differentials. 
	
	Let us consider the Cancellation Lemma first. Like in $\Ad$, the only homogeneous elements in $\Aminus$ of $\delta$-grading 0 are the idempotents. So if $\partial$ contains an identity component, then the corresponding component of $\partial^-$ is also an identity component. Therefore, the complex obtained from \((C,\partial)\) by any cancellation along an identity component of $\partial$ is equal to the image of the complex obtained from $(C^-,\partial^-)$ by the cancellation of the corresponding component of $\partial^-$. By induction on the size of the complex, we may therefore assume that   \((C,\partial)\), and hence also \((C^-,\partial^-)\), is reduced. 
	
	For any application of the Clean-Up Lemma, we need to verify that the morphism $h$ satisfies two conditions, namely $h^2=0$ and $h\partial h=0$. (Note that $D(h)=h\partial+\partial h$, so assuming $h^2=0$, the conditions $D(h)h=0$ and $hD(h)=0$ are equivalent to $h\partial h=0$.) In the algorithm, the first condition is always satisfied, because we always choose $h$ such that it connects two different generators. The second condition $h\partial h=0$ is satisfied for peculiar modules because of $\delta$-grading constraints and the fact that $p_iq_i=0$. The latter is no longer true in $\Aminus$. Instead, we can argue with the $\delta$- and the Alexander grading. Suppose $h$ goes between two generators $x$ and $y$ and there is a non-zero component of the differential $\partial$ going in the opposite direction:
	$$
	\begin{tikzcd}
	x
	\arrow[bend left=10,dashed]{r}{h}
	&
	y
	\arrow[bend left=10]{l}{\partial}
	\end{tikzcd}
	$$
	Let us abuse notation and use the letters $h$ and $\partial$ also for the algebra elements in $\Aminus$ labelling these two arrows. Then 
	$$
	\delta(y)-\delta(x)+\delta(h)=0
	\quad\text{ and }\quad
	\delta(x)-\delta(y)+\delta(\partial)=1,
	\quad\text{ so }\quad
	\delta(h)+\delta(\partial)=1.
	$$ 
	Similarly, we see that 
	\begin{equation}\label{eqn:arrow-pushing-Alex}
	A(h)+A(\partial)=0.
	\end{equation}
	Since we are assuming that \((C,\partial)\) is reduced, ie $\delta(\partial)\geq\frac{1}{2}$, we have $\delta(h)\leq\frac{1}{2}$. In other words, $h\in\{p_i,q_i,\iota_i\mid i=1,2,3,4\}$. If $h=p_i$, then $\delta(\partial)=\frac{1}{2}$. This implies $\partial=q_i$, since the left and right idempotents of $\partial$ are the right and left idempotents of $h$, respectively. But $q_i$ and $p_i$ have the same (non-zero) Alexander grading, contradicting \eqref{eqn:arrow-pushing-Alex}. Similarly, we can argue in the case $h=q_i$. Finally, $h=\iota_i$ implies that $\delta(\partial)=1$. Then the left and right idempotents of $\partial$ agree, so $\partial=U_j.\iota_i$ for some $j=1,2,3,4$ and, again, we obtain a contradiction to \eqref{eqn:arrow-pushing-Alex}. 
\end{proof}

%% file: sections/MCG.tex

\section{The action of the mapping class group}\label{sec:ActionMCG}

In our definition of a 4-ended tangle $T$ (Definition~\ref{def:tangle}), we fixed a parametrization of the boundary of the tangle. 
The mapping class group of the 4-punctured sphere 
$$\Mod(\FourPuncturedSphere)=\pi_0(\operatorname{Homeo^+}(\FourPuncturedSphere))$$
acts on this parametrization and hence on the tangles. It also acts on the underlying curves of $\HFT(T)$. The main result from this section is that these actions commute with $\HFT$, so the invariant $\HFT(T)$ is independent of the parametrization: 

\begin{theorem}\label{thm:MCGaction}
	Let \(T\) be a 4-ended tangle in a \(\mathbb{Z}\)-homology 3-ball and \(\rho\in\Mod(\FourPuncturedSphere)\). Then 
	$$\HFT(\rho T)=\rho(\HFT(T)).$$ 
\end{theorem}

$\Mod(\FourPuncturedSphere)$ is generated by the half-twists along the four parametrizing arcs, so it suffices to prove the theorem for those four generators. Because of the inherent symmetry of the parametrization, it is actually sufficient to understand just one of them. Let us focus on the half-twist $\tau$ whose action on tangles and the parametrizing arcs is shown in Figures~\ref{fig:AddingASingleCrossing} and~\ref{fig:DehnTwistArcs}, respectively. 

\begin{wrapfigure}{r}{0.3\textwidth}
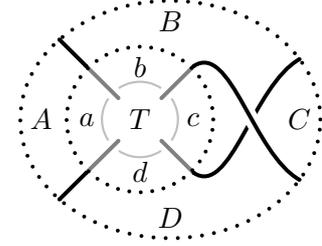

	\centering
	\vspace{-5pt}
	\AddingASingleCrossing
	\vspace{-5pt}
	\caption{The tangle $\tau T$}\label{fig:AddingASingleCrossing}
\end{wrapfigure}
Consider the two parametrizations $\{\textcolor{blue}{a}, \textcolor{blue}{b}, \textcolor{blue}{c}, \textcolor{blue}{d}\}$ and $\{\textcolor{red}{A}, \textcolor{red}{B}, \textcolor{red}{C}, \textcolor{red}{D}\}$ in Figure~\ref{fig:DehnTwistArcs}. They give rise to two distinct maps $\Pi\coloneqq\Pi_{\{\textcolor{blue}{a}, \textcolor{blue}{b}, \textcolor{blue}{c}, \textcolor{blue}{d}\}}$ and $\Pi'\coloneqq\Pi_{\{\textcolor{red}{A}, \textcolor{red}{B}, \textcolor{red}{C}, \textcolor{red}{D}\}}$ which establish the correspondence from Definition~\ref{def:loopsTOpqMod} between multicurves and peculiar modules. The reason why the labels $\textcolor{blue}{p_4}$ and $\textcolor{blue}{q_3}$ of the first parametrization and $\textcolor{red}{p_3}$ and $\textcolor{red}{q_4}$ of the second are crossed out is because it turns out to be useful to set these variables equal to zero in the corresponding algebras, ie to work with type D structures over the algebras $\Ad_{43}$ and $\Ad_{34}$.

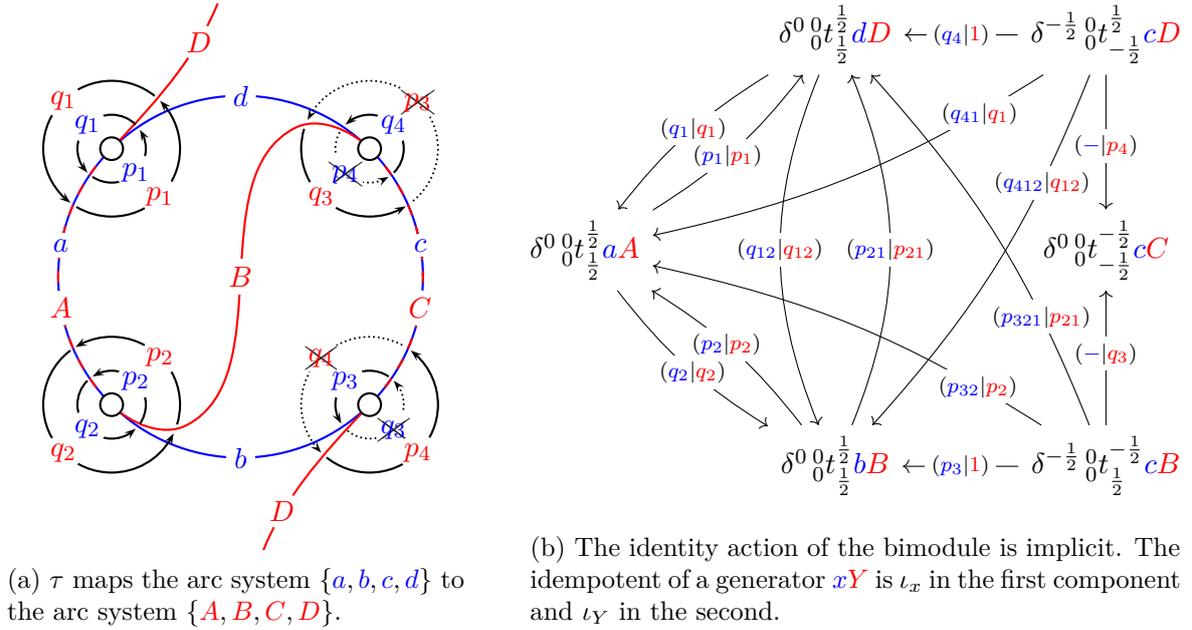
\begin{figure}[t]
	\centering
\begin{subfigure}[b]{0.4\textwidth}\centering
	\DehnTwistArcs
	\caption{$\tau$ maps the arc system $\{\textcolor{blue}{a}, \textcolor{blue}{b}, \textcolor{blue}{c}, \textcolor{blue}{d}\}$ to the arc system $\{\textcolor{red}{A}, \textcolor{red}{B}, \textcolor{red}{C}, \textcolor{red}{D}\}$.}\label{fig:DehnTwistArcs}
\end{subfigure}
\quad
\begin{subfigure}[b]{0.56\textwidth}
	$$
	\begin{tikzcd}[row sep=1.8cm, column sep=1.5cm]
	&
	\delta^0\Alex{0}{0}{\frac{1}{2}}{\frac{1}{2}}\textcolor{blue}{d}\textcolor{red}{D}
	\arrow[bend right=10,pos=0.45]{ld}[description]{\qq{1}{1}}
	\arrow[bend right=22]{dd}[description]{\qq{12}{12}}
	&
	\delta^{-\frac{1}{2}}\Alex{0}{0}{-\frac{1}{2}}{\frac{1}{2}}\textcolor{blue}{c}\textcolor{red}{D}
	\arrow{l}[description]{\qn{4}}
	\arrow[bend left=10,pos=0.16]{lld}[description]{\qq{41}{1}}
	\arrow[bend left=10,pos=0.27]{ldd}[description]{\qq{412}{12}}
	\arrow[pos=0.53]{d}[description]{\np{4}}
	\\
	\delta^{0}\Alex{0}{0}{\frac{1}{2}}{\frac{1}{2}}\textcolor{blue}{ a}\textcolor{red}{A}
	\arrow[bend right=10,pos=0.45]{ur}[description]{\pp{1}{1}}
	\arrow[bend right=10,pos=0.55]{rd}[description]{\qq{2}{2}}
	&&
	\delta^{0}\Alex{0}{0}{-\frac{1}{2}}{-\frac{1}{2}}\textcolor{blue}{ c}\textcolor{red}{C}
	\\
	&
	\delta^{0}\Alex{0}{0}{\frac{1}{2}}{\frac{1}{2}}\textcolor{blue}{ b}\textcolor{red}{B}
	\arrow[bend right=10,pos=0.55]{lu}[description]{\pp{2}{2}}
	\arrow[bend right=22]{uu}[description]{\pp{21}{21}}
	&
	\delta^{-\frac{1}{2}}\Alex{0}{0}{\frac{1}{2}}{-\frac{1}{2}}\textcolor{blue}{c}\textcolor{red}{B}
	\arrow[pos=0.53]{u}[description]{\nq{3}}
	\arrow[bend right=10,pos=0.16]{llu}[description]{\pp{32}{2}}
	\arrow[bend right=10,pos=0.27]{luu}[description]{\pp{321}{21}}
	\arrow{l}[description]{\pn{3}}
	\end{tikzcd}$$
\caption{The identity action of the bimodule is implicit. The idempotent of a generator $\textcolor{blue}{ x}\textcolor{red}{ Y}$ is $\iota_x$ in the first component and $\iota_Y$ in the second.}\label{fig:DehnTwistBimodule}
\end{subfigure}
\caption{The action of the half-twist $\tau$ on the parametrization of the 4-punctured sphere (a) and the type AD bimodule $\mathcal{D}(\tau)$ (b)}\label{fig:DehnTwist}
\end{figure}

\begin{definition}
	Let $\mathcal{D}(\tau)$ be the strictly unital type~AD $\Ad_{43}$-$\Ad_{34}$-bimodule defined in Figure~\ref{fig:DehnTwistBimodule}. Furthermore, let $\rr_{34}$ be the operation on type D structures which interchanges the third and the fourth entry in the generalized Alexander grading:  $\Alex{a}{b}{c}{d}\mapsto\Alex{a}{b}{d}{c}$.
\end{definition}

\begin{Remark}
	Note that the action of $\rr_{34}(\cdot)\boxtimes\mathcal{D}(\tau)$ on generators corresponds exactly to the action of $\tau$ on Kauffman states of $T$ from~\cite{HDsForTangles}.
\end{Remark}

\begin{lemma}\label{lem:half DehnTwistBimodule}
	For any loop $\gamma_X$, we have the following bigraded homotopy equivalence:
	$$
	\mathcal{F}_{34}(\Pi'(\gamma_X))
	\cong
	\rr_{34}(\mathcal{F}_{43}(\Pi(\gamma_X)))\boxtimes\,\mathcal{D}(\tau).
	$$ 
\end{lemma}

\begin{lemma}\label{lem:AddingASingleCrossing}
	Let \(T\) be a 4-ended tangle and \(\tau T\) the tangle obtained by adding a single crossing to it as shown in Figure~\ref{fig:AddingASingleCrossing}. Then,
	$$\mathcal{F}_{34}(\CFTd(\tau T))\cong\rr_{34}(\mathcal{F}_{43}(\CFTd(T)))\boxtimes\,\mathcal{D}(\tau).$$ 
\end{lemma}	

\begin{proof}[Proof of Theorem~\ref{thm:MCGaction}]
	We have
	\begin{align*}
	\mathcal{F}_{34}(\Pi(\tau(\HFT(T))))
	&
	=
	\mathcal{F}_{34}(\Pi'(\HFT(T)))
	&&
	\text{(by the definition of $\tau$)}
	\\
	&
	\cong \rr_{34}(\mathcal{F}_{43}(\Pi(\HFT(T))))\boxtimes\mathcal{D}(\tau)
	&&
	\text{(by Lemma~\ref{lem:half DehnTwistBimodule})}
	\\
	&
	\cong \rr_{34}(\mathcal{F}_{43}(\CFTd(T)))\boxtimes\mathcal{D}(\tau)
	&&
	\text{(by the definition of $\HFT(T)$)}
	\\
	&
	\cong 
	\mathcal{F}_{34}(\CFTd(\tau T))
	&&
	\text{(by Lemma~\ref{lem:AddingASingleCrossing})}
	\\
	&
	\cong\mathcal{F}_{34}(\Pi(\HFT(\tau T)))
	&&
	\text{(by the definition of $\HFT(\tau T)$).}
	\end{align*}
	By Theorem~\ref{thm:BasepointZero}, this implies that
	$$
	\Pi(\tau(\HFT(T)))=\Pi(\HFT(\tau T))
	$$
	and by Theorem~\ref{thm:classificationPecMod}, this in turn implies
	$$
	\tau(\HFT(T))=\HFT(\tau T).
	$$
	Using cyclic permutations of the arcs, we see that the Lemmas~\ref{lem:half DehnTwistBimodule} and~\ref{lem:AddingASingleCrossing} also hold for the corresponding half-twists along the other three arcs. 
\end{proof}

\begin{figure}[t]
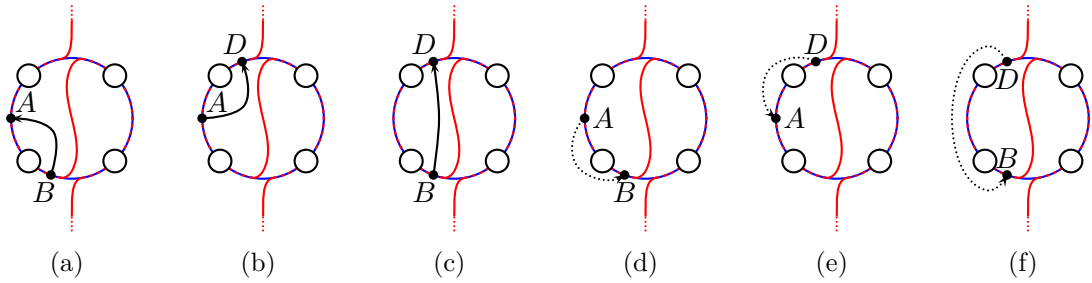

	\centering
	\begin{subfigure}[b]{0.15\textwidth}\centering
		\MCGActionOnBoringSegmentsA
		\caption{}\label{fig:MCGActionOnBoringSegmentsA}
	\end{subfigure}
	\begin{subfigure}[b]{0.15\textwidth}\centering
		\MCGActionOnBoringSegmentsB
		\caption{}\label{fig:MCGActionOnBoringSegmentsB}
	\end{subfigure}
	\begin{subfigure}[b]{0.15\textwidth}\centering
		\MCGActionOnBoringSegmentsC	
		\caption{}\label{fig:MCGActionOnBoringSegmentsC}
	\end{subfigure}
	\begin{subfigure}[b]{0.15\textwidth}\centering
		\MCGActionOnBoringSegmentsD		
		\caption{}\label{fig:MCGActionOnBoringSegmentsD}
	\end{subfigure}
	\begin{subfigure}[b]{0.15\textwidth}\centering
		\MCGActionOnBoringSegmentsE
		\caption{}\label{fig:MCGActionOnBoringSegmentsE}
	\end{subfigure}
	\begin{subfigure}[b]{0.15\textwidth}\centering
		\MCGActionOnBoringSegmentsF
		\caption{}\label{fig:MCGActionOnBoringSegmentsF}
	\end{subfigure}
	\caption{Elementary curve segments which stay away from the arc $c$. The action of the bimodule $\mathcal{D}(\tau)$ on these is trivial.}\label{fig:MCGActionOnBoringSegments}
\end{figure}

\begin{figure}[t]
	\begin{tabular}{ccc}
		&
		curve segment $S$ 
		&
		$S\boxtimes\mathcal{D}(\tau)$
		\\
		\hline
		(a)&
		$
		\begin{tikzcd}
		b
		&
		c
		\arrow[swap]{l}{p_{3}}
		\end{tikzcd}
		$
		&
		$
		\begin{tikzcd}
		B
		&
		B
		\arrow[swap,dashed]{l}{1}
		\arrow[dotted]{r}{q_3}
		&
		C
		&
		D
		\arrow[swap]{l}{p_4}
		\end{tikzcd}
		$
		\\
		(b)&
		$
		\begin{tikzcd}
		a
		&
		c
		\arrow[swap]{l}{p_{32}}
		\end{tikzcd}
		$
		&
		$
		\begin{tikzcd}
		A
		&
		B
		\arrow[swap]{l}{p_2}
		\arrow[dotted]{r}{q_3}
		&
		C
		&
		D
		\arrow[swap]{l}{p_4}
		\end{tikzcd}
		$
		\\
		(c)&
		$
		\begin{tikzcd}
		d
		&
		c
		\arrow[swap]{l}{p_{321}}
		\end{tikzcd}
		$
		&
		$
		\begin{tikzcd}
		D
		&
		B
		\arrow[swap]{l}{p_{21}}
		\arrow[dotted]{r}{q_3}
		&
		C
		&
		D
		\arrow[swap]{l}{p_4}
		\end{tikzcd}
		$
		\\
		(d)&
		$
		\begin{tikzcd}
		d
		&
		c
		\arrow[swap,dotted]{l}{q_{4}}
		\end{tikzcd}
		$
		&
		$
		\begin{tikzcd}
		D
		&
		D
		\arrow[swap,dashed]{l}{1}
		\arrow{r}{p_4}
		&
		C
		&
		B
		\arrow[swap,dotted]{l}{q_3}
		\end{tikzcd}
		$
		\\
		(e)&
		$
		\begin{tikzcd}
		a
		&
		c
		\arrow[swap,dotted]{l}{q_{41}}
		\end{tikzcd}
		$
		&
		$
		\begin{tikzcd}
		A
		&
		D
		\arrow[swap,dotted]{l}{q_1}
		\arrow{r}{p_4}
		&
		C
		&
		B
		\arrow[swap,dotted]{l}{q_3}
		\end{tikzcd}
		$
		\\
		(f)&
		$
		\begin{tikzcd}
		b
		&
		c
		\arrow[swap,dotted]{l}{q_{412}}
		\end{tikzcd}
		$
		&
		$
		\begin{tikzcd}
		B
		&
		D
		\arrow[swap,dotted]{l}{q_{12}}
		\arrow{r}{p_4}
		&
		C
		&
		B
		\arrow[swap,dotted]{l}{q_3}
		\end{tikzcd}
		$
	\end{tabular}
	\caption{The action of the bimodule $\mathcal{D}(\tau)$ on each elementary curve segment meeting the arc $c$}\label{fig:MCGActionOnSegments}
\end{figure}
\begin{figure}[t]
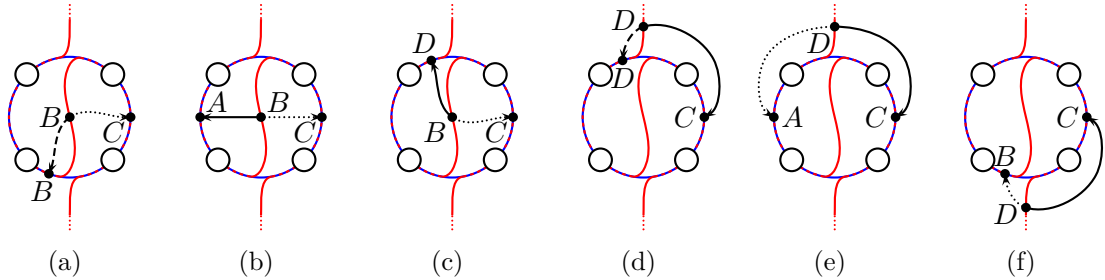

	\centering
	\begin{subfigure}[b]{0.15\textwidth}\centering
		\MCGActionOnInterestingSegmentsA
		\caption{}\label{fig:MCGActionOnInterestingSegmentsA}
	\end{subfigure}
	\begin{subfigure}[b]{0.15\textwidth}\centering
		\MCGActionOnInterestingSegmentsB
		\caption{}\label{fig:MCGActionOnInterestingSegmentsB}
	\end{subfigure}
	\begin{subfigure}[b]{0.15\textwidth}\centering
		\MCGActionOnInterestingSegmentsC
		\caption{}\label{fig:MCGActionOnInterestingSegmentsC}
	\end{subfigure}
	\begin{subfigure}[b]{0.15\textwidth}\centering
		\MCGActionOnInterestingSegmentsD
		\caption{}\label{fig:MCGActionOnInterestingSegmentsD}
	\end{subfigure}
	\begin{subfigure}[b]{0.15\textwidth}\centering
		\MCGActionOnInterestingSegmentsE
		\caption{}\label{fig:MCGActionOnInterestingSegmentsE}
	\end{subfigure}
	\begin{subfigure}[b]{0.15\textwidth}\centering
		\MCGActionOnInterestingSegmentsF
		\caption{}\label{fig:MCGActionOnInterestingSegmentsF}
	\end{subfigure}
	\caption{Graphical interpretation of the computation in Figure~\ref{fig:MCGActionOnSegments}. The curve segments are drawn onto separate copies of Figure~\ref{fig:DehnTwistArcs}.}\label{fig:MCGActionOnInterestingSegments}
\end{figure}

\begin{proof}[Proof of Lemma~\ref{lem:half DehnTwistBimodule}]
	Let us first assume for simplicity that the local system $X$ is trivial. 
	On elementary curve segments of $\gamma$ with respect to the first arc system, $\mathcal{D}(\tau)$ acts like the identity if the segment in question does not end on the arc $c$; for an illustration, see Figure~\ref{fig:MCGActionOnBoringSegments}.
	In Figure~\ref{fig:MCGActionOnSegments}, we compute the result of pairing the bimodule $\mathcal{D}(\tau)$ with the remaining elementary curve segments. Note that generators in idempotents $a$, $b$ and $d$ contribute one new generator in the $\boxtimes$-product, so for these generators, it is easy to put the curve segments back together. Generators in idempotent $c$ contribute three generators in the $\boxtimes$-product. Nonetheless, we can easily reassemble the complexes along the new generators in idempotent $C$, since an arrow leaving a generator in idempotent $c$ contributes an arrow from the new generator $B$ (respectively $D$) iff it is labelled by a power of $p$ (respectively $q$). In other words, we may safely ignore each of the final arrows in the last column of Figure~\ref{fig:MCGActionOnSegments} and reassemble the complexes like puzzle pieces which are joined along the generators/vertices $aA$, $bB$, $cC$ and $dD$. Figure~\ref{fig:MCGActionOnInterestingSegments} shows those pieces on the 4-punctured sphere with the two arc systems. As we can see, they agree with the elementary curve segments that we started with. However, there could still be some identity components, namely the dashed arrows in Subfigures (a) and (d). Those can be cancelled.  One can easily see that such a cancellation just homotopes the curve away from the arcs $B$ and $D$, which proves the claim.
	
	One can easily adapt this argument to curves $\gamma_X$ with non-trivial local systems $X$; we leave the details to the reader. 
\end{proof}

\begin{figure}[hp!]
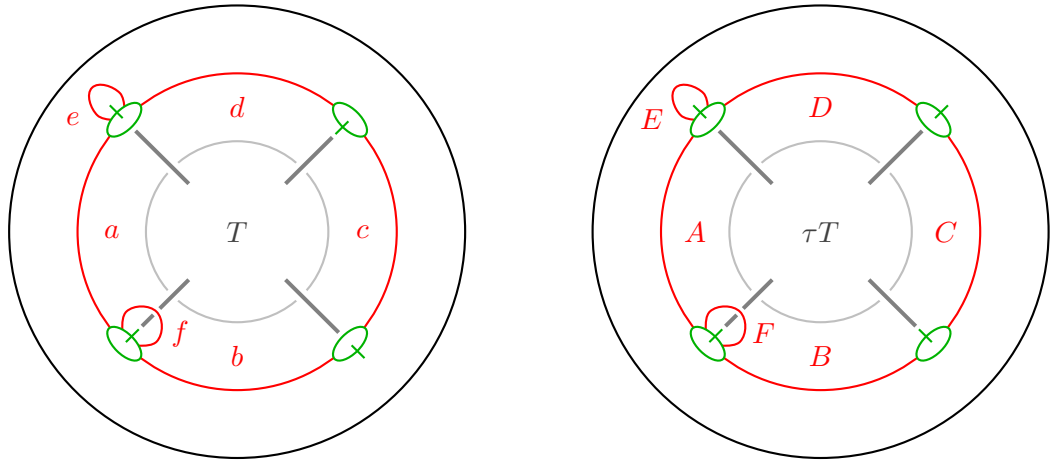
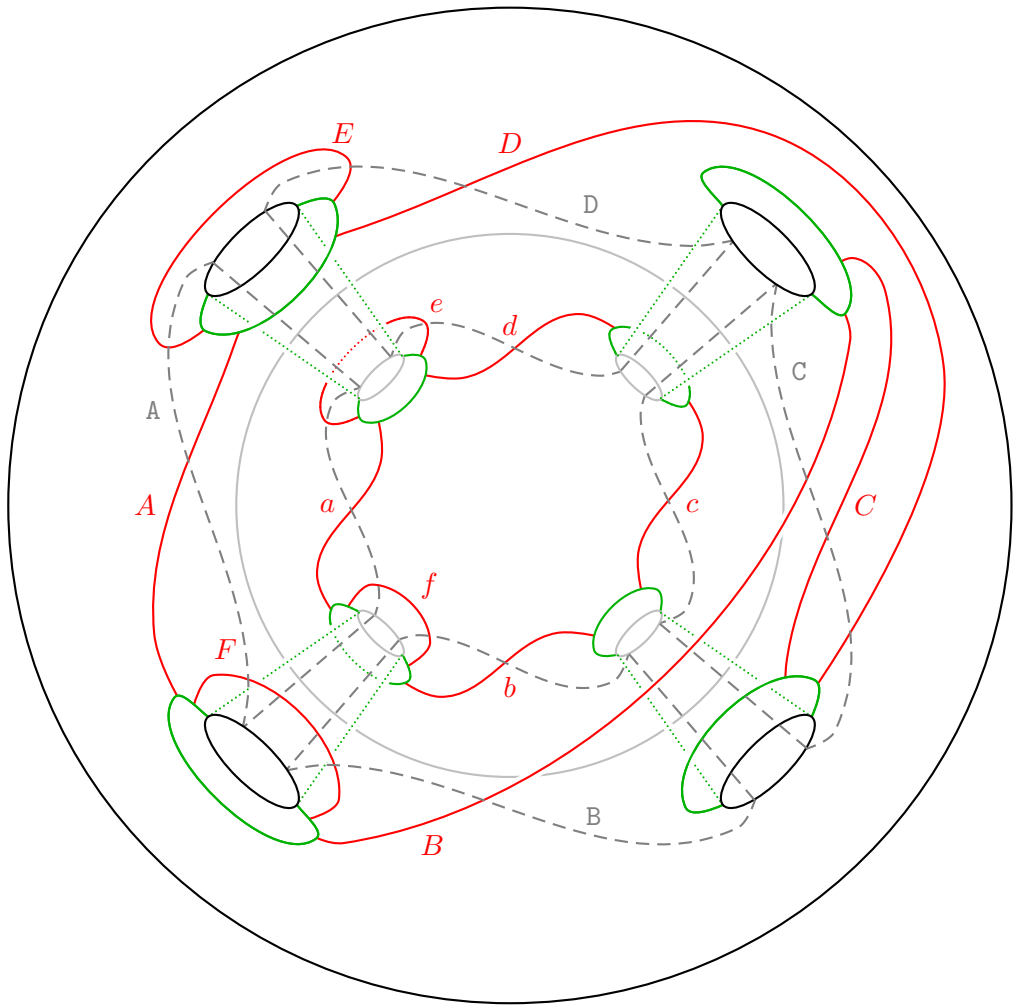

	\centering
	\bigskip
	\bigskip
	\begin{subfigure}[b]{0.45\textwidth}\centering
		\GlueingTangleT
		\caption{The parametrization on $\partial M_T$}\label{fig:GlueingTangleT}
	\end{subfigure}
	\quad
	\begin{subfigure}[b]{0.45\textwidth}\centering
		\GlueingTangleTP
		\caption{The parametrization on $\partial M_{\tau T}$}\label{fig:GlueingTangleTP}
	\end{subfigure}
	\\
	\bigskip
	\bigskip
	\begin{subfigure}[b]{0.9\textwidth}\centering
		\DehnTwistAD
		\caption{The parametrization on the boundary of the thickened 4-punctured sphere $X$. If we cut the boundary along the (dashed) simple closed curves \textcolor{gray}{\texttt{A}}, \textcolor{gray}{\texttt{B}}, \textcolor{gray}{\texttt{C}} and \textcolor{gray}{\texttt{D}}, we obtain the Heegaard surface from Figure~\ref{fig:DehnTwistHDs}.}\label{fig:DehnTwistAD}
	\end{subfigure}
	\caption{A decomposition of $M_{\tau T}$ from Figure~\ref{fig:AddingASingleCrossing}}\label{fig:DehnTwistADMflds}
\end{figure}

\begin{figure}[t]
	\centering
	\begin{subfigure}[b]{\textwidth}
		\centering
		\DehnTwistHDs
		\caption{A Heegaard diagram for $X$, consisting of two punctured spheres, drawn as discs whose boundaries correspond to the points at $\infty$. The two spheres are connected to each other by tubes along the grey rectangles, which correspond to the simple closed curves \textcolor{gray}{\texttt{A}}, \textcolor{gray}{\texttt{B}}, \textcolor{gray}{\texttt{C}} and \textcolor{gray}{\texttt{D}} from Figure~\ref{fig:DehnTwistADMflds}. The orientation on the surface is such that the normal vector (determined by the right-hand rule) points into the projection plane. 
		}\label{fig:DehnTwistHDs}
	\end{subfigure}
	\begin{subfigure}[b]{\textwidth}\centering
		$$
		\begin{tikzcd}[row sep=2.5cm, column sep=2.5cm]
		&
		\textcolor{blue}{ d}\textcolor{red}{D}:acD_1bef
		\arrow[bend right=10,pos=0.45]{ld}[description]{\{\textcolor{blue}{ q_1}/\textcolor{red}{ Q_1}\}}
		\arrow[bend right=30]{dd}[description]{\{\textcolor{blue}{ q_{1}}/\textcolor{red}{ Q_{1}},\textcolor{blue}{ q_{2}}/\textcolor{red}{ Q_{2}}\}}
		&
		\textcolor{blue}{ c}\textcolor{red}{D}:aD_2dbef
		\arrow{l}[description]{\{\textcolor{blue}{ q_4}\}}
		\arrow[bend left=10,pos=0.2]{lld}[description]{\{\textcolor{blue}{ q_{4}},\textcolor{blue}{ q_{1}}/\textcolor{red}{ Q_{1}}\}}
		\arrow[bend left=8,pos=0.29]{ldd}[description]{\{\textcolor{blue}{ q_{4}},\textcolor{blue}{ q_{1}}/\textcolor{red}{ Q_{1}},\textcolor{blue}{ q_{2}}/\textcolor{red}{ Q_{2}}\}}
		\arrow[pos=0.53]{d}[description]{\{\textcolor{red}{ P_4}\}}
		\\
		\textcolor{blue}{ a}\textcolor{red}{A}:Acdbef
		\arrow[bend right=10,pos=0.45]{ur}[description]{\{\textcolor{blue}{ p_1}/\textcolor{red}{ P_1}\}}
		\arrow[bend right=10,pos=0.55]{rd}[description]{\{\textcolor{blue}{ q_2}/\textcolor{red}{ Q_2}\}}
		&&
		\textcolor{blue}{c}\textcolor{red}{C}:aCdbef
		\\
		&
		\textcolor{blue}{b}\textcolor{red}{B}:acdB_1ef
		\arrow[bend right=10,pos=0.55]{lu}[description]{\{\textcolor{blue}{ p_2}/\textcolor{red}{P_2}\}}
		\arrow[bend right=30]{uu}[description]{\{\textcolor{blue}{p_{1}}/\textcolor{red}{P_{1}},\textcolor{blue}{p_{2}}/\textcolor{red}{P_{2}}\}}
		&
		\textcolor{blue}{c}\textcolor{red}{B}:aB_2dbef
		\arrow[pos=0.53]{u}[description]{\textcolor{red}{Q_3}\}}
		\arrow[bend right=10,pos=0.2]{llu}[description]{\{\textcolor{blue}{p_{3}},\textcolor{blue}{p_{2}}/\textcolor{red}{P_{2}}\}}
		\arrow[bend right=8,pos=0.29]{luu}[description]{\{\textcolor{blue}{p_{3}},\textcolor{blue}{p_{2}}/\textcolor{red}{P_{2}},\textcolor{blue}{ p_{1}}/\textcolor{red}{P_{1}}\}}
		\arrow{l}[description]{\{\textcolor{blue}{p_3}\}}
		\end{tikzcd}$$
		\caption{The domains that contribute to the type~AD structure $\mathcal{D}(\tau)$}\label{fig:DehnTwistDomains}
	\end{subfigure}
	\caption{A Heegaard diagram for the bordered sutured manifold~$X$ from Figure~\ref{fig:DehnTwistADMflds} and some computations of generators and domains}\label{fig:DehnTwistADresult}
\end{figure}
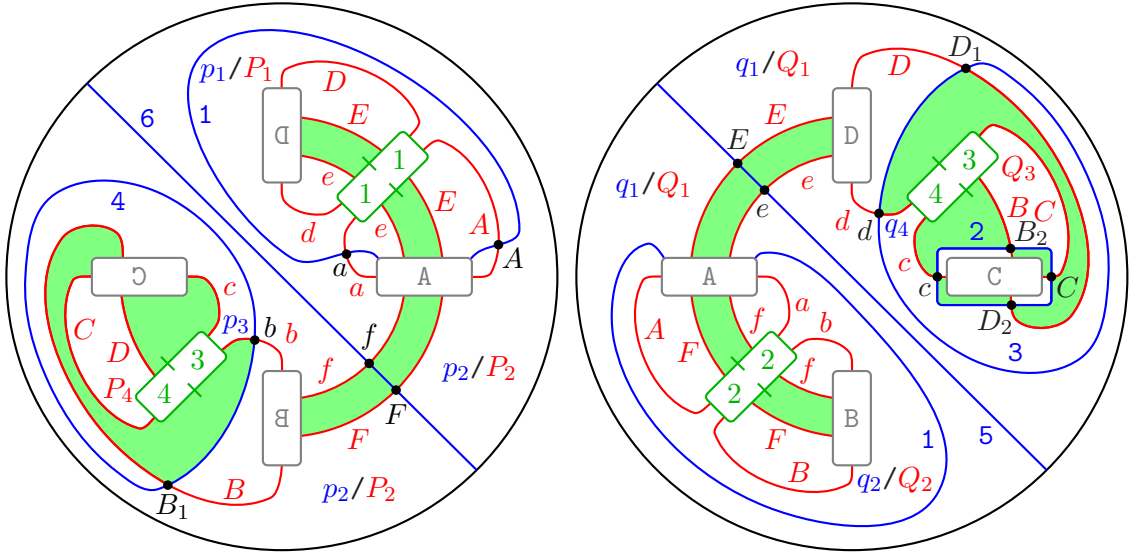
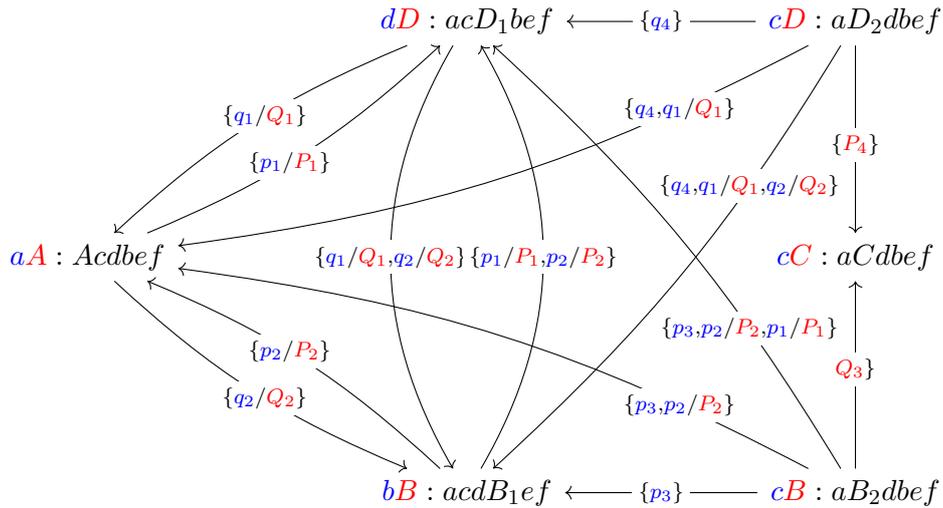

\begin{proof}[Proof of Lemma~\ref{lem:AddingASingleCrossing}]
	The strategy of this proof is very similar to that of the Glueing Theorem~\cite[Theorem~3.7]{PQMod}: we glue a thickened 4-punctured sphere $X=I\times (\FourPuncturedSphere)$ onto the tangle complement $M_T$. This allows us to change the parametrization on the boundary of the tangle complement as specified by the arc diagrams in Figure~\ref{fig:DehnTwistADMflds}. 
	
	Before we begin with the actual proof, let us introduce some notation. 
	Let $\mathcal{B}_{43}$ be the bordered sutured algebra with five moving strands corresponding to the arc diagram on $\partial M_T$ and $\mathcal{I}_{43}$ the corresponding ring of idempotents. Let $\mathcal{I}'_{43}$ be the subring of idempotents occupying the $\alpha$-arcs $\textcolor{red}{ e}$ and $\textcolor{red}{ f}$ and $\iota_{43}\co\mathcal{B}'_{43}\hookrightarrow\mathcal{B}_{43}$ the subalgebra $\mathcal{I}'_{43}.\mathcal{B}_{43}.\mathcal{I}'_{43}$ of~$\mathcal{B}_{43}$. Similarly define $\mathcal{I}_{34}$, $\mathcal{I}'_{34}$ and $\iota_{2}\co\mathcal{B}'_{34}\hookrightarrow\mathcal{B}_{34}$ for $M_{\tau T}$. Note that we can regard $\Ad_{43}$ and $\Ad_{34}$ as quotients of $\mathcal{B}'_{43}$ and $\mathcal{B}'_{34}$, respectively. Let $\pi'_{43}\co\mathcal{B}'_{43}\rightarrow\Ad_{43}$ and $\pi'_{34}\co\mathcal{B}'_{34}\rightarrow\Ad_{34}$ be the corresponding quotient maps.
	
	The lemma follows from a computation of the type AD bimodule of the bordered sutured manifold $X$. Figure~\ref{fig:DehnTwistHDs} shows a  Heegaard diagram which represents $X$. This can be seen by finding a sequence of handleslides of $\beta$-curves such that four of the $\beta$-curves are parallel to the simple closed curves $\textcolor{gray}{\texttt{A}}$, $\textcolor{gray}{\texttt{B}}$, $\textcolor{gray}{\texttt{C}}$ and $\textcolor{gray}{\texttt{D}}$, respectively, and the other two are parallel to, say, the two boundary components labelled $\textcolor{darkgreen}{\texttt{3}}/\textcolor{darkgreen}{\texttt{4}}$ and $\textcolor{darkgreen}{\texttt{4}}/\textcolor{darkgreen}{\texttt{3}}$, respectively. This particular Heegaard diagram is chosen such that it minimizes the number of generators. Let us explain this diagram in more detail. The regions adjacent to basepoints are shaded light green (\textcolor{lightgreen}{$\blacksquare$}). Intersection points are labelled by black Roman letters. The $\alpha$-arcs are labelled by $\textcolor{red}{ a}$, $\textcolor{red}{ b}$, $\textcolor{red}{ c}$, $\textcolor{red}{ d}$, $\textcolor{red}{ e}$, $\textcolor{red}{ f}$ and $\textcolor{red}{ A}$, $\textcolor{red}{ B}$, $\textcolor{red}{ C}$, $\textcolor{red}{ D}$, $\textcolor{red}{ E}$, $\textcolor{red}{ F}$, respectively, as in Figure~\ref{fig:DehnTwistADMflds}. Furthermore, the $\beta$-curves are numbered from $\textcolor{blue}{\texttt{1}}$ to $\textcolor{blue}{\texttt{6}}$. Finally, all regions in the Heegaard diagram are labelled by $\textcolor{red}{ Q_i}$, $\textcolor{red}{ P_i}$, $\textcolor{blue}{ p_i}$ and $\textcolor{blue}{ q_i}$ or combinations thereof. Note that the diagram is provincially admissible, since all regions meet the boundary. Therefore, we can use it to compute $\mathcal{I}'_{43}.\BSAD(X)$. Its generators are shown as the vertices of the graph in Figure~\ref{fig:DehnTwistDomains}; for each generator, the first two letters indicate its corresponding idempotents; the letters after the colon indicate the corresponding intersection points, ordered by the indices of the $\beta$-curves they occupy. Since the generators are uniquely determined by their idempotents, we will use them in the following as the names of the generators. 
	
	In the following, let us write domains $D$ as formal differences $D_+-D_-$ of unordered sets of regions $D_+$ and $D_-$ with $D_+\cap D_-=\emptyset$ such that 
	$$D=\sum_{r\in D_+}r-\sum_{r\in D_-}r.$$
	Let us calculate some connecting domains between the generators. First of all, here are some bigons and squares with a single boundary puncture:
	\begin{center}
		\begin{tabular}{cccc}
			$\{\textcolor{blue}{ q_4}\}: \textcolor{blue}{ c}\textcolor{red}{ D}\rightarrow\textcolor{blue}{ d}\textcolor{red}{ D},$ &
			$\{\textcolor{red}{Q_3}\}: \textcolor{blue}{ c}\textcolor{red}{ B}\rightarrow\textcolor{blue}{ c}\textcolor{red}{ C},$&
			$\{\textcolor{blue}{ p_3}\}: \textcolor{blue}{ c}\textcolor{red}{ B}\rightarrow\textcolor{blue}{ b}\textcolor{red}{ B},$ &
			$\{\textcolor{red}{ P_4}\}: \textcolor{blue}{ c}\textcolor{red}{ D}\rightarrow\textcolor{blue}{ c}\textcolor{red}{ C}.$
		\end{tabular}
	\end{center}
	The following domains are squares with two boundary punctures, one on the type A and one on the type D side:
	\begin{center}
		\begin{tabular}{cccc}
			$\{\textcolor{blue}{ p_1},\textcolor{red}{ P_1}\}: \textcolor{blue}{ a}\textcolor{red}{ A} \rightarrow\textcolor{blue}{ d}\textcolor{red}{ D},$ &
			$\{\textcolor{blue}{ q_2},\textcolor{red}{ Q_2}\}: \textcolor{blue}{ a}\textcolor{red}{ A} \rightarrow\textcolor{blue}{ b}\textcolor{red}{ B}.$
		\end{tabular}
	\end{center}
	All those domains above contribute to the type~AD structure. Next, let us compute the group of periodic domains of our Heegaard diagram. It is easy to see that it is freely generated by the following three domains:
	$$
	\{\textcolor{blue}{ p_1}/\textcolor{red}{ P_1},\textcolor{blue}{ q_1}/\textcolor{red}{ Q_1}\},\quad
	\{\textcolor{blue}{ p_2}/\textcolor{red}{ P_2},\textcolor{blue}{ q_2}/\textcolor{red}{ Q_2}\},\quad
	\{\textcolor{blue}{ p_3},\textcolor{blue}{ p_1}/\textcolor{red}{ P_1},\textcolor{red}{ P_4}\}-\{\textcolor{blue}{ q_4},\textcolor{blue}{ q_2}/\textcolor{red}{ Q_2},\textcolor{red}{ Q_3}\}.
	$$
	It is now elementary to check that indeed, Figure~\ref{fig:DehnTwistDomains} shows all connecting domains with non-negative multiplicities whose non-zero multiplicity regions do not include \textit{both} those labelled by $\textcolor{red}{ P_i}$ and $\textcolor{red}{Q_j}$. We do not know (yet) whether they all contribute to the type AD structure, but it turns out, this is all we need to compute from the Heegaard diagram. To explain this, let us step back for a moment. Zarev's Pairing Theorem \cite[Theorem~4]{ZarevThesis} tells us that
	\begin{equation}\label{eqn:FirstGlueingProofStep1}
	\BSD(M_{\tau T})^{\mathcal{B}_{34}}\cong \BSD(M_T)^{\mathcal{B}_{43}}\boxtimes\typeA{\mathcal{B}_{43}}{\BSAD(X)}^{\mathcal{B}_{34}}.
	\end{equation}
	As in the proof of the Glueing Theorem for $\CFTd$, observe that we can choose a Heegaard diagram for $M_T$ where the two $\alpha$-arcs that have ends on the same suture do not intersect any $\beta$-curve. Thus, generators of its type~D structure belong to the idempotents that occupy one of the four arcs $\textcolor{red}{a}$, $\textcolor{red}{ b}$, $\textcolor{red}{ c}$ or $\textcolor{red}{ d}$. Moreover, the labels of the type~D structure for $M_T$ are contained in $\mathcal{B}'_{43}$, ie $\BSD(M_T)$ lies in the image of the functor of type D structures $\mathcal{F}^D_{\iota_{43}}$ induced by the inclusion $\iota_{43}$ (see \cite[Definition~1.15]{PQMod}). Thus, by the Pairing Adjunction \cite[Theorem~1.21]{PQMod}, the right hand side of (\ref{eqn:FirstGlueingProofStep1}) is equal to	
	\begin{equation}\label{eqn:FirstGlueingProofStep2}
	\BSD(M_T)^{\mathcal{B}'_{43}}\boxtimes\typeA{\mathcal{B}'_{43}}{\mathcal{F}^{AD}_{\iota_{43},\id}\left(\BSAD(X)\right)}^{\mathcal{B}_{34}}
	\end{equation}
	where $\mathcal{F}^{AD}_{\iota_{43},\id}$ is the functor induced by the inclusion $\iota_{43}$ and the identity on $\mathcal{B}_{34}$.
	Similarly, we may ensure that the two $\alpha$-arcs of $M_{\tau T}$ that have ends on the same suture do not intersect any $\beta$-curve in a Heegaard diagram of $M_{\tau T}$. Moreover, since generators of $\BSAD(X)$ which occupy both $\textcolor{red}{e}$ and $\textcolor{red}{f}$ cannot occupy $\textcolor{red}{E}$ nor $\textcolor{red}{F}$, we see that $\mathcal{F}^{AD}_{\iota_{43},\id}\left(\BSAD(X)\right)$ lies in the image of the functor $\mathcal{F}^{AD}_{\id,\iota_{34}}$, say 
	$$
	\mathcal{F}^{AD}_{\id,\iota_{34}}(\mathcal{D}')=\mathcal{F}^{AD}_{\iota_{43},\id}\left(\BSAD(X)\right).
	$$
	So, 
	$$
	\BSD(M_{\tau T})^{\mathcal{B}'_{34}}\cong\BSD(M_T)^{\mathcal{B}'_{43}}\boxtimes \typeA{\mathcal{B}'_{43}}{\left(\mathcal{D}'\right)}^{\mathcal{B}'_{34}}.
	$$
	Next, we can pass from $\mathcal{B}'_{34}$ to $\Ad_{34}$ via the quotient map $\pi'_{34}$, which induces the functors $\mathcal{F}^D_{\pi'_{34}}$ and $\mathcal{F}^{AD}_{\id,\pi'_{34}}$. As in the proof of the Glueing Theorem~\cite[Theorem~3.7]{PQMod}, we may identify $\mathcal{F}^D_{\pi'_{34}}(\BSD(M_{\tau T}))$ with $\mathcal{F}_{34}(\CFTd(\tau T))$, so we obtain
	$$\mathcal{F}_{34}(\CFTd(\tau T))^{\Ad_{34}}\cong \BSD(M_T)^{\mathcal{B}'_{43}}\boxtimes \typeA{\mathcal{B}'_{43}}{\mathcal{F}^{AD}_{\id,\pi'_{34}}(\mathcal{D}')}^{\Ad_{34}}.$$
	Now, the only domains that can contribute to $\mathcal{F}^{AD}_{\id,\pi'_{34}}(\mathcal{D}')$ are those computed above in Figure~\ref{fig:DehnTwistDomains}. Moreover, from the $A_\infty$-relations, we deduce that the domain $\{\textcolor{blue}{ q_1}/\textcolor{red}{ Q_1}\}$ contributes iff the domains $\{\textcolor{blue}{ q_4},\textcolor{blue}{ q_1}/\textcolor{red}{ Q_1}\}$ and $\{\textcolor{blue}{ q_4},\textcolor{blue}{ q_1}/\textcolor{red}{ Q_1},\textcolor{blue}{ q_2}/\textcolor{red}{ Q_2}\}$ contribute; similarly, 
	the domain $\{\textcolor{blue}{ p_2}/\textcolor{red}{ P_2}\}$ contributes iff the domains $\{\textcolor{blue}{ p_3},\textcolor{blue}{ p_2}/\textcolor{red}{ P_2}\}$ and $\{\textcolor{blue}{ p_3},\textcolor{blue}{ p_2}/\textcolor{red}{ P_2},\textcolor{blue}{ p_1}/\textcolor{red}{ P_1}\}$ contribute. In any case, the type AD structure $\mathcal{F}^{AD}_{\id,\pi'_{34}}(\mathcal{D}')$ lies in the image of the functor $\mathcal{F}^{AD}_{\pi'_{43},\id}$, say 
	$$\mathcal{F}^{AD}_{\id,\pi'_{34}}(\mathcal{D}')=\mathcal{F}^{AD}_{\pi'_{43},\id}(\mathcal{D}'').$$
	Then, using the Pairing Adjunction and the fact that we can identify $\mathcal{F}^D_{\pi'_{43}}(\BSD(M_T))$ with $\mathcal{F}_{43}(\CFTd(T))$, we obtain 
	$$\mathcal{F}_{34}(\CFTd(\tau T))^{\Ad_{34}}=\mathcal{F}_{43}(\CFTd(T))^{\Ad_{43}}\boxtimes\typeA{\Ad_{43}}{\left(\mathcal{D}''\right)}^{\Ad_{34}}.$$
	To identify $\mathcal{D}''$ with $\mathcal{D}(\tau)$, it remains to show that the two domains $\{\textcolor{blue}{ q_1}/\textcolor{red}{ Q_1}\}$ and $\{\textcolor{blue}{ p_2}/\textcolor{red}{ P_2}\}$ contribute to the differential. Instead of analysing their moduli spaces, we argue indirectly. If the first domain does not contribute, $\mathcal{F}_{34}(\CFTd(\tau T))$ does not contain any arrows labelled $\textcolor{red}{q_1}$ for any tangle $T$. This is clearly false, so the domain must contribute. We can argue similarly for the other domain. 
	
	Finally, the graded statement follows from the additivity of gradings as in the proof of the Glueing Theorem for $\CFTd$, together with the fact that the meridional suture around the tangle end $\texttt{3}$ of $T$ becomes the meridional suture around the tangle end $\texttt{4}$ of $\tau T$ and vice versa. 
\end{proof}

%% file: sections/Linearity.tex
\section{Linearity}\label{sec:Linearity}

\begin{definition}
	Consider the covering space $\FourPuncturedTorus\rightarrow \FourPuncturedSphere$ corresponding to the kernel of the homomorphism 
	$\pi_1(\FourPuncturedSphere)\rightarrow \mathbb{Z}/2$
	which sends the circle around each puncture to 1. We can also think of this covering space as the restriction of the double-branched cover $T^2\rightarrow S^2$ with four branched points. Let $$\eta\co\PuncturedPlane\longrightarrow\FourPuncturedTorus\longrightarrow\FourPuncturedSphere$$
	be the composition of this map with the restriction of the universal cover of the torus $T^2=\mathbb{R}^2/(2\mathbb{Z})^2$. Our usual parametrization of $\FourPuncturedSphere$ lifts to a parametrization of $\PuncturedPlane$ under $\eta$. This is illustrated in Figure~\ref{fig:coveringmapINTRO} on page~\pageref{fig:coveringmapINTRO}, where the front face of $\FourPuncturedSphere$ and its preimage under $\eta$ are shaded gray. 
\end{definition}

\begin{figure}[t]
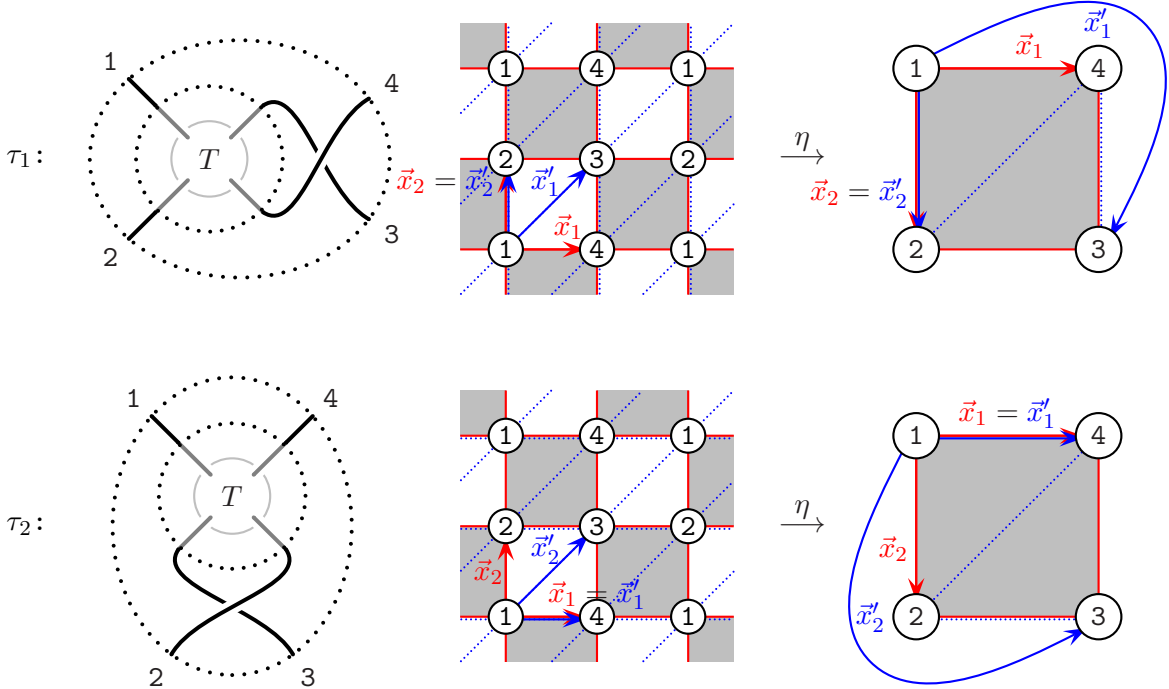

	\centering
	\TwistPSLRight
	\\
	\TwistPSLBottom
	\caption{Illustration of Observation~\ref{obs:TwistPSL}. The twists $\tau_1$ and $\tau_2$ correspond to the generating linear transformations of 
	$\SL(2,\mathbb{Z})$, which, in each case, send the basis
	$\{\textcolor{red}{\vec{x}_1},\textcolor{red}{\vec{x}_2}\}$ of $\mathbb{R}^2$ to $\{\textcolor{blue}{\vec{x}'_1},\textcolor{blue}{\vec{x}'_2}\}$.
	}\label{fig:TwistPSL}
\end{figure}

\begin{observation}\label{obs:TwistPSL}
$\PuncturedPlane$ comes with a natural action by $\SL(2,\mathbb{Z})$. The two generators 
$$
\begin{bmatrix}
1 & 0 \\
1 & 1
\end{bmatrix}
\quad\text{ and }\quad
\begin{bmatrix}
1 & 1 \\
0 & 1
\end{bmatrix}
$$
of $\SL(2,\mathbb{Z})$ induce two half-twists $\tau_1$ and $\tau_2$ on $\FourPuncturedSphere$, respectively, as illustrated in Figure~\ref{fig:TwistPSL}. Therefore, the action of $\SL(2,\mathbb{Z})$ on $\PuncturedPlane$ induces a well-defined action on $\FourPuncturedSphere$. Moreover, $-\id\in\SL(2,\mathbb{Z})$ acts trivially on $\FourPuncturedSphere$, so we obtain an action of
$\PSL(2,\mathbb{Z})
$ on $\FourPuncturedSphere$.

Alternatively, we can regard the half-twists $\tau_1$ and $\tau_2$ as the actions of the two generators of the 3-stranded braid group $\Braidgroup$ fixing one of the four punctures. (In Figure~\ref{fig:TwistPSL}, this is the puncture corresponding to the first tangle end.) So under this identification of generators, the action of $\PSL(2,\mathbb{Z})$ on $\FourPuncturedSphere$ agrees with the action of $\Braidgroup$. In fact, $\PSL(2,\mathbb{Z})$ is isomorphic to the braid group $\Braidgroup$ modulo its centre, which is generated by a full twist of all three strands. 
In short: if we fix one of the four punctures of $\FourPuncturedSphere$, a reparametrization of $\FourPuncturedSphere$ up to homotopy corresponds to an element in $\PSL(2,\mathbb{Z})$. 
\end{observation}

\begin{definition}
	As in~\cite[Section~7.1]{HRW}, it is useful to introduce a ``normal form'' for curves in $\PuncturedPlane$. For this, let us consider the standard metric on $\PuncturedPlane$, which induces a metric on $\FourPuncturedTorus$. Let us also fix some $\varepsilon$ with $\frac{1}{2}>\varepsilon>0$. We then define a \textbf{peg-board representative} of a closed curve $\gamma$ in $\FourPuncturedTorus$ as a representative of the homotopy class of $\gamma$ which has minimal length among all representatives of distance $\varepsilon$ to all four punctures in $\FourPuncturedTorus$. The intuition behind this is to think of the four punctures of $\FourPuncturedTorus$ as discs of radii $\varepsilon$ (or, if we think of the punctures of $\FourPuncturedTorus$ as marked points, pegs centred at those marked points) and then we imagine pulling the curve $\gamma$ ``tight'', like a rubber band. 
	It is easy to see that a peg-board representative of the lift of some curve in $\FourPuncturedSphere$ is the lift of another curve in $\FourPuncturedSphere$ which is homotopic to the original one.
	We also call lifts of such curves to $\PuncturedPlane$ peg-board representatives, and in fact, this is how we usually think of them. In this case, the pegs of radii $\varepsilon$ sit at the lattice points. If we take the limit $\varepsilon\rightarrow0$, a peg-board representative in $\PuncturedPlane$ becomes a piecewise linear curve, which we call a limit peg-board representative. A curve \(\gamma\) is called \textbf{linear}, if its peg-board representative does not wrap around a lattice point and if its limit peg-board representative is an infinite straight line. If its limit peg-board representative can be chosen not to intersect the lattice $\Lattice$, we call $\gamma$ \textbf{loose} and otherwise \textbf{rigid}. 
	We call a curve with local system  $(\gamma,X)$ linear, loose or rigid if the lift of $\gamma$ to $\PuncturedPlane$ is. 
\end{definition}

\begin{theorem}[Linearity]\label{thm:linearCurves}
	The underlying curve of every component of \(\HFT(T)\) for a 4-ended tangle \(T\) in a \(\mathbb{Z}\)-homology 3-ball is linear. 
\end{theorem}

\begin{lemma}\label{lem:nowrapping}
	For any 4-ended tangle \(T\), \(\Pi(\HFT(T))\) does not contain two consecutive arrows labelled $p_i$ and $q_i$, ie neither
	$$
	\begin{tikzcd}
	\bullet
	\arrow{r}{p_i}
	&
	\bullet
	\arrow{r}{q_i}
	&
	\bullet
	\end{tikzcd}
	\quad\text{nor}\quad 
	\begin{tikzcd}
	\bullet
	\arrow{r}{q_i}
	&
	\bullet
	\arrow{r}{p_i}
	&
	\bullet
	\end{tikzcd}.
	$$
\end{lemma}

\begin{wrapfigure}{r}{0.2\textwidth}
	\centering
	\wrappingi
	\caption{}\label{fig:wrapping}\vspace{15pt}
\end{wrapfigure}
\myfixwrapfig

\begin{proof}[Proof of Lemma~\ref{lem:nowrapping}]
	The case in which the left and right generators  coincide can be ruled out using Observation~\ref{obs:AlexGradingOfCFTdLoops}.
	So let us assume that these two generators are distinct, which is illustrated in Figure~\ref{fig:wrapping}. Without loss of generality, we may assume that the local system of $\Pi(\HFT(T))$ sits on the second elementary curve segment, the one corresponding to the arrow labelled $q_i$ in the first case and $p_i$ in the second. In particular, this means that the arrow leaving the left generator is the only one labelled $p_i$ in the first case and $q_i$ in the second. By Proposition~\ref{prop:arrow-pushing-for-CFTminus}, we can extend the curved type D structure $\Pi(\HFT(T))$ over $\Ad$ to a curved type D structure $\CFTminus(T)$ over $\Aminus$ by adding extra arrows labelled by elements in $\Aminus$ which are mapped to zero in $\Ad$. In $\CFTminus(T)$, the two consecutive arrows contribute a term $q_ip_i$ or $p_iq_i$ to the $d^2$-relations, which needs to be cancelled. Since there are no identity components in the differential, there must be another pair of consecutive arrows between the two generators, identical to the first, but via a fourth generator. However, this arrow pair would already have to be present in $\Pi(\HFT(T))$. In particular, there would have to be another arrow labelled $p_i$ in the first case and $q_i$ in the second leaving the left generator. Contradiction.
\end{proof}

\begin{figure}[t]
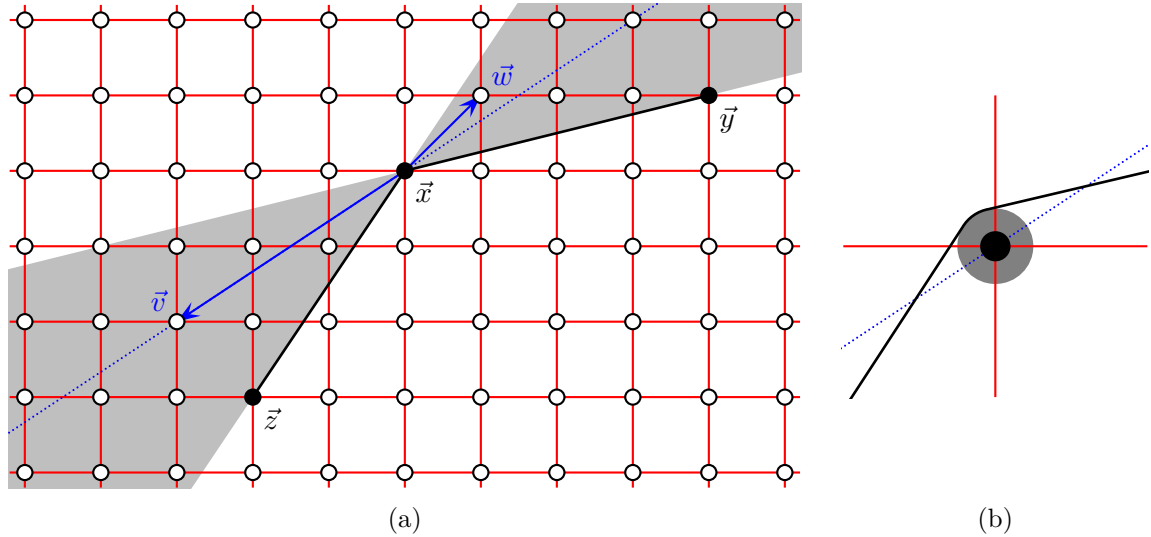

	\centering
	\begin{subfigure}[b]{0.7\textwidth}
		\centering
		\LinearityOverview
		\caption{}\label{fig:proof_of_linearity:overview}
	\end{subfigure}
	\begin{subfigure}[b]{0.26\textwidth}
		\centering
		\LinearityCloseup
		\\
		\vspace*{1.2cm}
		\caption{}\label{fig:proof_of_linearity:closeup}
	\end{subfigure}
	\caption{An illustration of how we choose the new basis $\{\textcolor{blue}{\vec{v}},\textcolor{blue}{\vec{w}}\}$ for the reparametrization in the proof of Lemma~\ref{lem:nowrapping} (a) and what this new parametrization and the curve $\gamma$ look like in a neighbourhood of the lattice point $\vec{x}$ (b)}\label{fig:proof_of_linearity}
\end{figure} 

\begin{proof}[Proof of Theorem~\ref{thm:linearCurves}]
	Suppose the lift $\gamma$ of the underlying curve of a component of $\HFT(T)$ is not linear. The limit peg-board representative of $\gamma$ does not wrap around a lattice point, because this would immediately contradict Lemma~\ref{lem:nowrapping}. Hence, there exist three non-collinear lattice points $\vec{x}$, $\vec{y}$ and $\vec{z}$ on the limit peg-board representative of $\gamma$. The main idea is to choose a reparametrization of $\FourPuncturedSphere$ such that the lift of one of the parametrizing arcs is a straight line through $\vec{x}$ which lies in the shaded region in Figure~\ref{fig:proof_of_linearity:overview}. The blue dotted line is an example of such a straight line. Then, any peg-board representative of $\gamma$ intersects this straight line near $\vec{x}$ as shown in Figure~\ref{fig:proof_of_linearity:closeup}. However, this configuration is ruled out by Lemma~\ref{lem:nowrapping}.  
	
	To find a suitable reparametrization, let us choose $\vec{x}$ as the origin of $\Plane$. In other words, the reparametrization should fix the puncture of $\FourPuncturedSphere$ corresponding to $\vec{x}$. Let $\vec{v}=(v_1,v_2)^t$ be the vector in $\Lattice\subset\Plane$ which is shortest among all such vectors satisfying $n\cdot \vec{v}=\vec{y}-\vec{z}$ for some positive integer $n$. By minimality, we can use Euclid's algorithm to find integers $w_1$ and $w_2$ such that $v_1w_2-v_2w_1=1$. This gives us an element
	$$
	\pm
	\begin{bmatrix}
	v_1 & w_1 \\
	v_2 & w_2
	\end{bmatrix}
	\in\PSL(2,\mathbb{Z})
	$$
	which defines a reparametrization of $\FourPuncturedSphere$ with the desired property. 
\end{proof}

\begin{figure}[t]
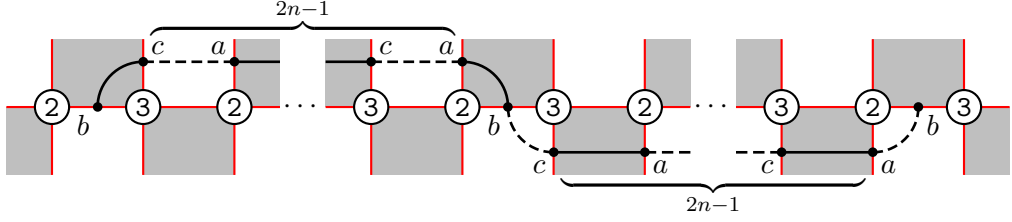

	\centering
	\rigidcurve
	\caption{The lift of the rigid curve $\mathfrak{b}_n$ from Theorem~\ref{thm:rigid_curves}}\label{fig:rigid_curves}
\end{figure}

\begin{figure}[t]
	\centering
	\begin{subfigure}[b]{\textwidth}
		\centering
		$$
		\begin{tikzcd}
		&
		a
		\arrow[dashed,bend right=20,swap]{r}{q_{32}}
		\arrow[dashed,bend left=20,leftarrow]{r}{q_{14}}
		&
		c
		\arrow[bend right=20,swap]{r}{p_{23}}
		\arrow[bend left=20,leftarrow]{r}{p_{41}}
		&
		a
		\arrow[phantom,pos=0.05]{rr}{\cdots}
		\arrow[phantom,pos=0.95]{rr}{\cdots}
		&
		&
		a
		\arrow[dashed,bend right=20,swap]{r}{q_{32}}
		\arrow[dashed,bend left=20,leftarrow]{r}{q_{14}}
		&
		c
		\arrow[bend right=20,swap]{r}{p_{3}}
		\arrow[bend left=20,leftarrow]{r}{p_{412}}
		&
		b
		\arrow[phantom,pos=0.9]{r}{\cdots}
		&
		c
		\arrow[bend right=20,swap]{r}{p_{23}}
		\arrow[bend left=20,leftarrow]{r}{p_{41}}
		&
		a
		\arrow[dashed,bend right=20,swap]{r}{q_{2}}
		\arrow[dashed,bend left=20,leftarrow]{r}{q_{143}}
		&
		b
		\\
		b
		\arrow[dashed,bend right=15,swap]{dr}{q_{3}}
		\arrow[dashed,bend left=15,leftarrow]{dr}{q_{214}}
		\arrow[bend right=15,swap]{ur}{p_{2}}
		\arrow[bend left=15,leftarrow]{ur}{p_{341}}
		\\
		&
		c
		\arrow[bend right=20,swap]{r}{p_{23}}
		\arrow[bend left=20,leftarrow]{r}{p_{41}}
		\arrow[dotted,pos=0.55,swap]{uur}{U_2.\iota_c}
		&
		a
		\arrow[dashed,bend right=20,swap]{r}{q_{32}}
		\arrow[dashed,bend left=20,leftarrow]{r}{q_{14}}
		\arrow[dotted,pos=0.55,swap]{uur}{U_2.\iota_a}
		&
		c
		\arrow[phantom,pos=0.05]{rr}{\cdots}
		\arrow[phantom,pos=0.95]{rr}{\cdots}
		&&
		c
		\arrow[bend right=20,swap]{r}{p_{23}}
		\arrow[bend left=20,leftarrow]{r}{p_{41}}
		\arrow[dotted,pos=0.55,swap]{uur}{U_2.\iota_c}
		&
		a
		\arrow[dotted,pos=0.55,swap]{uur}{q_2}
		&
		\phantom{b}
		\arrow[phantom,pos=0.9]{r}{\cdots}
		&
		a
		\arrow[dashed,bend right=20,swap]{r}{q_{32}}
		\arrow[dashed,bend left=20,leftarrow]{r}{q_{14}}
		\arrow[dotted,pos=0.55,swap]{uur}{U_2.\iota_a}
		&
		c
		\end{tikzcd}
		$$
		\caption{}\label{fig:rigid_curves:curveB:Graph}
	\end{subfigure}
	\bigskip\\
	\begin{subfigure}[b]{\textwidth}
		\centering
		\begin{tabular}{ccccc}
			$q_{32}p_2$
			&
			$p_{23}U_{2}$
			&
			$q_{32}U_{2}$
			&
			$p_{3}U_{2}$
			&
			$q_{2}U_{2}$
			\\
			\midrule
			$
			\begin{tikzcd}[row sep=0.2cm,column sep=0.5cm,ampersand replacement = \&]
			b
			\arrow{r}{p_{2}}
			\&
			a
			\arrow[dashed]{r}{q_{32}}
			\&
			c
			\\
			b
			\arrow[dotted]{r}{U_2.\iota_b}
			\&
			b
			\arrow[dashed]{r}{q_{3}}
			\&
			c
			\\
			b
			\arrow[dashed]{r}{q_{3}}
			\&
			c
			\arrow[dotted]{r}{U_2.\iota_c}
			\&
			c
			\end{tikzcd}
			$
			&
			$$
			\begin{tikzcd}[row sep=0.2cm,column sep=0.5cm,ampersand replacement = \&]
			c
			\arrow[dotted]{r}{U_2.\iota_c}
			\&
			c
			\arrow{r}{p_{23}}
			\&
			a
			\\
			c
			\arrow[dotted]{r}{U_2p_3}
			\&
			b
			\arrow{r}{p_{2}}
			\&
			a
			\\
			c
			\arrow{r}{p_3}
			\&
			b
			\arrow[dotted]{r}{U_2p_{2}}
			\&
			a
			\\
			c
			\arrow{r}{p_{23}}
			\&
			a
			\arrow[dotted]{r}{U_2.\iota_a}
			\&
			a
			\end{tikzcd}
			$$
			&
			$$
			\begin{tikzcd}[row sep=0.2cm,column sep=0.5cm,ampersand replacement = \&]
			a
			\arrow[dotted]{r}{U_2.\iota_a}
			\&
			a
			\arrow[dashed]{r}{q_{32}}
			\&
			c
			\\
			a
			\arrow[dotted]{r}{U_2q_2}
			\&
			b
			\arrow[dashed]{r}{q_{3}}
			\&
			c
			\\
			a
			\arrow[dashed]{r}{q_2}
			\&
			b
			\arrow[dotted]{r}{U_2q_{3}}
			\&
			c
			\\
			a
			\arrow[dashed]{r}{q_{23}}
			\&
			c
			\arrow[dotted]{r}{U_2.\iota_c}
			\&
			c
			\end{tikzcd}
			$$
			&
			$$
			\begin{tikzcd}[row sep=0.2cm,column sep=0.5cm,ampersand replacement = \&]
			c
			\arrow[dotted]{r}{U_2.\iota_c}
			\&
			c
			\arrow{r}{p_{3}}
			\&
			b
			\\
			c
			\arrow{r}{p_3}
			\&
			b
			\arrow[dotted]{r}{U_2.\iota_b}
			\&
			b
			\\
			c
			\arrow{r}{p_{23}}
			\&
			a
			\arrow[dashed]{r}{q_{2}}
			\&
			b
			\end{tikzcd}
			$$
			&
			$$
			\begin{tikzcd}[row sep=0.2cm,column sep=0.5cm,ampersand replacement = \&]
			a
			\arrow[dotted]{r}{U_2.\iota_a}
			\&
			a
			\arrow[dashed]{r}{q_{2}}
			\&
			b
			\\
			a
			\arrow[dashed]{r}{q_2}
			\&
			b
			\arrow[dotted]{r}{U_2.\iota_b}
			\&
			b
			\end{tikzcd}
			$$
		\end{tabular}
		\caption{}\label{fig:rigid_curves:curveB:Splittings}
	\end{subfigure}
	\caption{The (generalized) peculiar module considered in the proof of Theorem~\ref{thm:rigid_curves} (a) and splittings needed to determine the existence of the dotted arrows (b). Each column in (b) shows all possible ways to write the algebra element at the top as a product of two non-identity elements of $\Aminus$. In each column, the first splitting is the one which appears in~(a).}\label{fig:rigid_curves:curveB}
\end{figure}

\begin{theorem}\label{thm:rigid_curves}
	Up to reparametrization, the underlying curve of any rigid component of \(\HFT(T)\) for a 4-ended tangle \(T\) is equal to \(\mathfrak{b}_n\coloneqq\Irr_n(\frac{0}{1};2,3)\) for some \(n\). Moreover, the local system on any such component is equal to the identity matrix. 
\end{theorem}

\begin{proof}
	By reparametrizing the tangle boundary, we may assume without loss of generality that the limit peg-board representative of the curve is equal to the horizontal line which agrees with the lift of the site $b$, as for the curve in Figure~\ref{fig:rigid_curves}. As we follow a peg-board representative of the curve from left to right, it intersects the site $b$ in at least two places, namely once as it goes from below the horizontal line to above and once as it goes in the opposite direction. Moreover, from the perspective of the first of these two points, the curve turns right in both adjacent faces, whereas from the perspective of the second intersection point, the curve turns left in both directions. From this, we see that one of these intersection points is connected to the site $a$ via the front face. We may assume without loss of generality that there are at most as many horizontal line segments starting at this intersection point with the site $a$ as there are on the other side of $b$, since a single half-twist along the site $b$ interchanges these two cases. 
	
	Let us assume for a moment that the curve carries the 1-dimensional (trivial) local system. The diagram on the left-hand side of Figure~\ref{fig:rigid_curves:curveB:Graph} (without the dotted arrows) shows the corresponding peculiar module. It contains two consecutive arrows 
	\begin{equation*}
	\begin{tikzcd}
	b
	\arrow{r}{p_{2}}
	&
	a
	\arrow[dashed]{r}{q_{32}}
	&
	c,
	\end{tikzcd}
	\end{equation*}
	which compose to zero in $\Pi(\HFT(T))$. However, by Proposition~\ref{prop:arrow-pushing-for-CFTminus}, we can extend $\Pi(\HFT(T))$ to $\CFTminus(T)$ by adding components to the differential which lie in the kernel of $\Aminus\rightarrow\Ad$. In $\CFTminus(T)$, the composition of the two arrows above is non-zero. We claim that its contribution to the $\partial^2$-relation can only be cancelled by a differential 
	\begin{equation*}
	\begin{tikzcd}
	c
	\arrow[dotted]{r}{U_2.\iota_c}
	&
	c.
	\end{tikzcd}
	\end{equation*}
	To see this, we may argue as follows: since $\Pi(\HFT(T))$ is reduced, it suffices to consider all possible ways to write $q_{32}p_2$ as a product of two non-idempotent elements. These are shown in the first column of Figure~\ref{fig:rigid_curves:curveB:Splittings}. For the first such splitting, both arrows must already be present in $\Pi(\HFT(T))$. But there is only one such arrow pair and it is exactly the one whose contribution to the $\partial^2$-relation we are trying to cancel. In the second case, there has to be an arrow labelled $q_3$ ending at the generator $c$ which would have to appear in $\Pi(\HFT(T))$, but it does not. So only the third case remains, and indeed, there is just a single arrow labelled $q_3$ leaving the generator $b$. 
	
	We now distinguish two cases: either the number of horizontal curve segments starting at the generator $a$ adjacent to the generator $b$ is odd or even. These two cases are illustrated in the middle and the right of Figure~\ref{fig:rigid_curves:curveB:Graph}, respectively. In both cases, all the dotted arrows labelled $U_2.\iota_c$ and $U_2.\iota_a$ have to be present in $\CFTminus(T)$. For this, we can repeat essentially the same argument as above, using the splittings in the second and third columns of Figure~\ref{fig:rigid_curves:curveB:Splittings}. Then, in the first case, there is a pair of consecutive arrows 
	\begin{equation*}
	\begin{tikzcd}
	c
	\arrow[dotted]{r}{U_2.\iota_c}
	&
	c
	\arrow{r}{p_3}
	&
	b.
	\end{tikzcd}
	\end{equation*} 
	By considering the corresponding splittings in the fourth column of Figure~\ref{fig:rigid_curves:curveB:Splittings}, we see that there has to be a component 
	\begin{equation*}
	\begin{tikzcd}
	a
	\arrow[dashed]{r}{q_2}
	&
	b
	\end{tikzcd}
	\end{equation*}
	in $\CFTminus(T)$, which, of course, already has to be there in $\Pi(\HFT(T))$. But this means that the curve is equal to $\mathfrak{b}_n$, where $n$ is determined by the number of horizontal curve segments. In the second case, there is a pair of consecutive arrows 
	\begin{equation*}
	\begin{tikzcd}
	a
	\arrow[dotted]{r}{U_2.\iota_a}
	&
	a
	\arrow[dashed]{r}{q_2}
	&
	b.
	\end{tikzcd}
	\end{equation*} 
	As we can see from the corresponding splittings in the last column of Figure~\ref{fig:rigid_curves:curveB:Splittings}, this contribution to the $\partial^2$-relation for the differential of $\CFTminus(T)$ cannot be cancelled, so we arrive at a contradiction. 
	
	Let us now consider local systems. We can assume without loss of generality that the local system $X$ of the curve sits on the curve segment that connects the generator $b$ to $a$. Algebraically, this corresponds to tensoring the peculiar modules from Figure~\ref{fig:rigid_curves:curveB:Graph} by $\field^m$ for $m=\dim X$ and changing the resulting differentials 
	\begin{equation*}
	\begin{tikzcd}[column sep=2cm]
	b\otimes \field^m
	\arrow{r}{p_2\otimes\id_{\field^m}}
	&
	a\otimes \field^m
	\end{tikzcd}
	\quad\text{ and }\quad
	\begin{tikzcd}[column sep=2cm]
	a\otimes 
	\field^m
	\arrow{r}{p_{341}\otimes\id_{\field^m}}
	&
	b\otimes \field^m
	\end{tikzcd}
	\end{equation*} 
	to
	\begin{equation*}
	\begin{tikzcd}[column sep=2cm]
	b\otimes \field^m
	\arrow{r}{p_2\otimes X}
	&
	a\otimes \field^m
	\end{tikzcd}
	\quad\text{ and }\quad
	\begin{tikzcd}[column sep=2cm]
	a\otimes 
	\field^m
	\arrow{r}{p_{341}\otimes X^{-1}}
	&
	b\otimes \field^m,
	\end{tikzcd}
	\end{equation*} 
	respectively. Then, by the same argument as above, the dotted arrows from Figure~\ref{fig:rigid_curves:curveB:Graph} still have to be present in $\CFTminus(T)$, but this time they need to be labelled by $U_2.\iota_c\otimes X$ and $U_2.\iota_a\otimes X$. In the case of an even number of straight line segments, we also arrive at a contradiction. In the case of an odd number of straight line segments, the dotted arrow labelled by $q_2\otimes X$ has to be present in $\CFTminus(T)$. The local system on the resulting curve is $X\cdot X^{-1}=\id$.
\end{proof}

\begin{Remark}
	It is relatively easy to construct bigraded extensions of peculiar modules corresponding to rational curves with non-trivial local systems. (In fact, for some fixed local systems, one can construct multiple such extensions which are non-homotopic.) This shows that the arguments from the proof of Theorem~\ref{thm:rigid_curves} cannot work for rational components. On the other hand, the following question remains open.
\end{Remark}	

\begin{question}\label{que:local_systems_for_rationals}
	Is there a 4-ended tangle \(T\) such that a rational component of \(\HFT(T)\) contains a non-trivial local system?
\end{question}

%% file: sections/Stabilization.tex

\section{Stabilization}\label{sec:Stabilization}

Let us recall from~\cite[Definition~2.4]{PQMod} the definition of a peculiar Heegaard diagram $\mathcal{H}_T$ for an oriented 4-ended tangle $T$ in a $\mathbb{Z}$-homology 3-sphere $M$. We start with a sutured Heegaard for the complement $M_T$ of a tubular neighbourhood of $T$ in $M$ equipped with a meridional suture at each of the four tangle ends and a pair of oppositely oriented meridional sutures on each closed component of $T$. We then collapse each suture to a single point and label those collapsed sutures that correspond to closed tangle components by $z_i$ and $w_i$. Furthermore, we connect the collapsed sutures around the tangle ends by an additional $\alpha$-circle which corresponds to the parametrization $\textcolor{red}{S^1}$ on $\partial M$. Finally, at each tangle end, we add marked points on either side of this $\alpha$-circle which we label by $p_i$ and $q_i$, see~\cite[Figures~10 and~11]{PQMod}.
For the definition of the differential of $\CFTd(T)\coloneqq\CFTd(\mathcal{H}_T)$, we only count holomorphic discs avoiding the basepoints $z_i$ and $w_i$, so these are treated exactly like sutures in sutured Floer homology.

We now give a mild generalization of this construction: 

\begin{definition}\label{def:stabilized_tangle}
	A \textbf{stabilized tangle} \(T(n_1,\dots,n_{|T|})\) in a \(\mathbb{Z}\)-homology 3-sphere \(M\) is an oriented 4-ended tangle \(T\) in \(M\) together with a list $(n_1,\dots,n_{|T|})$ of non-negative integers. Given such a stabilized tangle $T$, let $\mathcal{H}_T(n_1,\dots,n_{|T|})$ be the peculiar Heegaard diagram obtained through the same procedure as for an ordinary 4-ended tangle, but starting with a sutured Heegaard diagram for the complement $M_T$ equipped with $n_i$ \emph{additional} pairs of meridional sutures on the $i^\text{th}$ tangle component for each $i=1,\dots, |T|$. Thus, stabilized tangles $T(0,\dots,0)$ correspond to ordinary tangles $T$ in the sense that $\mathcal{H}_T(0,\dots,0)=\mathcal{H}_T$.  
	We define $\CFTd(\mathcal{H}_T(n_1,\dots,n_{|T|}))$ like $\CFTd(\mathcal{H}_T)$, treating the additional collapsed sutures in the same way as the basepoints $z_i$ and~$w_i$. Its relatively bigraded chain homotopy type is independent of the chosen Heegaard diagram, so we write instead \(\CFTd(T(n_1,\dots,n_{|T|}))\). 
\end{definition}

\begin{theorem}\label{thm:stabilization}
	Let \(T(n_1,\dots,n_{|T|})\) be a stabilized tangle. Then,
	$$
	\CFTd(T(n_1,\dots,n_{|T|}))
	\cong
	\bigotimes_{k=1}^{|T|} V_{t_k}^{\otimes n_k}\otimes\CFTd(T)
	$$
	where $V_t$ is a 2-dimensional vector space supported in a single relative $\delta$-grading and two consecutive relative Alexander gradings $t^{+1}$ and $t^{-1}$. 
\end{theorem}

Naturally, we can apply the classification result from Theorem~\ref{thm:classificationPecMod} to $\CFTd(T(n_1,\dots,n_{|T|}))$ and define $\HFT(T(n_1,\dots,n_{|T|}))$ as the corresponding multicurve. Then, the theorem above corresponds to Theorem~\ref{thm:intro:stabilization} from the introduction. 

\begin{figure}[t]
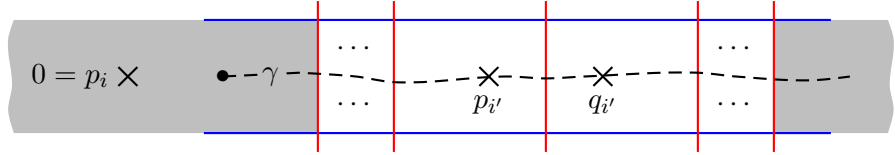
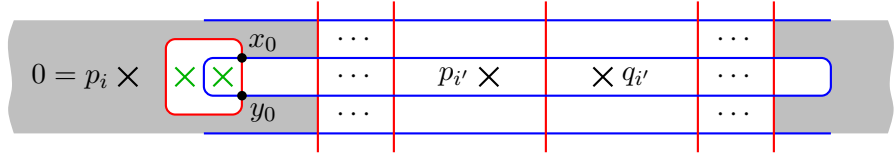

	\centering
	\begin{subfigure}[t]{0.9\textwidth}
		\centering
		\stabilizationfingerBefore
		\caption{A schematic picture of the regions that the path $\gamma$ passes through. Note that the labels $p_{i'}$ and $q_{i'}$ might be swapped and $\gamma$ might not intersect any $\alpha$-curve before or after it passes through $p_{i'}$ or $q_{i'}$. The shaded region on the left is the bad region~$\varphi$; the shaded region on the right is either a bigon or a bad region.}\label{fig:stabilizationfingerBefore}
	\end{subfigure}
	\\
	\begin{subfigure}[t]{0.9\textwidth}
		\centering
		\stabilizationfinger
		\caption{This picture illustrates how we change the Heegaard diagram~$\mathcal{H}_T$ in a neighbourhood of the path $\gamma$ from (a). The shaded region on the right is either a rectangle or a bad region.}\label{fig:stabilizationfingerAfter}
	\end{subfigure}
	\caption{Illustration of the proof of Theorem~\ref{thm:stabilization}}\label{fig:stabilizationfinger}
\end{figure}

\begin{proof}
	In the presence of two meridional sutures in the homology class $t$, adding another such pair corresponds to tensoring with the 2-dimensional vector space $V_t$. This well-known fact can be proven either directly using Heegaard diagrams (similar to the ones we are using below) or alternatively using Juhasz's surface decomposition formula~\cite[Proposition~8.6]{SurfaceDecomposition}. So it suffices to show the above theorem for the cases $0\leq n_1,n_2\leq1$ and $n_k=0$ for $k>2$. 
	
	By~\cite[Theorem~5.6]{PQMod}, we can always choose a peculiar Heegaard diagram $\mathcal{H}_T$ for $T$ which is nice with respect to any choice of special basepoints $p_i$ and $q_j$. Let us pick $p_i$ and $q_j$ such that the ends $i$ and $j$ belong to different components of the tangle $T$, so that each open component ``has a bad region''. 
	
	Let us consider the open tangle component $t_1$ and let us assume without loss of generality that it connects the tangle end $i$ with $i'\neq j$. Let $\varphi$ be the bad region containing the special basepoint $p_i$. The annulus which is the boundary of the removed tubular neighbourhood of the tangle component $t_1$ corresponds to an elementary periodic domain $\psi$ of $\mathcal{H}_T$, see~\cite[Definition~5.1 and Lemma~5.2]{HDsForTangles}. On said annulus, we can choose a path with connects $i$ to $i'$. By pushing it off the 2-handles attached along the $\beta$-circles, this path gives rise to a path $\gamma$ from the bad region $\varphi$ to  $p_{i'}$ which avoids all $\beta$-curves and also stays away from the other bad region (the one marked $q_j$). Let us choose a basepoint $\bullet\neq p_i$ in $\varphi$ as the starting point of $\gamma$. Let us extend the path $\gamma$ to the basepoint $q_{i'}$. We may assume without loss of generality that the path $\gamma$ does not ``back-track'' in the sense that $\gamma$ does not leave a region through the same part of the $\alpha$-curve through which it just entered the region. This can be achieved by ``pulling $\gamma$ tight'' and potentially changing the order in which $\gamma$ meets the basepoints $p_{i'}$ and $q_{i'}$. This is illustrated on the left-hand side of Figure~\ref{fig:stabilizationfingerBefore}.
	
	Note that the basepoints $p_k$ for $k\in\{1,2,3,4\}$ lie in the elementary periodic domain corresponding to the front component of $\partial M\smallsetminus\textcolor{red}{S^1}$. Similarly, the basepoints $q_{k'}$ for $k'\in\{1,2,3,4\}$ lie in the elementary periodic domain corresponding to the back component of $\partial M \smallsetminus\textcolor{red}{S^1}$. So in particular, $p_i$ and $q_{i'}$ do not lie in the same region. So we conclude that $\gamma$ leaves the bad region $\varphi$ at least once. Moreover, whenever the path returns to $\varphi$, we may replace the first segment of $\gamma$ up to this point by a direct path contained in $\varphi$ starting at $\bullet$. This strictly shortens the number of times the path $\gamma$ crosses an $\alpha$-curve, so after repeating this step finitely many times, we may assume that $\gamma$ does not return to~$\varphi$. Then $\gamma$ ends at a bigon or a rectangle region (the one marked $q_{i'}$ or $p_{i'}$) and all regions that it passes through are rectangles, each of which is cut into two rectangles by $\gamma$. If $\gamma$ ends in a  rectangle, let us extend $\gamma$ (by a path which still avoids $\beta$-curves and does not back-track) until it arrives at a region which is not a rectangle, ie a bigon or a bad region. Note that this process does terminate, since the path never crosses a rectangle twice and there are only finitely many regions. 
	
	So we have now constructed a path $\gamma$ which starts at $\bullet$ and ends either in a bad region or in a bigon. All other regions that $\gamma$ passes through are rectangles, each of which is cut into two rectangles by $\gamma$. We now modify the Heegaard diagram $\mathcal{H}_T$ in a neighbourhood of $\gamma$ as follows: we first replace $\bullet$ by an unlabelled basepoint $\textcolor{darkgreen}{\hm{\times}}$ which we enclose by a new $\beta$-circle $\beta_0$, which in turn is enclosed by a new $\alpha$-circle $\alpha_0$. We also add a second unlabelled basepoint $\textcolor{darkgreen}{\hm{\times}}$ in the region between $\beta_0$ and $\alpha_0$. Finally, we do a finger move of the $\beta$-curve $\beta_0$ along the path $\gamma$, moving across the basepoints $p_{i'}$ and $q_{i'}$. The result is shown in Figure~\ref{fig:stabilizationfingerAfter}. 
	
	We now observe that the new Heegaard diagram represents the tangle complement $M_T$ with an extra pair of meridional sutures around the tangle component $t_1$, ie it is a peculiar Heegaard diagram for $T(1,0,\dots,0)$. Moreover, it is a nice Heegaard diagram. Let us consider it more closely. The new $\alpha$-circle $\alpha_0$ must be occupied by any generator of the new diagram, so the new $\beta$-curve $\beta_0$ is always occupied by one of the two intersection points labelled $x_0$ or $y_0$ in Figure~\ref{fig:stabilizationfingerAfter}. Hence, there is a 2:1-correspondence between generators of the new diagram and the old, given by dropping the intersection points $x_0$ and $y_0$. The only bigon or rectangle domain that connects a generator containing $x_0$ to one containing $y_0$ (and avoiding the unlabelled basepoints $\textcolor{darkgreen}{\hm{\times}}$ and the bad region $\varphi$) is the bigon connecting $x_0$ to $y_0$. However, this bigon contains both $p_{i'}$ and $q_{i'}$, so it does not contribute to the differential. This bigon has $\delta$-grading 0 and Alexander grading $t_1^{\pm2}$. So the new Heegaard diagram indeed computes the once stabilized version of the complex $\CFTd(T)$. This finishes the case $(n_1,n_2)=(1,0)$. The case $(n_1,n_2)=(0,1)$ can be shown similarly. The third and last case $(n_1,n_2)=(1,1)$ also follows from the arguments above, since we can modify the Heegaard diagram $\mathcal{H}_T$ around both open tangle components independently. 
\end{proof}

%% file: sections/HalfIdentityBimodule.tex

\section{The action of the half-identity bimodule}\label{sec:HalfId}

In this section, we describe an alternative definition of peculiar modules, which was first proposed to the author by Andy Manion. It uses an arc diagram which is dual to the original one. The definition requires a choice of two distinguished tangle ends, so actually, we will obtain multiple, \textit{a prior} different algebraic objects. The advantage of these definitions is that they are simply instances of bordered sutured Heegaard Floer theory, so both invariance and the existence of a pairing theorem are immediate from Zarev's bordered sutured theory~\cite{ZarevThesis}. A potential disadvantage is that it is not immediately clear if the invariants thus defined are classified by \textit{compact} curves on the 4-punctured sphere, nor whether they are independent of the choice of distinguished tangle ends. However, we show that the alternative definitions are indeed equivalent to the original one, so we can answer both questions positively. 

The central ingredient in the identification is the bordered sutured Heegaard diagram in Figure~\ref{fig:HalfId1HD}, which is obtained by cutting the standard Heegaard diagram of the identity in half (see for example \cite[Figure~19]{LOTBimodules} and \cite[Figure~12]{LOTMor}). Hence, its corresponding bimodule is sometimes called the half-identity bimodule. 
In our case, this bimodule actually acts like the identity.


\begin{wrapfigure}{r}{0.3333\textwidth}
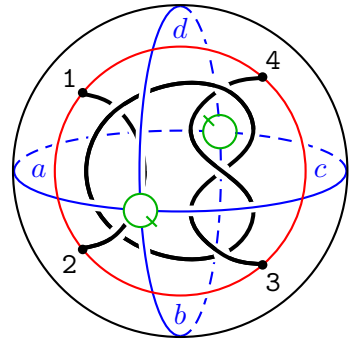

	\centering
	\pretzeltangleDUAL
	\caption{The bordered sutured manifold $M_T^\beta(p_3,q_1)$ for $T$ the $(2,-3)$-pretzel tangle. The parametrization given by the $\beta$-arcs (blue) is dual to the one from Figure~\ref{fig:2m3pt}.}\label{fig:2m3ptDUAL}
	\bigskip
\end{wrapfigure}
\myfixwrapfig

\begin{definition}\label{def:TangleComplementDualBSstructure}
	Given a 4-ended tangle \(T\) in a $\mathbb{Z}$-homology 3-ball $M$, choose an embedded disc on the front and one on the back of the sphere $\partial M$. Connect the boundaries of these two discs by four arcs which partition the complement of the two discs in $\partial M$ into four contractible components, each of which contains one tangle end. Let us call these arcs $\beta$-arcs.  We can arrange these in such a way that each of them intersects the circle $\textcolor{red}{S^1}$ on $\partial M$ exactly once. So we can label each $\beta$-arc by the same letter $a$, $b$, $c$ or $d$ as the arc in $\textcolor{red}{S^1}\smallsetminus \im(T)$ that the $\beta$-arc intersects. By removing a tubular neighbourhood of the tangle (chosen sufficiently small such that it avoids the two embedded discs and four $\beta$-arcs) and regarding the boundary of the two discs as sutures, we obtain a $\beta$-bordered sutured 3-manifold. To make the bordered sutured structure well-defined, we need to add a basepoint to each of the two sutures. Let us call the basepoint of the suture on the front \(p_i\) if it is adjacent to the tangle end $i$ in the sense that one can draw a path from the tangle end $i$ to the basepoint without crossing a $\beta$-arc or a suture. Similarly, call the basepoint of the suture on the back \(q_j\) if it is adjacent to the tangle end $j$. Moreover, if the two basepoints lie adjacent to the ends of the same open tangle strand, we also need to add a pair of meridional sutures to the other open tangle strand. Finally, to each closed tangle component, we add two meridional sutures as usual. Let us call the tangle complement $M_T$ equipped with this bordered sutured structure $M_T^{\beta}(p_i,q_j)$. Figure~\ref{fig:2m3ptDUAL} illustrates this definition for the tangle from Figure~\ref{fig:2m3pt}.
\end{definition}

\begin{Remark}
	The bordered sutured structure of $M_T^{\beta}(p_i,q_j)$ is degenerate, but it satisfies homological linear independence. Hence, Zarev's bordered sutured invariants are well-defined. We can define $\delta$-, homological and Alexander gradings on $\BSD(M_T^{\beta}(p_i,q_j))$ in just the same way as on $\CFTd(T)$. The type D structure $\BSD(M_T^{\beta}(p_i,q_j))$ is defined over a certain moving strands algebra. Since all $\beta$-arcs but one are occupied by a generator in a Heegaard diagram for $M_T^{\beta}(p_i,q_j)$, the moving strands algebra can be canonically identified with $\Ad_{ij}$ such that an elementary segment of the suture on the front (respectively back) corresponds to the variable $p_k$ (respectively $q_k$) if it is adjacent to the $k^\text{th}$ tangle end. 
\end{Remark}
	
\begin{theorem}\label{thm:HalfIdBimodAction}
	Let \(T\) be a 4-ended tangle and \((i,j)\) an ordered 2-element subset of \(\{1,2,3,4\}\). Then, if the ends $i$ and $j$ of \(T\) belong to different open components,
	\[\mathcal{F}_{ij}(\CFTd(T))\cong\BSD(M_T^{\beta}(p_i,q_j))\]
	as type D structures, up to bigrading preserving chain homotopy. If the ends $i$ and $j$ belong to the same open component and $t_2$ is the colour of the other, then
	\[\mathcal{F}_{ij}(\CFTd(T))\otimes V_{t_2}\cong\BSD(M_T^{\beta}(p_i,q_j)),\]
	where $V_{t_2}$ is as in Theorem~\ref{thm:stabilization}.
\end{theorem}

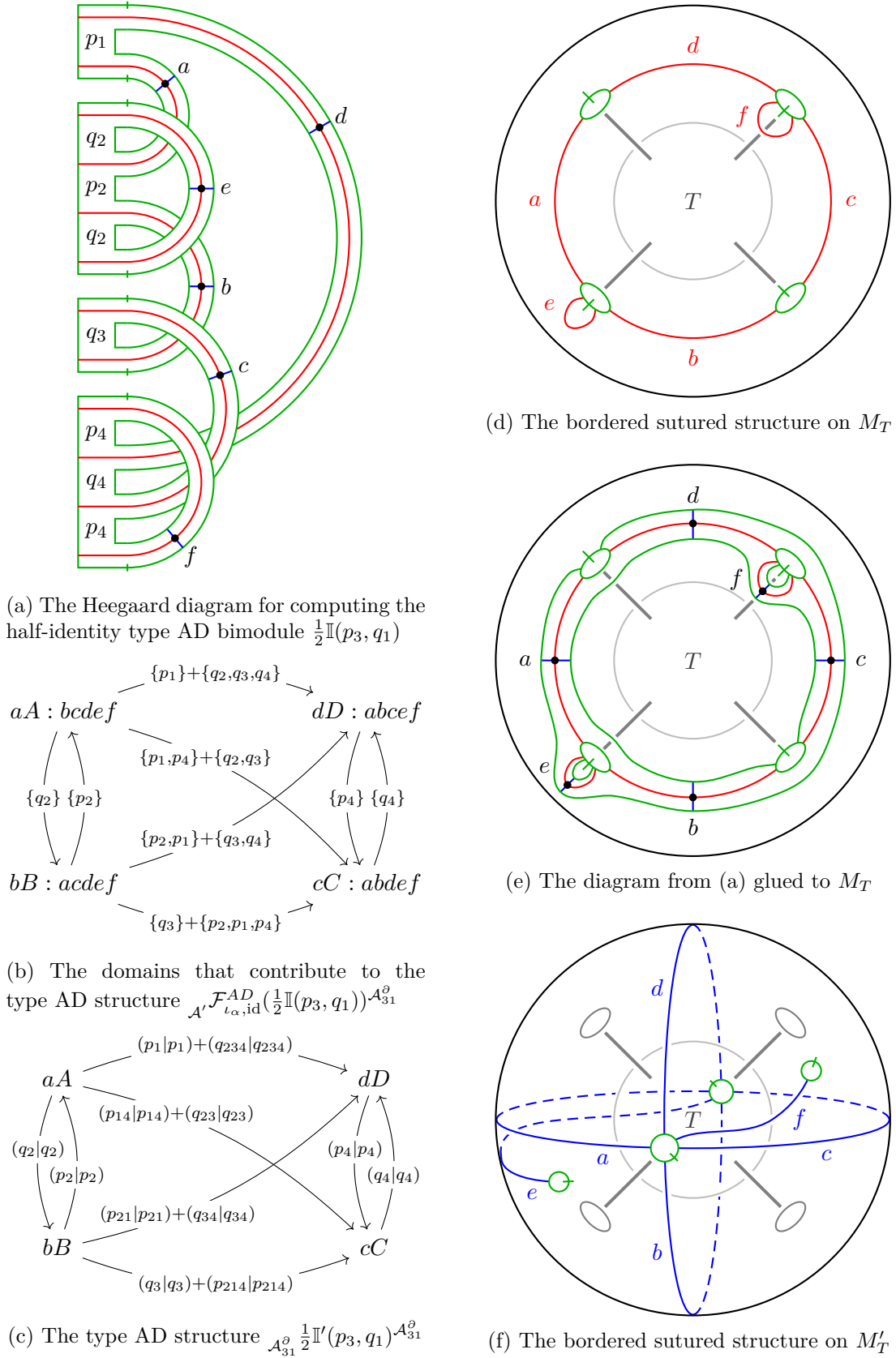
\begin{figure}[p]
	\centering
	\begin{minipage}[b]{0.45\textwidth}
	\begin{subfigure}[b]{\textwidth}\centering
		\HalfIdHD
		\caption{The Heegaard diagram for computing the half-identity type AD bimodule $\tfrac{1}{2}\mathbb{I}(p_3,q_1)$}\label{fig:HalfId1HD}
	\end{subfigure}
	\\
	\begin{subfigure}[b]{\textwidth}\centering
		$$
		\begin{tikzcd}[row sep=2.2cm, column sep=2.8cm]
		aA: bcdef
		\arrow[bend right=20]{d}[description]{\{q_2\}}
		\arrow[bend left=15,pos=0.3]{dr}[description]{\{p_1,p_4\}+\{q_2,q_3\}}
		\arrow[bend left=20]{r}[description]{\{p_1\}+\{q_2,q_3,q_4\}}
		&
		dD: abcef
		\arrow[bend right=20]{d}[description]{\{p_4\}}
		\\
		bB: acdef
		\arrow[bend right=20]{u}[description]{\{p_2\}}
		\arrow[bend right=15,pos=0.3]{ur}[description]{\{p_2,p_1\}+\{q_3,q_4\}}
		\arrow[bend right=20]{r}[description]{\{q_3\}+\{p_2,p_1,p_4\}}
		&
		cC: abdef
		\arrow[bend right=20]{u}[description]{\{q_4\}}
		\end{tikzcd}$$
		\caption{The domains that contribute to the type~AD structure $\typeA{\mathcal{A}'}{\mathcal{F}^{AD}_{\iota_\alpha,\id}(\tfrac{1}{2}\mathbb{I}(p_3,q_1))}^{\Ad_{31}}$}\label{fig:HalfId1Domains}
	\end{subfigure}
	\\
	\begin{subfigure}[b]{\textwidth}\centering
		$$
		\begin{tikzcd}[row sep=2.2cm, column sep=4.3cm]
		aA
		\arrow[bend right=20,pos=0.4]{d}[description]{(q_2\vert q_2)}
		\arrow[bend left=15,pos=0.3]{dr}[description]{(p_{14} \vert p_{14})+(q_{23} \vert q_{23})}
		\arrow[bend left=20]{r}[description]{(p_{1} \vert p_{1})+(q_{234} \vert q_{234})}
		&
		dD
		\arrow[bend right=20,pos=0.4]{d}[description]{(p_{4} \vert p_{4})}
		\\
		bB
		\arrow[bend right=20,pos=0.4]{u}[description]{(p_{2} \vert p_{2})}
		\arrow[bend right=15,pos=0.3]{ur}[description]{(p_{21} \vert p_{21})+(q_{34} \vert q_{34})}
		\arrow[bend right=20]{r}[description]{(q_{3} \vert q_{3})+(p_{214} \vert p_{214})}
		&
		cC
		\arrow[bend right=20,pos=0.4]{u}[description]{(q_{4} \vert q_{4})}
		\end{tikzcd}$$
		\caption{The type~AD structure $\typeA{\Ad_{31}}{\tfrac{1}{2}\mathbb{I}'(p_3,q_1)}^{\Ad_{31}}$}\label{fig:HalfId1AD}
	\end{subfigure}
	\end{minipage}
	\quad 
	\begin{minipage}[b]{0.45\textwidth}
		\begin{subfigure}[b]{\textwidth}\centering
			\HalfIdbefore
			\caption{The bordered sutured structure on $M_T$}\label{fig:HalfId1before}
		\end{subfigure}
		\smallskip\\
		\begin{subfigure}[b]{\textwidth}\centering
			\HalfIdWhile
			\caption{The diagram from (a) glued to $M_T$}\label{fig:HalfId1While}
		\end{subfigure}
		\smallskip\\
		\begin{subfigure}[b]{\textwidth}\centering
			\HalfIdAfter
			\caption{The bordered sutured structure on $M'_T$}\label{fig:HalfId1After}
		\end{subfigure}
	\end{minipage}
	\caption{Illustrations for the proof of Theorem~\ref{thm:HalfIdBimodAction} for the case $(i,j)=(3,1)$}\label{fig:HalfId1}
\end{figure}

\begin{proof}
	It suffices to consider four different cases for $(i,j)$, namely $(i,j)=(3,1)$, $(1,2)$, $(2,1)$ and $(1,1)$, since all other are obtained by cyclic permutation. Of those four cases, we will only consider the case $(i,j)=(3,1)$, since this is the one that we will work with in Section~\ref{sec:ConjugationBimodule}. The other three cases follow from the same arguments, only the computations in Figure~\ref{fig:HalfId1} need to be adapted to each case, which we leave to the reader.
	
	Let $M_T$ be the bordered sutured manifold whose underlying sutured manifold is the tangle complement with meridional sutures at each tangle end and a pair of meridional sutures on each closed tangle component and whose bordered sutured structure is given by the red arcs in Figure~\ref{fig:HalfId1before}. If the ends $i$ and $j$ belong to the same open component, we add an extra pair of meridional sutures to the other component. The result of glueing the bordered sutured diagram from Figure~\ref{fig:HalfId1HD} to $M_T$ is shown in Figure~\ref{fig:HalfId1While}. After omitting the $\alpha$-curves from the picture, we recognize it as the tangle complement with the bordered sutured structure from Figure~\ref{fig:HalfId1After}. Let us call this bordered sutured manifold $M'_T$. 
	
	Its corresponding bordered sutured algebra has six idempotents, since only one of the $\beta$-arcs is unoccupied. However, there are two sutures on $M'_T$ which are distinguished from the other two by the property that only a single $\beta$-arc ends at them, namely $\textcolor{blue}{e}$ and $\textcolor{blue}{f}$, respectively. We may choose a Heegaard digram for $M'_T$ such that for each of the two distinguished sutures, there is an $\alpha$-curve which is a push-off of the corresponding suture (=boundary of the Heegaard diagram). We can isotope these two $\alpha$-curves such that they have a unique intersection point with the $\beta$-curve ending at the corresponding distinguished suture and do not intersect any other $\beta$-curve. So we may assume that $\BSD(M'_T)$ only has generators which occupy both arcs $\textcolor{blue}{e}$ and $\textcolor{blue}{f}$. In other words, we may regard $\BSD(M'_T)$ as a type D structure over $\Ad_{31}$; we will add the superscript $\BSD(M'_T)^{\Ad_{31}}$ when we take this point of view. Note that $\BSD(M'_T)^{\Ad_{31}}$ is still well-defined up to homotopy, since any chain homotopy over the full bordered sutured algebra factors through $\Ad_{31}$.
		
	The bordered sutured manifold $M'_T$ looks already very similar to $M_T^\beta(p_3,q_1)$. In fact, 
	\[\BSD(M'_T)^{\Ad_{31}}\cong\BSD(M_T^\beta(p_3,q_1)).\] 
	We can see this as follows: let us go back to the Heegaard diagram that we chose above for $M'_T$. If the ends $i$ and $j$ belong to different components, we can do a sequence of handleslides which corresponds to pushing the two distinguished sutures all the way along the open tangle strands to the other tangle ends, to make sure that both sides of the two distinguished $\alpha$-circles are regions adjacent to basepoints. Otherwise, we can push them to the extra pair of meridional sutures with the same effect. By removing the two sutures (ie contracting these two boundary components of the Heegaard diagram) as well as the two corresponding $\alpha$-circles and $\beta$-arcs, we change neither the set of generators nor the domains nor the holomorphic curve count, so the claim follows. 
	
	Before we proceed, let us recall some notation from~\cite[proof of Theorem~3.7]{PQMod}. Let $\mathcal{A}$ be the bordered sutured algebra corresponding to the arc diagram on $M_T$ consisting of $\alpha$-arcs and $\mathcal{I}_\alpha$ the corresponding ring of idempotents. Let $\mathcal{I}'_\alpha$ be the subring of idempotents occupying the $\alpha$-arcs $\textcolor{red}{e}$ and $\textcolor{red}{f}$ and $\iota_{\alpha}:\mathcal{A}'\hookrightarrow\mathcal{A}$ the subalgebra $\mathcal{I}'_\alpha.\mathcal{A}.\mathcal{I}'_\alpha$ of~$\mathcal{A}$. There is also a quotient homomorphism $\pi_\alpha\co \mathcal{A}'\rightarrow\Ad_{31}$, which induces a functor $\mathcal{F}^D_{\pi_\alpha}$ between the corresponding categories of type D structures. As in the proof of the Glueing Theorem~\cite[Theorem~3.7]{PQMod}, 
	\[\mathcal{F}_{31}(\CFTd(T))\cong \mathcal{F}^D_{\pi_\alpha}(\BSD(M_T))\]
	unless there is an additional pair of sutures on $M_T$. In that case, the left-hand side needs to be stabilized once by Theorem~\ref{thm:stabilization}.
	In any case, it remains to show that 	
	\[\mathcal{F}^D_{\pi_\alpha}(\BSD(M_T))\cong \BSD(M'_T)^{\Ad_{31}}.\]
	Zarev's Pairing Theorem tells us that
	\[\BSD(M_T)\boxtimes\tfrac{1}{2}\mathbb{I}(p_3,q_1)\cong \BSD(M'_T),\]
	where $\tfrac{1}{2}\mathbb{I}(p_3,q_1)$ is the type AD structure computed from the Heegaard diagram in Figure~\ref{fig:HalfId1HD}. 
	By the Pairing Adjunction~\cite[Theorem~1.21]{PQMod}, the left hand side is equal to 
	\[\BSD(M_T)^{\mathcal{A}'}\boxtimes \typeA{\mathcal{A}'}{\mathcal{F}^{AD}_{\iota_\alpha,\id}(\tfrac{1}{2}\mathbb{I}(p_3,q_1))},\]
	where $\mathcal{F}^{AD}_{\iota_\alpha,\id}$ is the functor induced by the map $\iota_\alpha$ on the type A side and the identity on the type D side. From the Heegaard diagram in Figure~\ref{fig:HalfId1HD}, it is clear that we may regard $\mathcal{F}^{AD}_{\iota_\alpha,\id}(\tfrac{1}{2}\mathbb{I}(p_3,q_1))$ as a type AD structure over $\Ad_{31}$ on the type D side. So we obtain:
	\[\BSD(M_T)^{\mathcal{A}'}\boxtimes\typeA{\mathcal{A}'}{\mathcal{F}^{AD}_{\iota_\alpha,\id}(\tfrac{1}{2}\mathbb{I}(p_3,q_1))}^{\Ad_{31}}\cong \BSD(M'_T)^{\Ad_{31}}.\]
	Let us compute $\typeA{\mathcal{A}'}{\mathcal{F}^{AD}_{\iota_\alpha,\id}(\tfrac{1}{2}\mathbb{I}(p_3,q_1))}^{\Ad_{31}}$. All contributing domains are shown in Figure~\ref{fig:HalfId1Domains}. Let $\mathcal{F}^{AD}_{\pi_\alpha,\id}$ be the functor induced by the map $\pi_\alpha$ on the type A side and the identity on the type D side. We observe that $\typeA{\mathcal{A}'}{\mathcal{F}^{AD}_{\iota_\alpha,\id}(\tfrac{1}{2}\mathbb{I}(p_3,q_1))}^{\Ad_{31}}$ is in the image of this functor and equal to $\mathcal{F}^{AD}_{\pi_\alpha,\id}(\tfrac{1}{2}\mathbb{I}'(p_3,q_1)))$, where $\typeA{\Ad_{31}}{\tfrac{1}{2}\mathbb{I}'(p_3,q_1)}^{\Ad_{31}}$ is the type AD structure given in Figure~\ref{fig:HalfId1AD}. Then, another application of the Pairing Adjunction gives
	\[\mathcal{F}^D_{\pi_\alpha}(\BSD(M_T)^{\mathcal{A}'})^{\Ad_{31}}\boxtimes\typeA{\Ad_{31}}{\tfrac{1}{2}\mathbb{I}'(p_3,q_1)}^{\Ad_{31}}\cong \BSD(M'_T)^{\Ad_{31}}.\]
	Now observe that $\typeA{\Ad_{31}}{\tfrac{1}{2}\mathbb{I}'(p_3,q_1)}^{\Ad_{31}}$ is actually the identity type AD bimodule, so the claim follows. 
\end{proof}

%% file: sections/ConjugationBimodule.tex
\section{Conjugation}\label{sec:ConjugationBimodule}

In this section, we prove the conjugation symmetry results from the introduction and, as a corollary, mutation invariance of $\delta$-graded link Floer homology. Let us start by stating a slightly more precise version of  Theorem~\ref{thm:Conjugation:Horizontals:INTRO}.

\begin{theorem}[Bigraded conjugation symmetry for horizontal components]\label{thm:Conjugation:Horizontals}
	Given an oriented 4-ended tangle \(T\), let us fix an absolute lift of the \(\delta\)-grading of \(\HFT(T)\). For \(n>0\) and any local system \(X\), fix the Alexander grading on \(\mathfrak{b}_n\coloneqq \mathfrak{i}_n(\frac{0}{1};2,3)\), \(\mathfrak{d}_n\coloneqq \mathfrak{i}_n(\frac{0}{1};4,1)\) and \(\Rat_X\coloneqq \Rat_X(\frac{0}{1})\) as in Figure~\ref{fig:ConjugationBimodule:ImmersedCurve}, where $m=\dim X$. Let us assume that \(\HFT(T)\) contains a component of slope $\slopeZero$. Then there exists a unique lift of the Alexander grading on \(\HFT(T)\), independent of \(n\) and \(X\), such that
	$$\rr\Big(\Rat_X(T)\Big)= \Rat_{X^{-1}}(T)$$
	and:
	$$\rr\Big(\mathfrak{b}_n(T)\Big) = \mathfrak{d}_n(T)$$
	if the tangle ends \(\texttt{2}\) and \(\texttt{3}\) belong to different open tangle components and
	$$\rr\Big(\mathfrak{b}_n(T)\Big)\cdot(t_b+t_b^{-1})= \mathfrak{d}_n(T)\cdot(t_d+t_d^{-1})$$
	otherwise, where \(t_b\) and \(t_d\) are the colours of the tangle components at those tangle ends adjacent to the sites \(b\) and \(d\), respectively. 
\end{theorem}

\begin{figure}[b]
	\centering
	\begin{subfigure}[b]{\textwidth}
		\centering
		$$
		\begin{tikzcd}[row sep=1cm]
		\delta^{0}\Alex{0}{n-2}{n-1}{0}c
		&
		\delta^{0}\Alex{0}{n-1}{n}{0}a
		\arrow[dashed,swap]{l}{q_{32}}
		&
		\delta^{-\frac{1}{2}}\Alex{0}{n}{n}{0}b
		\arrow[dashed]{r}{q_{3}}
		\arrow[swap]{l}{p_{2}}
		&
		\delta^{0}\Alex{0}{n}{n-1}{0}c
		&
		\delta^{0}\Alex{0}{n-1}{n-2}{0}a
		\arrow[swap]{l}{p_{41}}
		\\
		\delta^{0}\Alex{0}{1-n}{2-n}{0}a
		\arrow[dashed]{r}{q_{32}}
		\arrow[phantom,pos=0.6]{u}{\vdots}
		&
		\delta^{0}\Alex{0}{-n}{1-n}{0}c
		&
		\delta^{\frac{1}{2}}\Alex{0}{-n}{-n}{0}b
		\arrow[swap]{l}{p_{412}}
		&
		\delta^{0}\Alex{0}{1-n}{-n}{0}a
		\arrow[dashed,swap]{l}{q_{2}}
		\arrow{r}{p_{41}}
		&
		\delta^{0}\Alex{0}{2-n}{1-n}{0}c
		\arrow[phantom,pos=0.6]{u}{\vdots}
		\end{tikzcd}
		$$
		\caption{$\mathcal{F}_{31}(\Pi(\mathfrak{b}_n))$}\label{fig:ConjugationBimodule:ImmersedCurve:Input}
	\end{subfigure}
	\begin{subfigure}[b]{\textwidth}
		\centering
		$$
		\begin{tikzcd}[row sep=1cm]
		\delta^{0}\Alex{2-n}{0}{0}{1-n}c
		\arrow[dashed]{r}{q_{14}}
		&
		\delta^{0}\Alex{1-n}{0}{0}{-n}a
		&
		\delta^{\frac{1}{2}}\Alex{-n}{0}{0}{-n}d
		\arrow[swap]{l}{p_{234}}
		&
		\delta^{0}\Alex{-n}{0}{0}{1-n}c
		\arrow[swap,dashed]{l}{q_{4}}
		\arrow{r}{p_{23}}
		&
		\delta^{0}\Alex{1-n}{0}{0}{2-n}a
		\\
		\delta^{0}\Alex{n-1}{0}{0}{n-2}a
		\arrow[phantom,pos=0.6]{u}{\vdots}
		&
		\delta^{0}\Alex{n}{0}{0}{n-1}c
		\arrow[dashed,swap]{l}{q_{14}}
		&
		\delta^{-\frac{1}{2}}\Alex{n}{0}{0}{n}d
		\arrow[dashed]{r}{q_{1}}
		\arrow[swap]{l}{p_{4}}
		&
		\delta^{0}\Alex{n-1}{0}{0}{n}a		
		&
		\delta^{0}\Alex{n-2}{0}{0}{n-1}c
		\arrow[swap]{l}{p_{23}}
		\arrow[phantom,pos=0.6]{u}{\vdots}
		\end{tikzcd}
		$$
		\caption{$\mathcal{F}_{13}(\Pi(\mathfrak{d}_n))$}\label{fig:ConjugationBimodule:ImmersedCurve:Output}
	\end{subfigure}
	\begin{subfigure}[b]{\textwidth}
		\centering
		$$
		\begin{tikzcd}[column sep=1.5cm]
		\delta^{0}\Alex{0}{0}{0}{0}a\otimes\field^m
		\arrow[bend left=10]{r}{p_{41}\otimes X}
		\arrow[bend right=10,swap,dashed]{r}{q_{32}\otimes\id_{\field^m}}
		&
		\delta^{0}\Alex{0}{0}{0}{0}c\otimes\field^m
		\end{tikzcd}
		$$
		\caption{$\mathcal{F}_{31}(\Pi(\Rat_X))$}\label{fig:ConjugationBimodule:RationalCurve:Output}
	\end{subfigure}
	\caption{Bigraded type D structures for irrational and rational curves of slope $\slopeZero$ }\label{fig:ConjugationBimodule:ImmersedCurve}
\end{figure}

\begin{definition}\label{def:ConjBimod}
	Let $\ConjBimod$ be the bigraded type AD $\Ad_{31}$-$\Ad_{13}$-bimodule from Figure~\ref{fig:ConjugationBimodule} on page~\pageref{fig:ConjugationBimodule}. Each vertex in this figure corresponds to two of the in total 32 generators of this type AD structure. Each vertex is labelled by a $\delta$- and generalized Alexander grading as well as an expression $\frac{x}{y}z$ with $x,y,z\in\{a,b,c,d\}$. Here, $x$ and $y$ specify the idempotents of the two generators on the type A side, respectively, and $z$ specifies the idempotent on the type D side. The vertices are connected by two rings of bold arrows carrying labels whose first component is empty. Each of these arrows corresponds to two components of the differential connecting the generators with the same first idempotents. The labels on all other arrows are decorated with super- and subscripts indicating the left idempotent of the source and target generators, respectively. Finally, there are also components of the differential which are not drawn in this graph and which connect each pair of generators corresponding to the same vertex, namely
	$$
	\delta^{0}\Alex{0}{0}{0}{0}bx
	\xrightarrow{(q_3|1)}
	\delta^{\frac{1}{2}}\Alex{0}{0}{-1}{0}cx
	\quad\text{ and }\quad
	\delta^{0}\Alex{0}{0}{0}{0}ax
	\xrightarrow{(p_1|1)}
	\delta^{\frac{1}{2}}\Alex{-1}{0}{0}{0}dx
	$$
	for any $x\in\{a,b,c,d\}$. 
	The bigrading indicated by the vertex labelling is for those generators whose first idempotent is $a$ or $b$. To obtain the correct bigrading of the generators with first idempotents $c$ or $d$, we need to add an additional shift by $\delta^{\frac{1}{2}}\Alex{0}{0}{-1}{0}$ and $\delta^{\frac{1}{2}}\Alex{-1}{0}{0}{0}$, respectively. Finally, the generalized Alexander grading on the algebra $\Ad_{31}$ on the type A side of $\ConjBimod$ is the opposite of the Alexander grading on the algebra $\Ad_{13}$ on the type D side of $\ConjBimod$ which agrees with the standard one from Definition~\ref{def:generalizedAlexgrading}; so for example $\AlexGr(p_1\in\Ad_{31})=\Alex{-1}{0}{0}{0}=-\AlexGr(p_1\in\Ad_{13})$.
\end{definition}

\begin{wrapfigure}{r}{0.3333\textwidth}
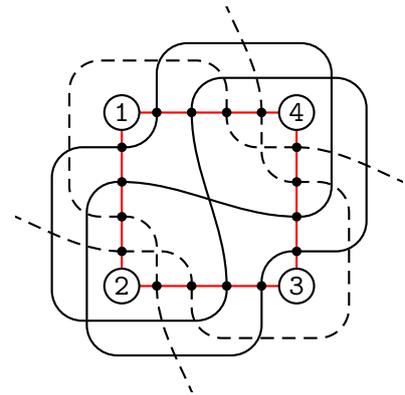

	\centering
	\bigskip\bigskip
	\twobasiccurves
	\caption{The two loops $L_{ad}$ (dashed) and $L_{bc}$ (solid)}\label{fig:twobasiccurvesINTRO}
	\bigskip\bigskip\bigskip
\end{wrapfigure}
\myfixwrapfig

\begin{Remark}
	The action of $\ConjBimod$ on a single generator $x$ can be easily described. 
	Depending on whether $x$ lives in one of the idempotent $a$ and $d$ or $b$ and $c$, the type D structure $x\boxtimes\ConjBimod$ corresponds to the immersed curve
	$L_{ad}$ or $L_{bc}$ from Figure~\ref{fig:twobasiccurvesINTRO}:
	$$
	a,d\mapsto\mathcal{F}_{13}(\Pi(L_{ad})),\qquad
	b,c\mapsto\mathcal{F}_{13}(\Pi(L_{bc})).\hspace{0.3333\textwidth}
	$$
	In Figure~\ref{fig:ConjugationBimodule}, these correspond to the two (or more precisely: four) rings of generators. 
	The action on morphisms is more complicated. We just note that the action of $\ConjBimod$ vanishes on any sequence of morphisms of length greater than 2.
\end{Remark}

\begin{lemma}\label{lem:ConjBimod}
	For any oriented 4-ended tangle \(T\), 
	$$
	\rr
	\Big(\mathcal{F}_{31}(\CFTd(T))\Big)\boxtimes\ConjBimod
	\cong
	V_{t_1}\otimes V_{t_2}\otimes\mathcal{F}_{13}(\CFTd(T))
	\hspace{0.3333\textwidth}
	$$
	as relatively bigraded type D structures over $\Ad_{13}$, where $V_t$ is as in Theorem~\ref{thm:stabilization}.
\end{lemma}

\begin{proof}
	By Theorem~\ref{thm:HalfIdBimodAction}, the bordered sutured type D structure of $M_T^\beta(p_3,q_1)$ is equal to the image of the peculiar module of $T$ under the functor $\mathcal{F}_{31}$ induced by the quotient map $\pi_{31}\co\Ad\rightarrow\Ad_{31}$, up to stabilization. Similarly, $M_T^\beta(p_1,q_3)$ is equal to $\mathcal{F}_{13}(\CFTd(T))$, up to stabilization. So it suffices to show that
	\begin{equation}\label{eqn:ConjBimod:proof}
	\rr
	\Big(\BSD(M_T^\beta(p_3,q_1))\Big)\boxtimes\ConjBimod
	\cong
	V_{t_1}\otimes V_{t_2}\otimes\BSD(M_T^\beta(p_1,q_3)).
	\end{equation}
	Let $M^\alpha_T(p_3,q_1)$ be the bordered sutured manifold whose underlying sutured manifold agrees with that of $M_T^\beta(p_3,q_1)$ up to an orientation reversal of the sutures and whose bordered sutured structure is obtained from that of $M_T^\beta(p_3,q_1)$ by turning the $\beta$-arcs into $\alpha$-arcs. In terms of Heegaard diagrams, this corresponds exactly to the usual procedure for conjugation symmetry, namely orientation reversal of the Heegaard surface and switching the roles of $\alpha$- and $\beta$-curves. For the algebraic invariants, this corresponds to reversing the Alexander grading:
	$$
	\rr
	\Big(\BSD(M_T^\beta(p_3,q_1))\Big)
	=
	\BSD(M_T^\alpha(p_3,q_1)).
	$$ 
	We now glue a thickened 4-punctured sphere onto $M_T^\alpha(p_3,q_1)$, which turns $\alpha$-arcs back into $\beta$-arcs. The Heegaard diagram of one such bordered sutured manifold $C$ is shown in Figure~\ref{fig:ConjugationHD}. Note that after glueing, we obtain the manifold $M_T^\beta(p_1,q_3)$ with one extra pair of oppositely oriented meridional sutures on each open tangle strand. So its bordered sutured structure agrees with the right hand side of \eqref{eqn:ConjBimod:proof} by Theorem~\ref{thm:stabilization}. 	 Then, using Zarev's Pairing Theorem for bordered sutured manifolds, it suffices to show that $\BSAD(C)=\ConjBimod$. 
	 
	The author originally did this computation using his Mathematica package~\cite{BSFH.m}. The reader might find the corresponding notebook \cite{ConjugationBimodule.nb} helpful. However, the computation is simple enough to do by hand. We sketch it below, but we leave the justification of some of the statements that we make to the reader. 
	
	\begin{figure}[t]
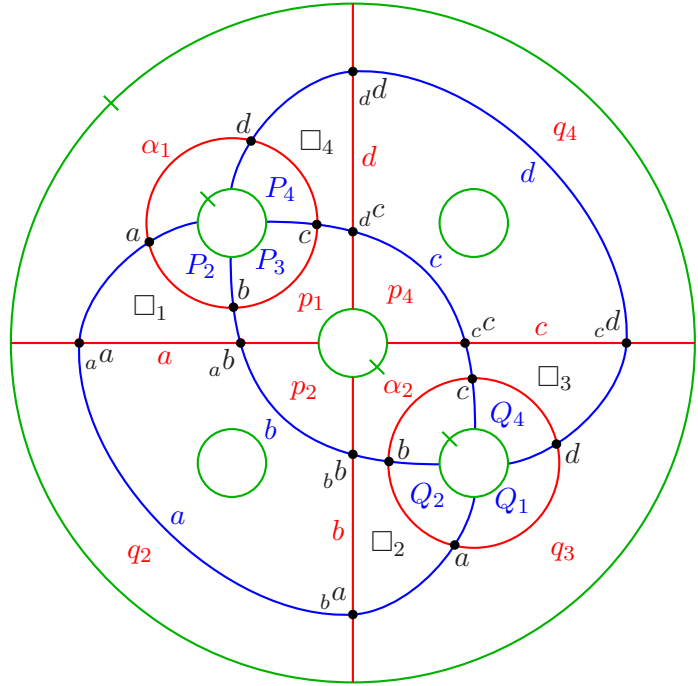

		\centering
		\ConjugationHD
		\caption{The Heegaard diagram of the $\alpha$-$\beta$-bordered sutured manifold $C$ whose corresponding type AD bimodule is $\ConjBimod$. The surface is oriented such that the normal vector points out of the plane of projection.}\label{fig:ConjugationHD}
	\end{figure}
	
	We name the generators in the Heegaard diagram from Figure~\ref{fig:ConjugationHD} by words $\prescript{}{w}{xyz}$, where $\prescript{}{w}{x}$ is the label of the intersection point on the $\alpha$-arcs, $y$ is the label of the intersection point on the curve $\alpha_1$ and $z$ is the label of the intersection point on the curve $\alpha_2$. Note that $w$ indicates the idempotent on the $\alpha$-bordered side and the complement of $\{x,y,z\}$ in $\{a,b,c,d\}$ is the idempotent on the $\beta$-bordered side. In particular, any triple of intersection points such that $x$, $y$ and $z$ are pairwise different gives rise to a generator. It is now easy to see that there are $4\cdot 3\cdot 4=48$ generators in total.
	There are in total eight identity components of the differential, namely the vertical arrows in Figure~\ref{fig:ConjugationBimodule:CancellationOnly}. 
	The remaining arrows of this figure show those domains between these 16 generators which contribute to the differential. Note that the only arrows from the first row to the second are said identity components, so we may do all cancellations simultaneously. For this, we need to find all contributing domains which end at the end or start at the start of the identity arrows. This is done in Figure~\ref{fig:ConjugationBimodule:CancellationContribution} for four of the eight identity arrows. The others can be obtained by switching the idempotents $a\leftrightarrow c$ and $b\leftrightarrow d$ on both sides, switching the last two letters of the generators, switching $p\leftrightarrow q$, $P\leftrightarrow Q$ and indices $1\leftrightarrow 3$ and $2\leftrightarrow 4$ on domains and reversing the orientation of all arrows. This symmetry can be seen in the Heegaard diagram from Figure~\ref{fig:ConjugationHD} in terms of a rotation by $\pi$ around the axis going through the two honest sutures followed by a reflection along the diagonal line from the top left to the bottom right connecting the other four sutures. 
	
	Next, let us compute the contributing domains between the remaining generators that survive the cancellation. We leave this computation to the reader; the result is shown in Figure~\ref{fig:ConjugationBimodule:partial}. Again, the above symmetry can be used to simplify the computation. In Figure~\ref{fig:ConjugationBimodule:partial}, this symmetry corresponds to a reflection along the vertical symmetry axis followed by switching the inner two rings with the other two rings of generators. We now add the contribution from the cancellation to this picture. These correspond to the $3\cdot 8=24$ labels with length 2 input sequences on the type A side, the four additional arrows
	$$
	\begin{tikzcd}[row sep=1cm, column sep=1.2cm]
	\prescript{}{c}{cba}
	\arrow{d}[description]{(\textcolor{red}{q_4}|\textcolor{blue}{Q_1})}
	&
	\prescript{}{d}{dba}
	\arrow{d}[description]{(\textcolor{red}{p_4}|\textcolor{blue}{P_3})}
	&
	\prescript{}{c}{cab}
	\arrow{d}[description]{(\textcolor{red}{q_4}|\textcolor{blue}{Q_2Q_1})}
	&
	\prescript{}{d}{dab}
	\arrow{d}[description]{(\textcolor{red}{p_4}|\textcolor{blue}{P_2P_3})}
	\\
	\prescript{}{d}{dbc}
	&
	\prescript{}{c}{cda}
	&
	\prescript{}{d}{dac}
	&
	\prescript{}{c}{cdb}
	\end{tikzcd}
	$$
	coming from the cancellation of the four identity arrows of Figure~\ref{fig:ConjugationBimodule:CancellationContribution} as well as another four arrows which are their symmetric counterparts. 
	
	\begin{figure}[t]
		\centering
		$$
		\begin{tikzcd}[row sep=1cm, column sep=1.2cm]
		\prescript{}{d}{cdb}
		\arrow{d}[description]{\{\square_4\}}
		\arrow[swap]{r}{\{\textcolor{blue}{Q_2}\}}
		&
		\prescript{}{d}{cda}
		\arrow{d}[description]{\{\square_4\}}
		&
		\prescript{}{b}{bca}
		\arrow{d}[description]{\{\square_2\}}
		\arrow[swap]{r}{\{\textcolor{blue}{P_4}\}}
		&
		\prescript{}{b}{bda}
		\arrow{d}[description]{\{\square_2\}}
		\arrow[bend right=20]{ll}[description]{\{\textcolor{red}{p_1},\textcolor{red}{p_2},\textcolor{blue}{P_3}\}}
		\\
		\prescript{}{d}{dcb}
		\arrow{r}{\{\textcolor{blue}{Q_2}\}}
		&
		\prescript{}{d}{dca}
		&
		\prescript{}{b}{acb}
		\arrow{r}{\{\textcolor{blue}{P_4}\}}
		\arrow[bend left=20]{ll}[description]{\{\textcolor{red}{q_3},\textcolor{red}{q_4},\textcolor{blue}{Q_1}\}}
		&
		\prescript{}{b}{adb}
		\end{tikzcd}
		\quad
		\begin{tikzcd}[row sep=1cm, column sep=1.2cm]
		\prescript{}{c}{dac}
		\arrow{d}[description]{\{\square_3\}}
		\arrow[swap]{r}{\{\textcolor{blue}{P_2}\}}
		&
		\prescript{}{c}{dbc}
		\arrow{d}[description]{\{\square_3\}}
		&
		\prescript{}{a}{abd}
		\arrow{d}[description]{\{\square_1\}}
		\arrow[swap]{r}{\{\textcolor{blue}{Q_4}\}}
		&
		\prescript{}{a}{abc}
		\arrow{d}[description]{\{\square_1\}}
		\arrow[bend right=20]{ll}[description]{\{\textcolor{red}{q_3},\textcolor{red}{q_2},\textcolor{blue}{Q_1}\}}
		\\
		\prescript{}{c}{cad}
		\arrow{r}{\{\textcolor{blue}{P_2}\}}
		&
		\prescript{}{c}{cbd}
		&
		\prescript{}{a}{bad}
		\arrow{r}{\{\textcolor{blue}{Q_4}\}}
		\arrow[bend left=20]{ll}[description]{\{\textcolor{red}{p_1},\textcolor{red}{p_4},\textcolor{blue}{P_3}\}}
		&
		\prescript{}{a}{bac}
		\end{tikzcd}
		$$
		\caption{All contributing domains connecting those generators of the Heegaard diagram from Figure~\ref{fig:ConjugationHD} that can be cancelled}\label{fig:ConjugationBimodule:CancellationOnly}
	\end{figure}
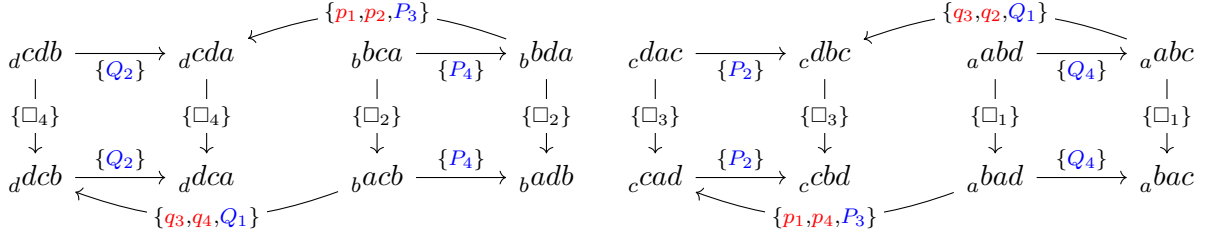
	
	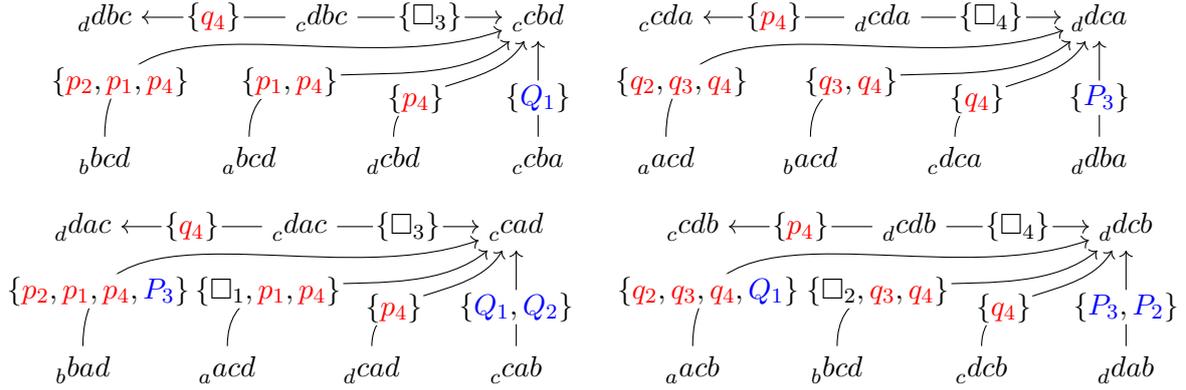
\begin{figure}[t]
		\centering
		
		\tikzstyle{tight}=[fill=white,inner sep=1pt,rounded corners]
		\begin{tikzpicture}[scale=1.9]
		\node(Ddbc) at (0,1)   {$\prescript{}{d}{dbc}$};
		\node(Cdbc) at (1.5,1) {$\prescript{}{c}{dbc}$};
		\node(Ccbd) at (3,1)   {$\prescript{}{c}{cbd}$};
		\node(Ccba) at (3,0)   {$\prescript{}{c}{cba}$};
		\node(Dcbd) at (2,0)   {$\prescript{}{d}{cbd}$};
		\node(Abcd) at (1,0)   {$\prescript{}{a}{bcd}$};
		\node(Bbcd) at (0,0)   {$\prescript{}{b}{bcd}$};
		
		\draw[->] (Cdbc) to node[tight]{$\{\square_3\}$} (Ccbd);
		\draw[->] (Cdbc) to node[tight]{$\{\textcolor{red}{q_4}\}$} (Ddbc);
		
		\draw[->,in=-90,out=90] (Ccba) to node[tight,pos=0.4]{$\{\textcolor{blue}{Q_1}\}$} (Ccbd);
		\draw[->,in=-120,out=90] (Dcbd) to node[tight,pos=0.3]{$\{\textcolor{red}{p_4}\}$} (Ccbd);
		\draw[->,in=-140,out=90] (Abcd) to node[tight,pos=0.3]{$\{\textcolor{red}{p_1},\textcolor{red}{p_4}\}$} (Ccbd);
		\draw[->,in=-160, out=90] (Bbcd) to node[tight,pos=0.15]{$\{\textcolor{red}{p_2},\textcolor{red}{p_1},\textcolor{red}{p_4}\}$} (Ccbd);
		\end{tikzpicture}
		\quad
		\begin{tikzpicture}[scale=1.9]
		\node(Ccda) at (0,1)   {$\prescript{}{c}{cda}$};
		\node(Dcda) at (1.5,1) {$\prescript{}{d}{cda}$};
		\node(Ddca) at (3,1)   {$\prescript{}{d}{dca}$};
		\node(Ddba) at (3,0)   {$\prescript{}{d}{dba}$};
		\node(Cdca) at (2,0)   {$\prescript{}{c}{dca}$};
		\node(Bacd) at (1,0)   {$\prescript{}{b}{acd}$};
		\node(Aacd) at (0,0)   {$\prescript{}{a}{acd}$};
		
		\draw[->] (Dcda) to node[tight]{$\{\square_4\}$} (Ddca);
		\draw[->] (Dcda) to node[tight]{$\{\textcolor{red}{p_4}\}$} (Ccda);
		
		\draw[->,in=-90,out=90] (Ddba) to node[tight,pos=0.4]{$\{\textcolor{blue}{P_3}\}$} (Ddca);
		\draw[->,in=-120,out=90] (Cdca) to node[tight,pos=0.3]{$\{\textcolor{red}{q_4}\}$} (Ddca);
		\draw[->,in=-140,out=90] (Bacd) to node[tight,pos=0.3]{$\{\textcolor{red}{q_3},\textcolor{red}{q_4}\}$} (Ccbd);
		\draw[->,in=-160, out=90] (Aacd) to node[tight,pos=0.15]{$\{\textcolor{red}{q_2},\textcolor{red}{q_3},\textcolor{red}{q_4}\}$} (Ccbd);
		\end{tikzpicture}
		\\
		\begin{tikzpicture}[scale=1.9]
		\node(Ddac) at (0,1)   {$\prescript{}{d}{dac}$};
		\node(Cdac) at (1.5,1) {$\prescript{}{c}{dac}$};
		\node(Ccad) at (3,1)   {$\prescript{}{c}{cad}$};
		\node(Ccab) at (3,0)   {$\prescript{}{c}{cab}$};
		\node(Dcad) at (2,0)   {$\prescript{}{d}{cad}$};
		\node(Aacd) at (1,0)   {$\prescript{}{a}{acd}$};
		\node(Bbad) at (0,0)   {$\prescript{}{b}{bad}$};
		
		\draw[->] (Cdac) to node[tight]{$\{\square_3\}$} (Ccad);
		\draw[->] (Cdac) to node[tight]{$\{\textcolor{red}{q_4}\}$} (Ddac);
		
		\draw[->,in=-90,out=90] (Ccab) to node[tight,pos=0.4]{$\{\textcolor{blue}{Q_1},\textcolor{blue}{Q_2}\}$} (Ccad);
		\draw[->,in=-120,out=90] (Dcad) to node[tight,pos=0.3]{$\{\textcolor{red}{p_4}\}$} (Ccad);
		\draw[->,in=-140,out=90] (Aacd) to node[tight,pos=0.3]{$\{\square_1,\textcolor{red}{p_1},\textcolor{red}{p_4}\}$} (Ccad);
		\draw[->,in=-160, out=90] (Bbad) to node[tight,pos=0.15]{$\{\textcolor{red}{p_2},\textcolor{red}{p_1},\textcolor{red}{p_4},\textcolor{blue}{P_3}\}$} (Ccad);
		\end{tikzpicture}
		\quad
		\begin{tikzpicture}[scale=1.9]
		\node(Ccdb) at (0,1)   {$\prescript{}{c}{cdb}$};
		\node(Dcdb) at (1.5,1) {$\prescript{}{d}{cdb}$};
		\node(Ddcb) at (3,1)   {$\prescript{}{d}{dcb}$};
		\node(Ddab) at (3,0)   {$\prescript{}{d}{dab}$};
		\node(Cdcb) at (2,0)   {$\prescript{}{c}{dcb}$};
		\node(Bbcd) at (1,0)   {$\prescript{}{b}{bcd}$};
		\node(Aacb) at (0,0)   {$\prescript{}{a}{acb}$};
		
		\draw[->] (Dcdb) to node[tight]{$\{\square_4\}$} (Ddcb);
		\draw[->] (Dcdb) to node[tight]{$\{\textcolor{red}{p_4}\}$} (Ccdb);
		
		\draw[->,in=-90,out=90] (Ddab) to node[tight,pos=0.4]{$\{\textcolor{blue}{P_3},\textcolor{blue}{P_2}\}$} (Ddcb);
		\draw[->,in=-120,out=90] (Cdcb) to node[tight,pos=0.3]{$\{\textcolor{red}{q_4}\}$} (Ddcb);
		\draw[->,in=-140,out=90] (Bbcd) to node[tight,pos=0.3]{$\{\square_2,\textcolor{red}{q_3},\textcolor{red}{q_4}\}$} (Ddcb);
		\draw[->,in=-160, out=90] (Aacb) to node[tight,pos=0.15]{$\{\textcolor{red}{q_2},\textcolor{red}{q_3},\textcolor{red}{q_4},\textcolor{blue}{Q_1}\}$} (Ddcb);
		\end{tikzpicture}
		\caption{All contributing domains which end at the end or start at the start of four identity components from Figure~\ref{fig:ConjugationBimodule:CancellationOnly} and whose other generator survives cancellation }\label{fig:ConjugationBimodule:CancellationContribution}
	\end{figure}
	It now remains to discuss the bigrading. Since both $\delta$-grading and Alexander grading are relative gradings and the bimodule is connected (ie not a direct sum of two non-trivial bimodules), it suffices to fix a bigrading for one generator and work out the bigrading of all other generators from the differentials. Again, we leave the verification that the bigrading from Figure~\ref{fig:ConjugationBimodule} is well-defined to the reader.  
\end{proof}

\begin{lemma}\label{lem:ConjBimod:ActionOnHorizontalComponents}
	For any $n>0$,
	$$
	\rr
	\Big(\mathcal{F}_{31}(\Pi(\mathfrak{b}_n))\Big)\boxtimes\ConjBimod\cong
	V_{t_b}\otimes V_{t'_b}\otimes\mathcal{F}_{13}(\Pi(\mathfrak{d}_n))
	$$
	as absolutely bigraded type D structures over $\Ad_{13}$, where $t_b$ and $t'_b$ are the Alexander gradings of the tangle ends adjacent to the site $b$. The same is true after switching $\mathfrak{b}_n$ and $\mathfrak{d}_n$ as well as $b$ and $d$. Moreover, for any local system $X\in\GL(n,\field)$,
	$$
	\rr
	\Big(
	\mathcal{F}_{31}(\Pi(\Rat_X))\Big)\boxtimes\ConjBimod
	\cong
	V_{t_1}\otimes V_{t_2}
	\otimes
	\mathcal{F}_{13}(\Pi(\Rat_{X^{-1}}))
	$$
	as absolutely bigraded type D structures over $\Ad_{13}$,
	where $t_1$ and $t_2$ are the Alexander gradings of the two open strands.  
\end{lemma}

\begin{Remark}
	The proof of this lemma is perhaps the most technical one in this paper. The action of the bimodule $\ConjBimod$ on any given curve can also be computed using the Mathematica notebook~\cite{ConjugationAction.nb} and the author's Mathematica package~\cite{PQM.m}. Based on such computations, we expect this lemma to be actually true for all curves $\Rat_X(\slope)$ and $\mathfrak{i}_n(\slope;i,j)$ equipped with a suitable absolute bigrading. 
\end{Remark}

\begin{figure}[t]
		\centering
		
		\tikzstyle{nodeTLI}=[draw,ellipse,inner sep=0pt,dotted,minimum width = 20pt, minimum height=12pt]
		\tikzstyle{nodeBRO}=[draw,rectangle,rounded corners,dotted,minimum width = 20pt, minimum height=17pt]
		\begin{tikzpicture}[scale=1.9]
		
		\draw (0,3)  node[nodeBRO](odl){$\delta^{-\frac{1}{2}}\Alex{0}{1-n}{1-n}{0}d$};
		\draw (0,0)  node[nodeBRO](ocl){$c$};
		\node(opl1) at (1,3){};
		\node(opl2) at (1,0){};
		
		\draw (0.9,2)  node[nodeTLI](idl){$\delta^{-\frac{1}{2}}\Alex{0}{-n-1}{-n-1}{0}d$};
		\draw (0.9,1)  node[nodeTLI](icl1){$c$};
		\draw (2.3,2)  node[nodeTLI](ial1){$a$};
		\draw (2.3,1)  node[nodeTLI](ial2){$a$};
		\draw (3,2)  node[nodeTLI](icl2){$c$};
		\node(ipl1) at (3,1){};
		\node(ipl2) at (3.7,2){};
		
		\node(opr1) at (3,3){};
		\node(opr2) at (3,0){};
		\draw (4,3)  node[nodeBRO](oar1){$a$};
		\draw (4,0)  node[nodeBRO](oar2){$a$};
		\draw (5,3)  node[nodeBRO](ocr1){$c$};
		\draw (5,0)  node[nodeBRO](ocr2){$c$};
		\draw (6,3)  node[nodeBRO](oar3){$a$};
		\draw (6,0)  node[nodeBRO](oar4){$a$};
		\draw (7,3)  node[nodeBRO](ocr3){$c$};
		\draw (7,0)  node[nodeBRO](odr){$\delta^{\frac{1}{2}}\Alex{0}{n+1}{n+1}{0}d$};
		
		\node(ipr1) at (3.5,1){};
		\node(ipr2) at (4.2,2){};
		\draw (4.2,1)  node[nodeTLI](icr1){$c$};
		\draw (4.9,2)  node[nodeTLI](iar1){$a$};
		\draw (4.9,1)  node[nodeTLI](iar2){$a$};
		\draw (6.3,2)  node[nodeTLI](icr2){$c$};
		\draw (6.3,1)  node[nodeTLI](idr){$\delta^{\frac{1}{2}}\Alex{0}{n-1}{n-1}{0}d$};
		
		\draw[->] (odl) to [bend right=10] node[left]{$p_4$} (ocl);
		\draw[->] (ocr3) to [bend left=10] node[right]{$q_4$} (odr);
		
		\draw[->] (ocr3) to node[above]{$p_{23}$} (oar3);
		\draw[->] (odr) to node[above]{$p_{234}$} (oar4);
		
		\draw[->] (ocr1) to node[above]{$q_{14}$} (oar3);
		\draw[->] (ocr2) to node[above]{$q_{14}$} (oar4);
		
		\draw[->] (ocr1) to node[above]{$p_{23}$} (oar1);
		\draw[->] (ocr2) to node[above]{$p_{23}$} (oar2);
		
		\draw[->,dashed] (odl) to node[above]{$q_{1}$} (opl1);
		\draw[->,dashed] (ocl) to node[above]{$q_{14}$} (opl2);
		\draw[->,dashed] (opr1) to node[above]{$q_{1}$} (oar1);
		\draw[->,dashed] (opr2) to node[above]{$q_{14}$} (oar2);
		
		\draw[->] (idl) to [bend right=10] node[left]{$p_4$} (icl1);
		\draw[->] (icr2) to [bend left=10] node[right]{$q_4$} (idr);
		
		\draw[->] (icr2) to node[above]{$p_{23}$} (iar1);
		\draw[->] (idr) to node[above]{$p_{234}$} (iar2);
		
		\draw[->,dashed] (ipr2) to node[above]{$q_{14}$} (iar1);
		\draw[->] (icr1) to node[above]{$q_{14}$} (iar2);
		\draw[->,dashed] (icr1) to node[above]{$p_{23}$} (ipr1);
		
		\draw[->] (idl) to node[above]{$q_{1}$} (ial1);
		\draw[->] (icl1) to node[above]{$q_{14}$} (ial2);
		
		\draw[->] (icl2) to node[above]{$p_{23}$} (ial1);
		\draw[->,dashed] (ipl1) to node[above]{$p_{23}$} (ial2);
		\draw[->,dashed] (icl2) to node[above]{$q_{14}$} (ipl2);
		
		\draw[->,dotted] (ocr1) .. controls  +(0,-0.4) and +(-1,0) .. +(1,-0.5) node[above]{$q_{4}$} .. controls  +(1,0) and +(0,1) .. (idr.north east);
		\draw[->,dotted] (odl) .. controls  +(1,-0.4) and +(-1,0) .. +(2.3,-0.5) node[above]{$q_{1}$} .. controls  +(1,0) and +(0,0.5) .. (ipl2);
		\draw[->,dotted,in=-135,out=45] (ocl) to node[above]{$p_{23}$} (ial2);
		\draw[->,dotted,in=-90,out=90] (ocr2) to node[left]{$p_{23}$} (iar2);
		\draw[->,dotted] (odl) .. controls  +(-0.16,-1.4) and +(-2,0) .. +(1.5,-1.3) node[below]{$p_{4}$} .. controls  +(1,0) and +(0,0.5) .. (ipl1);
		
		\end{tikzpicture}
	\caption{The third component of $\rr(\mathfrak{b}_n)\boxtimes\ConjBimod$ after cancellation}\label{fig:ConjugationBimodule:ArrowPushing}
\end{figure}

\begin{proof}
	Let us first consider $\mathfrak{b}_n$ with $n>1$. We can divide the computation into smaller pieces by computing the action of $\rr(\cdot)\boxtimes\ConjBimod$ first on the following four complexes:
	$$
	\begin{tikzcd}[row sep=0.2cm, column sep=0.6cm]
	\cdots
	\arrow[dashed]{r}{p_{41}}
	&
	\delta^{0}\Alex{0}{-2}{-1}{0}c
	&
	\delta^{0}\Alex{0}{-1}{0}{0}a
	\arrow[swap]{l}{q_{32}}
	&
	\delta^{-\frac{1}{2}}\Alex{0}{0}{0}{0}b
	\arrow{r}{q_{3}}
	\arrow[swap]{l}{p_{2}}
	&
	\delta^{0}\Alex{0}{0}{-1}{0}c
	&
	\delta^{0}\Alex{0}{-1}{-2}{0}a
	\arrow[swap]{l}{p_{41}}
	\arrow[dashed]{r}{q_{32}}
	&
	\cdots
	\\
	\cdots
	&
	\delta^{0}\Alex{0}{1}{2}{0}a
	\arrow{r}{q_{32}}
	\arrow[dashed,swap]{l}{p_{41}}
	&
	\delta^{0}\Alex{0}{0}{1}{0}c
	&
	\delta^{\frac{1}{2}}\Alex{0}{0}{0}{0}b
	\arrow[swap]{l}{p_{412}}
	&
	\delta^{0}\Alex{0}{1}{0}{0}a
	\arrow[swap]{l}{q_{2}}
	\arrow{r}{p_{41}}
	&
	\delta^{0}\Alex{0}{2}{1}{0}c
	&
	\cdots
	\arrow[dashed,swap]{l}{q_{32}}
	\end{tikzcd}
	$$
	$$
	\begin{tikzcd}[row sep=0.5cm, column sep=0.7cm]
	\cdots
	&
	\delta^{0}\Alex{0}{1}{1}{0}a
	\arrow{r}{q_{32}}
	\arrow[dashed,swap]{l}{p_{41}}
	&
	\delta^{0}\Alex{0}{0}{0}{0}c
	&
	\cdots
	\arrow[swap,dashed]{l}{p_{41}}
	&
	\cdots
	&
	\delta^{0}\Alex{0}{0}{0}{0}a
	\arrow[dashed,swap]{l}{q_{32}}
	\arrow{r}{p_{41}}
	&
	\delta^{0}\Alex{0}{1}{1}{0}c
	&
	\cdots
	\arrow[dashed,swap]{l}{q_{32}}
	\end{tikzcd}
	$$
	These computations are done in the
	Figures~\ref{fig:ConjugationBimodule:down:raw}, \ref{fig:ConjugationBimodule:up:raw} and~\ref{fig:ConjugationBimodule:ac:raw}. In those diagrams, certain pairs of generators which are connected by identity arrows are shaded grey. If we apply the Cancellation Lemma to each of these, we obtain the complexes shown in the Figures~\ref{fig:ConjugationBimodule:down:cancelled}, \ref{fig:ConjugationBimodule:up:cancelled} and~\ref{fig:ConjugationBimodule:ac:cancelled}, respectively.
	
	Let us now put these complexes back together. In order to obtain the same Alexander grading as the corresponding curve segment of $\mathfrak{b}_n$, the first of the four segments above needs to be multiplied by $\Alex{0}{n}{n}{0}$; so its contribution to $\rr(\mathfrak{b}_n)\boxtimes\ConjBimod$ is the complex from Figure~\ref{fig:ConjugationBimodule:down:cancelled} multiplied by $\Alex{0}{-n}{-n}{0}$. Similarly, the second complex above needs to be multiplied by $\Alex{0}{-n}{-n}{0}$, so the complex from Figure~\ref{fig:ConjugationBimodule:up:cancelled} needs to be multiplied by $\Alex{0}{n}{n}{0}$. The third complex and the fourth also need to be shifted in Alexander grading, but the grading shifts depend on where the complexes sit in $\mathfrak{b}_n$; it turns out to be sufficient to compute the gradings for the first two complexes.
	
	In the Figures~\ref{fig:ConjugationBimodule:down:cancelled}, \ref{fig:ConjugationBimodule:up:cancelled} and~\ref{fig:ConjugationBimodule:ac:cancelled}, those generators which can be cancelled in $\rr(\mathfrak{b}_n)\boxtimes\ConjBimod$ are drawn without a frame. The remaining ones are partitioned into four groups of generators for each figure, indicated by the frame style (dotted/solid, ellipse/rectangle). We claim that the generators in each of these four groups and the differentials between them form a copy of $\mathcal{F}_{13}(\Pi(\mathfrak{d}_n))$, with a suitable shift in bigrading. This can be seen as follows: clearly, each group of generators in the first two figures forms a complex which looks like $\mathcal{F}_{13}(\Pi(\mathfrak{d}_n))$ around the generators in the idempotent $d$, but the lengths of those complexes vary. The third figure shows that alternating sequences of generators in idempotents $a$ and $c$ result in four such sequences, each of which sits either at the right, top, left or bottom of the digram. In $\rr(\mathfrak{b}_n)\boxtimes\ConjBimod$, these sequences connect each group of generators in Figure~\ref{fig:ConjugationBimodule:down:cancelled} with the corresponding group in Figure~\ref{fig:ConjugationBimodule:up:cancelled}. Putting these pieces together, we see that $\rr(\mathfrak{b}_n)\boxtimes\ConjBimod$ splits into three components, two of which are equal to $\delta^0\Alex{0}{-1}{1}{0}\mathfrak{d}_n$ (solid-ellipse) and $\delta^0\Alex{0}{1}{-1}{0}\mathfrak{d}_n$ (solid-rectangle). The third component is drawn in Figure~\ref{fig:ConjugationBimodule:ArrowPushing}. Note that the dotted arrows can be removed without changing the remaining complex by two isotopies. So we also get $\delta^0\Alex{0}{-1}{-1}{0}\mathfrak{d}_n$ (dotted-ellipse) and $\delta^0\Alex{0}{1}{1}{0}\mathfrak{d}_n$ (dotted-rectangle). This proves the statement for $\mathfrak{b}_n$ and $n>1$. For $n=1$, the computation is very similar; we leave the details to the reader. 
	
	Instead of doing the same computation again for $\mathfrak{d}_n$, we use the symmetries of the bimodule $\ConjBimod$. Namely, the bimodule is invariant under simultaneously interchanging $a\leftrightarrow c$, $b\leftrightarrow d$, algebra indices and Alexander gradings $2\leftrightarrow 4$ and $1\leftrightarrow 3$, $p\leftrightarrow q$, reversing bigrading and reversing the orientation of all arrows. This is the same symmetry that we already used for the computation of $\ConjBimod$ in the proof of Lemma~\ref{lem:ConjBimod}. It is also the symmetry between $\mathfrak{b}_n$ and $\mathfrak{d}_n$.
	
	Finally, let us consider the curve $\Rat_X$ from Figure~\ref{fig:ConjugationBimodule:RationalCurve:Output}. For the computation of $\rr(\Rat_X)\boxtimes\ConjBimod$, we can reuse the computation from Figures~\ref{fig:ConjugationBimodule:ac:raw} and~\ref{fig:ConjugationBimodule:ac:cancelled} to see that the complex splits into four components, whose underlying curves are equal to $\Rat$. It remains to determine their gradings and local systems. The component at the bottom of these figures is equal to
	\begin{eqnarray}\label{eqn:rat_curve}
	\begin{tikzcd}[column sep=1.5cm]
	\delta^{0}\Alex{0}{0}{-1}{-1}c\otimes\field^m
	\arrow[bend left=10]{r}{p_{23}\otimes X}
	\arrow[bend right=10,swap,dashed]{r}{q_{14}\otimes\id_{\field^m}}
	&
	\delta^{0}\Alex{0}{0}{-1}{-1}a\otimes\field^m
	\end{tikzcd}
	\end{eqnarray}
	For the calculation of the Alexander grading, note that the tangle ends $\texttt{1}$ and $\texttt{4}$ as well as $\texttt{2}$ and $\texttt{3}$ belong to the same tangle strands, so we need to take the quotient of the generalized Alexander grading $\AlexGr$ from Definition~\ref{def:generalizedAlexgrading} by the additional relation $\Alex{1}{0}{0}{1}\sim\Alex{0}{0}{0}{0}$.
	Now, when we write $\mathcal{F}_{31}(\Pi(\Rat_X))$ as in Figure~\ref{fig:ConjugationBimodule:RationalCurve:Output}, we are assuming that the underlying curve of $\Rat_X$ is oriented from $a$ to $c$ via the front face, ie like
	$$
	\begin{tikzpicture}[scale=1.9]
	\draw (1,0)  node(c){$c$};
	\draw (0,0)  node(a){$a$};
	\draw[bend left,->] (a) to node[draw, fill=white,rectangle]{$X$} (c);
	\draw[bend left,->,dashed] (c) to (a);
	\end{tikzpicture}
	$$
	Using the same orientation on $\Rat$, the complex \eqref{eqn:rat_curve} corresponds to $\Rat_{X^{-1}}$, ie
	$$
	\begin{tikzpicture}[scale=1.9]
	\draw (1,0)  node(c){$c$};
	\draw (0,0)  node(a){$a$};
	\draw[bend left,->] (a) to node[draw, fill=white,rectangle]{$X^{-1}$} (c);
	\draw[bend left,->,dashed] (c) to (a);
	\end{tikzpicture}
	$$
	We obtain the same result, albeit with different grading shifts, for the other three components. This computation is done in Figure~\ref{fig:ConjugationBimodule:EmbeddedCurve}.
\end{proof}

\begin{figure}[t]
	\centering
	\begin{subfigure}[b]{0.32\textwidth}
		\centering
		$$
		\begin{tikzcd}[column sep=0.4cm,row sep=1cm]
		\Alex{0}{0}{1}{-1}c\otimes\field^m
		\arrow[swap,dashed]{d}{q_{14}\otimes\id}
		\arrow[out=0,in=90,pos=0.4]{ddr}{p_{3}\otimes X}
		\\
		\Alex{0}{0}{1}{-1}a\otimes\field^m
		\\
		b\otimes\field^m
		\arrow{u}{p_{2}\otimes\id}
		\arrow{r}{1\otimes\id}
		&
		b\otimes\field^m
		\\
		\Alex{0}{0}{1}{-1}c\otimes\field^m
		\arrow[bend left=10]{r}{p_{23}\otimes X}
		\arrow[bend right=10,swap,dashed]{r}{q_{14}\otimes\id}
		\arrow[phantom,bend left=60]{r}{\cong}%
		&
		\Alex{0}{0}{1}{-1}a\otimes\field^m
		\end{tikzcd}
		$$
		\caption{left}\label{fig:ConjugationBimodule:EmbeddedCurve:Left}
	\end{subfigure}
	\begin{subfigure}[b]{0.32\textwidth}
		\centering
		$$
		\begin{tikzcd}[column sep=0.4cm,row sep=1cm]
		b\otimes\field^m
		\arrow[swap]{d}{p_{2}\otimes\id}
		\arrow{r}{1\otimes\id}
		&
		b\otimes\field^m
		\\
		\Alex{0}{0}{1}{1}a\otimes\field^m
		&
		\Alex{0}{0}{1}{1}c\otimes\field^m
		\arrow[swap]{u}{p_{3}\otimes\id}
		\arrow[dashed]{d}{q_{4}\otimes\id}
		\\
		d\otimes\field^m
		\arrow[dashed]{u}{q_{1}\otimes\id}
		\arrow{r}{1\otimes X}
		&
		d\otimes\field^m
		\\
		\Alex{0}{0}{1}{1}a\otimes\field^m
		\arrow[phantom,bend left=60]{r}{\cong}%
		&
		\Alex{0}{0}{1}{1}c\otimes\field^m
		\arrow[bend left=10]{l}{p_{23}\otimes \id}
		\arrow[bend right=10,swap,dashed]{l}{q_{14}\otimes X^{-1}}
		\end{tikzcd}
		$$
		\caption{top}\label{fig:ConjugationBimodule:EmbeddedCurve:Top}
	\end{subfigure}
	\begin{subfigure}[b]{0.32\textwidth}
		\centering
		$$
		\begin{tikzcd}[column sep=0.4cm,row sep=1cm]
		&
		\Alex{0}{0}{-1}{1}a\otimes\field^m
		\\
		&
		\Alex{0}{0}{-1}{1}c\otimes\field^m
		\arrow[swap]{u}{p_{23}\otimes\id}
		\arrow[dashed]{d}{q_{4}\otimes\id}
		\\
		d\otimes\field^m
		\arrow[out=90,in=180,dashed,pos=0.6]{uur}{q_{1}\otimes\id}
		\arrow{r}{1\otimes X}
		&
		d\otimes\field^m
		\\
		\Alex{0}{0}{-1}{1}a\otimes\field^m
		\arrow[phantom,bend left=60]{r}{\cong}%
		&
		\Alex{0}{0}{-1}{1}c\otimes\field^m
		\arrow[bend left=10]{l}{p_{23}\otimes \id}
		\arrow[bend right=10,swap,dashed]{l}{q_{14}\otimes X^{-1}}
		\end{tikzcd}
		$$
		\caption{right}\label{fig:ConjugationBimodule:EmbeddedCurve:Right}
	\end{subfigure}
	\caption{The calculation of the local systems on the rational curve $\Rat$ in the proof of Lemma~\ref{lem:ConjBimod:ActionOnHorizontalComponents}. The $\delta$-grading of all generators in idempotents $a$ and $c$ is 0.}\label{fig:ConjugationBimodule:EmbeddedCurve}
\end{figure}

\begin{proof}[Proof of Theorem~\ref{thm:Conjugation:Horizontals}]
Let us fix an absolute lift of the bigrading of $\HFT(T)$. Let us also introduce the following notation: given two Laurent polynomials $P$ and $Q$ in finitely many variables, we write $P\leq Q$ if for every monomial, the coefficient of that monomial in $P$ is smaller than or equal to the corresponding coefficient in $Q$. Then
\allowdisplaybreaks
\begin{align*}
&\phantom{=}(t_1+t_1^{-1})(t_2+t_2^{-1})\cdot \mathfrak{d}_n(T)
\\
&=
\mathfrak{d}_n(V_{t_1}\otimes V_{t_2}\otimes\CFTd(T))
\\
&=
\delta^mt^A \mathfrak{d}_n\left(\rr(\CFTd(T))\boxtimes\ConjBimod\right)
&&
\text{(by Lemma~\ref{lem:ConjBimod})}
\\
&\geq 
\delta^mt^A \mathfrak{d}_n\left(\rr(\mathfrak{b}_n(T)\cdot\mathfrak{b}_n)\boxtimes\ConjBimod\right)
&&
\text{($\ConjBimod$ commutes with direct sums)}
\\
&=
\delta^mt^A\rr(\mathfrak{b}_n(T))\cdot \mathfrak{d}_n(\rr(\mathfrak{b}_n)\boxtimes\ConjBimod)
\\
&=
\delta^mt^A\rr(\mathfrak{b}_n(T))\cdot (t_b+t_b^{-1})(t_b'+t_b^{\prime-1})
&&
\text{(by Lemma~\ref{lem:ConjBimod:ActionOnHorizontalComponents}).}
\end{align*}
Similarly, we see that
$$
(t_1+t_1^{-1})(t_2+t_2^{-1})\cdot \mathfrak{b}_n(T)
\geq
\delta^mt^A\rr(\mathfrak{d}_n(T))\cdot (t_d+t_d^{-1})(t_d'+t_d^{\prime-1})
$$
for the same bigrading $\delta^mt^A$, 
or equivalently, 
$$
(t_1+t_1^{-1})(t_2+t_2^{-1})\cdot \rr(\mathfrak{b}_n(T))
\geq
\delta^mt^{-A}\mathfrak{d}_n(T)\cdot (t_d+t_d^{-1})(t_d'+t_d^{\prime-1}).
$$
If we combine these two inequalities, we obtain
\begin{align*}
&\phantom{=}(t_1+t_1^{-1})^2(t_2+t_2^{-1})^2\cdot \mathfrak{d}_n(T)
\\
&\geq
\delta^mt^A(t_1+t_1^{-1})(t_2+t_2^{-1}) (t_b+t_b^{-1})(t_b'+t_b^{\prime-1})\cdot \rr(\mathfrak{b}_n(T)),
\\
&\geq
\delta^{2m}(t_d+t_d^{-1})(t'_d+t_d^{\prime-1}) (t_b+t_b^{-1})(t_b'+t_b^{\prime-1})\cdot \mathfrak{d}_n(T),
\end{align*}
Note that
$$(t_d+t_d^{-1})(t'_d+t_d^{\prime-1}) (t_b+t_b^{-1})(t_b'+t_b^{\prime-1})=(t_1+t_1^{-1})^2(t_2+t_2^{-1})^2,$$
so $m=0$ and all inequalities above are actually equalities. In particular, 
$$
(t_1+t_1^{-1})(t_2+t_2^{-1})\cdot t^{-\frac{1}{2}A}\mathfrak{d}_n(T)
=
\rr(t^{-\frac{1}{2}A}\mathfrak{b}_n(T))\cdot (t_b+t_b^{-1})(t_b'+t_b^{\prime-1}). 
$$
Similarly, we can show that
$$
(t_1+t_1^{-1})(t_2+t_2^{-1})\cdot \Rat_{X^{-1}}(T)
\geq
\delta^mt^A\rr(\Rat_{X}(T))\cdot (t_1+t_1^{-1})(t_2+t_2^{-1}),
$$
ie 
$$
\Rat_{X^{-1}}(T)
\geq
\delta^mt^A\rr(\Rat_{X}(T))
$$
for any $X\in\GL(n,\field)$ and the same bigrading $\delta^mt^A$ as above. So we also have
$$
\Rat_{X}(T)
\geq
\delta^mt^A\rr(\Rat_{X^{-1}}(T)),
$$
which, combined with the first inequality, shows that the inequalities are actually equalities and $m=0$. In particular, 
$$
t^{-\frac{1}{2}A}\Rat_{X^{-1}}(T)
=
\rr(t^{-\frac{1}{2}A}\Rat_{X}(T)).
$$
So if we shift the Alexander grading of $\HFT(T)$ by $t^{\frac{1}{2}A}$, we obtain the absolute lift of the Alexander grading with the desired properties. 
\end{proof}

\begin{definition}
	Let $\gamma$ be an absolutely $\delta$-graded linear curve of slope $\slope\in\QPI$, possibly with indecomposable local system. 
	It is easy to see that unless $\slope=\frac{1}{0}$, all intersection points of $\gamma$ with the vertical sites $a$ and $c$ have the same $\delta$-grading. We denote this $\delta$-grading by $\delV(\gamma)$. Similarly, if $\slope\neq\frac{0}{1}$, all intersection points of $\gamma$ with the horizontal sites $b$ and $d$ have the same $\delta$-grading, which we denote by $\delH(\gamma)$. Given two absolutely $\delta$-graded linear curves $\gamma$ and $\gamma'$ of the same slope $\slope\in\QPI$, we define
	$$
	\delta(\gamma,\gamma')
	\coloneqq
	\begin{cases*}
	\delV(\gamma')-\delV(\gamma) & if $\slope\neq\frac{1}{0}$
	\\
	\delH(\gamma')-\delH(\gamma) & if $\slope\neq\frac{0}{1}$.
	\end{cases*}
	$$
	This is well-defined, since 
	$$
	\delV(\gamma)=
	\begin{cases*}
	\delH(\gamma) -\tfrac{1}{2} & if $\frac{0}{1}<\slope<\frac{1}{0}$
	\\
	\delH(\gamma) +\tfrac{1}{2} & if $\frac{1}{0}<\slope<\frac{0}{1}$.
	\end{cases*}
	$$
	We say two absolutely $\delta$-graded curves $\gamma$ and $\gamma'$ of the same slope are in the same $\delta$-grading if $\delta(\gamma,\gamma')=0$.
\end{definition}

\begin{lemma}\label{lem:MCG_and_delta}
	For any two curves $\gamma$ and $\gamma'$ of the same slope $\slope\in\QPI$ and a single half-twist $\tau$ as in Section~\ref{sec:ActionMCG}, $\delta(\tau(\gamma),\tau(\gamma'))=\delta(\gamma,\gamma')$.
\end{lemma}

\begin{proof}
	A single twist at the right or left preserves the generators in the vertical sites $a$ and $c$ and also their $\delta$-gradings up to an overall shift which is the same for all curves. So, for curves $\gamma$ and $\gamma'$ of slope $\slope\neq\frac{1}{0}$,
	$$
	\delta(\tau(\gamma),\tau(\gamma'))
	=
	\delV(\tau(\gamma'))-\delV(\tau(\gamma))
	=
	\delV(\gamma')-\delV(\gamma)
	=
	\delta(\gamma,\gamma').
	$$
	Moreover, such a twist does not affect curves of slope $\frac{1}{0}$ up to an overall grading shift. We can argue in the same way for twists in the other direction.	
\end{proof}

\begin{definition}\label{def:choice_of_delta_grading}
	For \(n>0\) and \(\slope\in\QPI\) and any local system \(X\), fix a \(\delta\)-grading on the curves \(\Irr^\delta_n(\slope;i_1,i_2)\), \(\Irr^\delta_n(\slope;i_3,i_4)\) and \(\Rat^\delta_X(\slope)\) such that any two of them of the same slope are in the same \(\delta\)-grading.
\end{definition}

\begin{proof}[Proof of Theorem~\ref{thm:Conjugation:delta} (\(\delta\)-graded conjugation symmetry)]
	For $\slope=\frac{0}{1}$, this follows immediately from Theorem~\ref{thm:Conjugation:Horizontals}, noting that the curves $\mathfrak{b}_n$, $\mathfrak{d}_n$ and $\Rat_X$ have the same $\delta$-grading. For general $\slope\in\QPI$, let $T'$ be a tangle obtained from $T$ by applying an element $\rho$ of the mapping class group which sends linear curves of slope $\slope$ to curves of slope $\tfrac{0}{1}$. By Theorem~\ref{thm:MCGaction}, $\rho(\HFTdelta(T))=\delta^m\HFTdelta(T')$ for some $m\in\tfrac{1}{2}\mathbb{Z}$. By Lemma~\ref{lem:MCG_and_delta}, the difference between the $\delta$-gradings of two curves of the same slope is preserved under twisting. So in particular, the image of the reference curves \(\Irr_n(\slope;i_1,i_2)\), \(\Irr_n(\slope;i_3,i_4)\) and \(\Rat_X(\slope)\) have the same $\delta$-grading. Since we already know the result for horizontal curves, we can apply it to $\HFT(T')$ and see that
	$$
	\Rat^\delta_X(\slope)(T)
	=
	\delta^m\Rat^\delta_X(\tfrac{0}{1})(T')
	=
	\delta^m\Rat^\delta_{X^{-1}}(\tfrac{0}{1})(T')
	=
	\Rat^\delta_{X^{-1}}(\slope)(T)
	$$
	and also 
	\begin{align*}
	&\Irr^\delta_n(\slope;i_1,i_2)(T)
	=
	\delta^m \Irr^\delta_n(\tfrac{0}{1};2,3)(T')
	=
	\delta^m \Irr^\delta_n(\tfrac{0}{1};1,4)(T')
	=
	\Irr^\delta_n(\slope;i_3,i_4)(T).
	\qedhere
	\end{align*}
\end{proof}

\input{sections/ConjBimod_bimod_part.tex}
\input{sections/ConjBimod_bimod_final.tex}

\input{sections/ConjBimod_down_raw.tex}
\input{sections/ConjBimod_down_cancelled.tex}
\input{sections/ConjBimod_up_raw.tex}
\input{sections/ConjBimod_up_cancelled.tex}
\input{sections/ConjBimod_ac_raw.tex}
\input{sections/ConjBimod_ac_cancelled.tex}

\begin{proof}[Proof of Theorem~\ref{thm:MutInvHFT}]
	The slopes of the components of $\HFT$ are preserved under any mutation. This can be seen by observing that mutation preserves rational tangles and hence also the rational curves $\Rat(\slope)$ for all slopes $\slope\in\QPI$. The rational curve $\Rat(\slope)$ separates the irrational curves $\Irr_n(\slope;i_1,i_2)$ and $\Irr_n(\slope;i_3,i_4)$. So mutation either preserves an irrational curve $\Irr_n(\slope;i_1,i_2)$ or it sends it to its counterpart $\Irr_n(\slope;i_3,i_4)$. Now, by Theorem~\ref{thm:Conjugation:delta}, the number of irrational curves and their counterparts is the same in each $\delta$-grading. Also, the local systems of rational curves are either preserved or inverted, but again, by Theorem~\ref{thm:Conjugation:delta}, the number of each of these is the same in each grading. So $\HFTdelta(T)$ is preserved by mutation. 
	Alternatively, one may also study mutation about each mutation axis separately.
\end{proof}

%% file: sections/ConjBimod_bimod_part.tex
\begin{sidewaysfigure}[p]
	\vspace*{435pt}
	\centering
	\newlength{\RadAA}
	\setlength{\RadAA}{3.8cm}
	\newlength{\RadBB}
	\setlength{\RadBB}{7.5cm}
	\newlength{\RadCC}
	\setlength{\RadCC}{6.8cm}
	\newlength{\RadDD}
	\setlength{\RadDD}{4.5cm}
	\tikzstyle{tight}=[fill=white,inner sep=1pt,rounded corners]
	\tikzstyle{less tight}=[fill=white,inner sep=2pt,rounded corners]
	\begin{tikzpicture}
	
	\draw (20:1.5\RadBB)   node[less tight](Bbcd){$\prescript{}{b}{bcd}$};
	\draw (60:\RadBB)      node[less tight](Bacd){$\prescript{}{b}{acd}$};
	\draw (120:\RadBB)     node[less tight](Babd){$\prescript{}{b}{abd}$};
	\draw (160:1.5\RadBB)  node[less tight](Babc){$\prescript{}{b}{abc}$};
	\draw (-160:1.5\RadBB) node[less tight](Badc){$\prescript{}{b}{adc}$};
	\draw (-120:\RadBB)    node[less tight](Bbdc){$\prescript{}{b}{bdc}$};
	\draw (-60:\RadBB)     node[less tight](Bbac){$\prescript{}{b}{bac}$};
	\draw (-20:1.5\RadBB)  node[less tight](Bbad){$\prescript{}{b}{bad}$};
	
	\draw (20:1.5\RadCC)   node[less tight](Cdcb){$\prescript{}{c}{dcb}$};
	\draw (60:\RadCC)      node[less tight](Cdca){$\prescript{}{c}{dca}$};
	\draw (120:\RadCC)     node[less tight](Cdba){$\prescript{}{c}{dba}$};
	\draw (160:1.5\RadCC)  node[less tight](Ccba){$\prescript{}{c}{cba}$};
	\draw (-160:1.5\RadCC) node[less tight](Ccda){$\prescript{}{c}{cda}$};
	\draw (-120:\RadCC)    node[less tight](Ccdb){$\prescript{}{c}{cdb}$};
	\draw (-60:\RadCC)     node[less tight](Ccab){$\prescript{}{c}{cab}$};
	\draw (-20:1.5\RadCC)  node[less tight](Cdab){$\prescript{}{c}{dab}$};
	
	\draw (20:1.5\RadAA)   node[less tight](Aacd){$\prescript{}{a}{acd}$};
	\draw (60:\RadAA)      node[less tight](Abcd){$\prescript{}{a}{bcd}$};
	\draw (120:\RadAA)     node[less tight](Abca){$\prescript{}{a}{bca}$};
	\draw (160:1.5\RadAA)  node[less tight](Abda){$\prescript{}{a}{bda}$};
	\draw (-160:1.5\RadAA) node[less tight](Abdc){$\prescript{}{a}{bdc}$};
	\draw (-120:\RadAA)    node[less tight](Aadc){$\prescript{}{a}{adc}$};
	\draw (-60:\RadAA)     node[less tight](Aadb){$\prescript{}{a}{adb}$};
	\draw (-20:1.5\RadAA)  node[less tight](Aacb){$\prescript{}{a}{acb}$};
	
	\draw (20:1.5\RadDD)   node[less tight](Dcad){$\prescript{}{d}{cad}$};
	\draw (60:\RadDD)      node[less tight](Dcbd){$\prescript{}{d}{cbd}$};
	\draw (120:\RadDD)     node[less tight](Dcba){$\prescript{}{d}{cba}$};
	\draw (160:1.5\RadDD)  node[less tight](Ddba){$\prescript{}{d}{dba}$};
	\draw (-160:1.5\RadDD) node[less tight](Ddbc){$\prescript{}{d}{dbc}$};
	\draw (-120:\RadDD)    node[less tight](Ddac){$\prescript{}{d}{dac}$};
	\draw (-60:\RadDD)     node[less tight](Ddab){$\prescript{}{d}{dab}$};
	\draw (-20:1.5\RadDD)  node[less tight](Dcab){$\prescript{}{d}{cab}$};
	
	\draw[->,very thick] (Cdcb) to [bend right=10] node[tight]{$\{\textcolor{blue}{Q_{2}}\}$}   (Cdca);
	\draw[->,very thick] (Ccdb) to [bend left=10] node[tight]{$\{\textcolor{blue}{Q_{2}}\}$}   (Ccda);
	\draw[->,very thick] (Bbcd) to [bend right=10] node[tight]{$\{\square_2,\textcolor{blue}{Q_{2}}\}$}   (Bacd);
	\draw[->,very thick] (Bbdc) to [bend left=10] node[tight]{$\{\square_2,\textcolor{blue}{Q_{2}}\}$}   (Badc);
	
	\draw[->,very thick] (Babd) to [bend right=10] node[tight]{$\{\textcolor{blue}{Q_{4}}\}$}   (Babc);
	\draw[->,very thick] (Bbad) to [bend left=10] node[tight]{$\{\textcolor{blue}{Q_{4}}\}$}   (Bbac);
	\draw[->,very thick] (Cdba) to [bend right=10] node[tight,pos=0.6]{$\{\square_3,\textcolor{blue}{Q_{4}}\}$}   (Ccba);
	\draw[->,very thick] (Cdab) to [bend left=10] node[tight]{$\{\square_3,\textcolor{blue}{Q_{4}}\}$}   (Ccab);
	
	\draw[->,very thick] (Abca) to [bend left=10] node[tight]{$\{\textcolor{blue}{Q_{1}}\}$}   (Abcd);
	\draw[->,very thick] (Dcba) to [bend left=10] node[tight]{$\{\textcolor{blue}{Q_{1}}\}$}   (Dcbd);
	
	\draw[->,very thick] (Aacb) to [bend right=10] node[tight,pos=0.4]{$\{\textcolor{blue}{Q_{2}},\textcolor{blue}{Q_{1}}\}$}   (Aacd);
	\draw[->,very thick] (Dcab) to [bend right=10] node[tight,pos=0.6]{$\{\textcolor{blue}{Q_{2}},\textcolor{blue}{Q_{1}}\}$}   (Dcad);
	
	\draw[->,very thick] (Abda) to [bend right=10] node[tight,pos=0.4]{$\{\textcolor{blue}{Q_{1}},\textcolor{blue}{Q_{4}}\}$}   (Abdc);
	\draw[->,very thick] (Ddba) to [bend right=10] node[tight,pos=0.6]{$\{\textcolor{blue}{Q_{1}},\textcolor{blue}{Q_{4}}\}$}   (Ddbc);
	
	\draw[->,very thick] (Aadb) to [bend left=10] node[tight]{$\{\textcolor{blue}{Q_{2}},\textcolor{blue}{Q_{1}},\textcolor{blue}{Q_{4}}\}$}   (Aadc);
	\draw[->,very thick] (Ddab) to [bend left=10] node[tight]{$\{\textcolor{blue}{Q_{2}},\textcolor{blue}{Q_{1}},\textcolor{blue}{Q_{4}}\}$}   (Ddac);
	
	\draw[<-,very thick] (Abda) to [bend left=10] node[tight]{$\{\textcolor{blue}{P_{4}}\}$}   (Abca);
	\draw[<-,very thick] (Aadb) to [bend right=10] node[tight]{$\{\textcolor{blue}{P_{4}}\}$}   (Aacb);
	\draw[<-,very thick] (Ddba) to [bend left=10] node[tight]{$\{\square_4,\textcolor{blue}{P_{4}}\}$}   (Dcba);
	\draw[<-,very thick] (Ddab) to [bend right=10] node[tight]{$\{\square_4,\textcolor{blue}{P_{4}}\}$}   (Dcab);
	
	\draw[<-,very thick] (Dcbd) to [bend left=10] node[tight]{$\{\textcolor{blue}{P_{2}}\}$}   (Dcad);
	\draw[<-,very thick] (Ddbc) to [bend right=10] node[tight]{$\{\textcolor{blue}{P_{2}}\}$}   (Ddac);
	\draw[<-,very thick] (Abcd) to [bend left=10] node[tight]{$\{\square_1,\textcolor{blue}{P_{2}}\}$}   (Aacd);
	\draw[<-,very thick] (Abdc) to [bend right=10] node[tight]{$\{\square_1,\textcolor{blue}{P_{2}}\}$}   (Aadc);
	
	\draw[<-,very thick] (Cdca) to [bend right=10] node[tight]{$\{\textcolor{blue}{P_{3}}\}$}   (Cdba);
	\draw[<-,very thick] (Bacd) to [bend right=10] node[tight]{$\{\textcolor{blue}{P_{3}}\}$}   (Babd);
	
	\draw[<-,very thick] (Ccda) to [bend left=10] node[tight,pos=0.4]{$\{\textcolor{blue}{P_{4}},\textcolor{blue}{P_{3}}\}$}   (Ccba);
	\draw[<-,very thick] (Badc) to [bend left=10] node[tight,pos=0.6]{$\{\textcolor{blue}{P_{4}},\textcolor{blue}{P_{3}}\}$}   (Babc);
	
	\draw[<-,very thick] (Cdcb) to [bend left=10] node[tight,pos=0.4]{$\{\textcolor{blue}{P_{3}},\textcolor{blue}{P_{2}}\}$}   (Cdab);
	\draw[<-,very thick] (Bbcd) to [bend left=10] node[tight,pos=0.6]{$\{\textcolor{blue}{P_{3}},\textcolor{blue}{P_{2}}\}$}   (Bbad);
	
	\draw[<-,very thick] (Ccdb) to [bend right=10] node[tight]{$\{\textcolor{blue}{P_{4}},\textcolor{blue}{P_{3}},\textcolor{blue}{P_{2}}\}$}   (Ccab);
	\draw[<-,very thick] (Bbdc) to [bend right=10] node[tight]{$\{\textcolor{blue}{P_{4}},\textcolor{blue}{P_{3}},\textcolor{blue}{P_{2}}\}$}   (Bbac);
	
	\draw[->] (Dcba) to [bend right=17] node[tight,pos=0.6]{$\{\textcolor{red}{p_{4}}\}$}   (Ccba);
	\draw[->] (Cdba) to node[tight,pos=0.25]{$\{\textcolor{red}{q_{4}}\}$}   (Ddba);
	\draw[->] (Babd) to [bend right=20] node[tight,pos=0.4]{$\{\textcolor{red}{q_{4}},\textcolor{red}{q_{3}}\}$}   (Ddba);
	\draw[->] (Abca) to [bend right=10] node[tight,pos=0.7]{$\{\textcolor{red}{p_{4}},\textcolor{red}{p_{1}}\}$}   (Ccba);
	
	\draw[->] (Bbcd) to [bend right=5] node[tight]{$\{\textcolor{red}{p_{2}}\}$}   (Abcd);
	\draw[->] (Bbcd) to [bend right=10] node[tight,pos=0.5]{$\{\textcolor{red}{p_{2}},\textcolor{red}{p_{1}}\}$}   (Dcbd);
	\draw[->] (Aacd) to [bend right=35] node[tight,pos=0.65]{$\{\textcolor{red}{q_{2}}\}$}   (Bacd);
	\draw[->] (Aacd) to [bend right=10] node[tight,pos=0.8]{$\{\textcolor{red}{q_{2}},\textcolor{red}{q_{3}}\}$}   (Cdca);
	
	\draw[->] (Cdab) to  node[tight,pos=0.35]{$\{\textcolor{red}{q_{4}}\}$}   (Ddab);
	\draw[->] (Bbad) to [bend left=10] node[tight,pos=0.45]{$\{\square_2,\textcolor{red}{q_{4}},\textcolor{red}{q_{3}}\}$}   (Ddab);
	\draw[->] (Dcab) to [bend left=30] node[tight,pos=0.15]{$\{\textcolor{red}{p_{4}}\}$}   (Ccab);
	\draw[->] (Aacb) to [bend left=20] node[tight,pos=0.75]{$\{\square_1,\textcolor{red}{p_{4}},\textcolor{red}{p_{1}}\}$}   (Ccab);
	
	\draw[->] (Bbdc) to [bend left=10] node[tight,pos=0.35]{$\{\textcolor{red}{p_{2}}\}$}   (Abdc);
	\draw[->] (Bbdc) to [bend left=20] node[tight,pos=0.85]{$\{\square_2,\textcolor{red}{p_{2}},\textcolor{red}{p_{1}}\}$}   (Ddbc);
	\draw[->] (Aadc) to [bend left=20] node[tight,pos=0.2]{$\{\textcolor{red}{q_{2}}\}$}   (Badc);
	\draw[->] (Aadc) to [bend left=10] node[tight,pos=0.8]{$\{\square_3,\textcolor{red}{q_{2}},\textcolor{red}{q_{3}}\}$}   (Ccda);
	
	\draw[->] (Babc) to [bend right=10] node[tight]{$\{\textcolor{red}{q_{4}},\textcolor{red}{q_{3}},\textcolor{blue}{Q_{1}}\}$}   (Ddbc);
	\draw[->] (Aacb) to [bend right=10] node[tight,pos=0.75]{$\{\textcolor{red}{q_{3}},\textcolor{red}{q_{2}},\textcolor{blue}{Q_{1}}\}$}   (Cdcb);
	
	\draw[<-] (Dcad) to [bend left=10] node[tight]{$\{\textcolor{red}{p_{2}},\textcolor{red}{p_{1}},\textcolor{blue}{P_{3}}\}$}   (Bbad);
	\draw[<-] (Ccda) to [bend left=10] node[tight,pos=0.25]{$\{\textcolor{red}{p_{1}},\textcolor{red}{p_{4}},\textcolor{blue}{P_{3}}\}$}   (Abda);
	
	\draw[->] (Aadb) .. controls  +(-0.5,-1) and +(0,1) ..  (-75:5.2cm) node[tight,align=center]%
	{$\{\square_1,\textcolor{red}{p_{1}},\textcolor{red}{p_{4}},\textcolor{blue}{P_{2}},\textcolor{blue}{P_{3}}\}$\\$+\{\square_3,\textcolor{red}{q_{3}},\textcolor{red}{q_{2}},\textcolor{blue}{Q_{4}},\textcolor{blue}{Q_{1}}\}$}%
	.. controls  +(0,-1.3) and +(2,0.5) .. (Ccdb);
	
	\draw[->] (Bbac) .. controls  +(-2,1) and +(0,-1.5) ..  (-115:5.5cm) node[tight,align=center]%
	{$\{\square_4,\textcolor{red}{p_{2}},\textcolor{red}{p_{1}},\textcolor{blue}{P_{4}},\textcolor{blue}{P_{3}}\}$\\$+\{\square_2,\textcolor{red}{q_{4}},\textcolor{red}{q_{3}},\textcolor{blue}{Q_{2}},\textcolor{blue}{Q_{1}}\}$}%
	.. controls  +(0,0.1) and +(0,-1) .. (Ddac);
	
	\draw[->] (Aacd) .. controls  +(-2,-1.5) and  +(-3,0.5) .. (Dcad) node[tight,pos=0.5]	{$\{\square_1,\textcolor{red}{p_{1}}\}$};
	\draw[->] (Aacb) .. controls  +(-2,1.5) and  +(-3,-0.5) .. (Dcab) node[tight,pos=0.5]	{$\{\square_1,\textcolor{red}{p_{1}}\}$};
	
	\draw[->] (Abda) .. controls  +(2,-1.5) and  +(3,0.5) .. (Ddba) node[tight,pos=0.5]	{$\{\square_4,\textcolor{red}{p_{1}}\}$};
	\draw[->] (Abdc) .. controls  +(2,1.5) and  +(3,-0.5) .. (Ddbc) node[tight,pos=0.5]	{$\{\square_4,\textcolor{red}{p_{1}}\}$};
	
	\draw[->] (Abcd) .. controls  +(-1,-1) and  +(-2.5,-0.5) .. (Dcbd) node[tight,pos=0.5]	{$\{\textcolor{red}{p_{1}}\}$};
	\draw[->] (Abca) .. controls  +(1,-1) and  +(2.5,-0.5) .. (Dcba) node[tight,pos=0.5]	{$\{\textcolor{red}{p_{1}}\}$};
	
	\draw[->] (Aadb) .. controls  +(-1,1) and  +(1,2) .. (Ddab) node[tight,pos=0.5]	{$\{\square_1,\square_4,\textcolor{red}{p_{1}}\}$};
	\draw[->] (Aadc) .. controls  +(1,1) and  +(-1,2) .. (Ddac) node[tight,pos=0.5]	{$\{\square_1,\square_4,\textcolor{red}{p_{1}}\}$};
	
	\draw[->] (Bbcd) .. controls  +(0,1.5) and  +(0,2) .. (Cdcb) node[tight,pos=0.5]	{$\{\square_2,\textcolor{red}{q_{3}}\}$};
	\draw[->] (Bbad) .. controls  +(0,-1.5) and  +(0,-2) .. (Cdab) node[tight,pos=0.5]	{$\{\square_2,\textcolor{red}{q_{3}}\}$};
	
	\draw[->] (Babc) .. controls  +(0,1.5) and  +(0,2) .. (Ccba) node[tight,pos=0.5]	{$\{\square_3,\textcolor{red}{q_{3}}\}$};
	\draw[->] (Badc) .. controls  +(0,-1.5) and  +(0,-2) .. (Ccda) node[tight,pos=0.5]	{$\{\square_3,\textcolor{red}{q_{3}}\}$};
	
	\draw[->] (Bacd) .. controls  +(-2,-0.5) and  +(-1,-1) .. (Cdca) node[tight,pos=0.5]	{$\{\textcolor{red}{q_{3}}\}$};
	\draw[->] (Babd) .. controls  +(2,-0.5) and  +(1,-1) .. (Cdba) node[tight,pos=0.5]	{$\{\textcolor{red}{q_{3}}\}$};
	
	\draw[->] (Bbdc) .. controls  +(-2.5,-0.75) and  +(-3,-0.25) .. (Ccdb) node[tight,pos=0.4]	{$\{\square_2,\square_3,\textcolor{red}{q_{3}}\}$};
	\draw[->] (Bbac) .. controls  +(2.5,-0.75) and  +(3,-0.25) .. (Ccab) node[tight,pos=0.4]	{$\{\square_2,\square_3,\textcolor{red}{q_{3}}\}$};
	\end{tikzpicture}
	\caption{The contributing domains between all generators from the Heegaard diagram in Figure~\ref{fig:ConjugationHD} which survive cancellation
	}\label{fig:ConjugationBimodule:partial}
\end{sidewaysfigure}
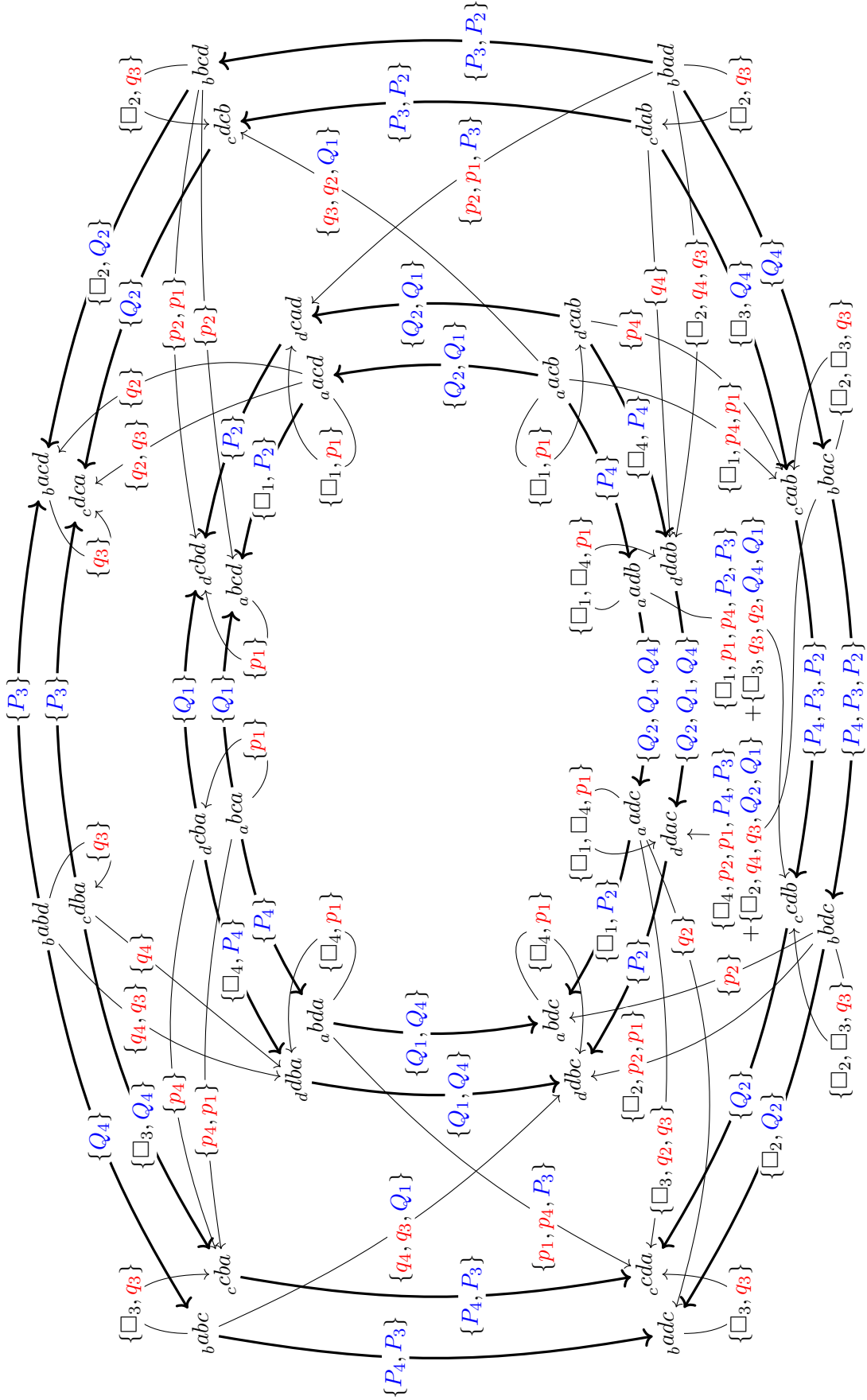

%% file: sections/ConjBimod_bimod_final.tex
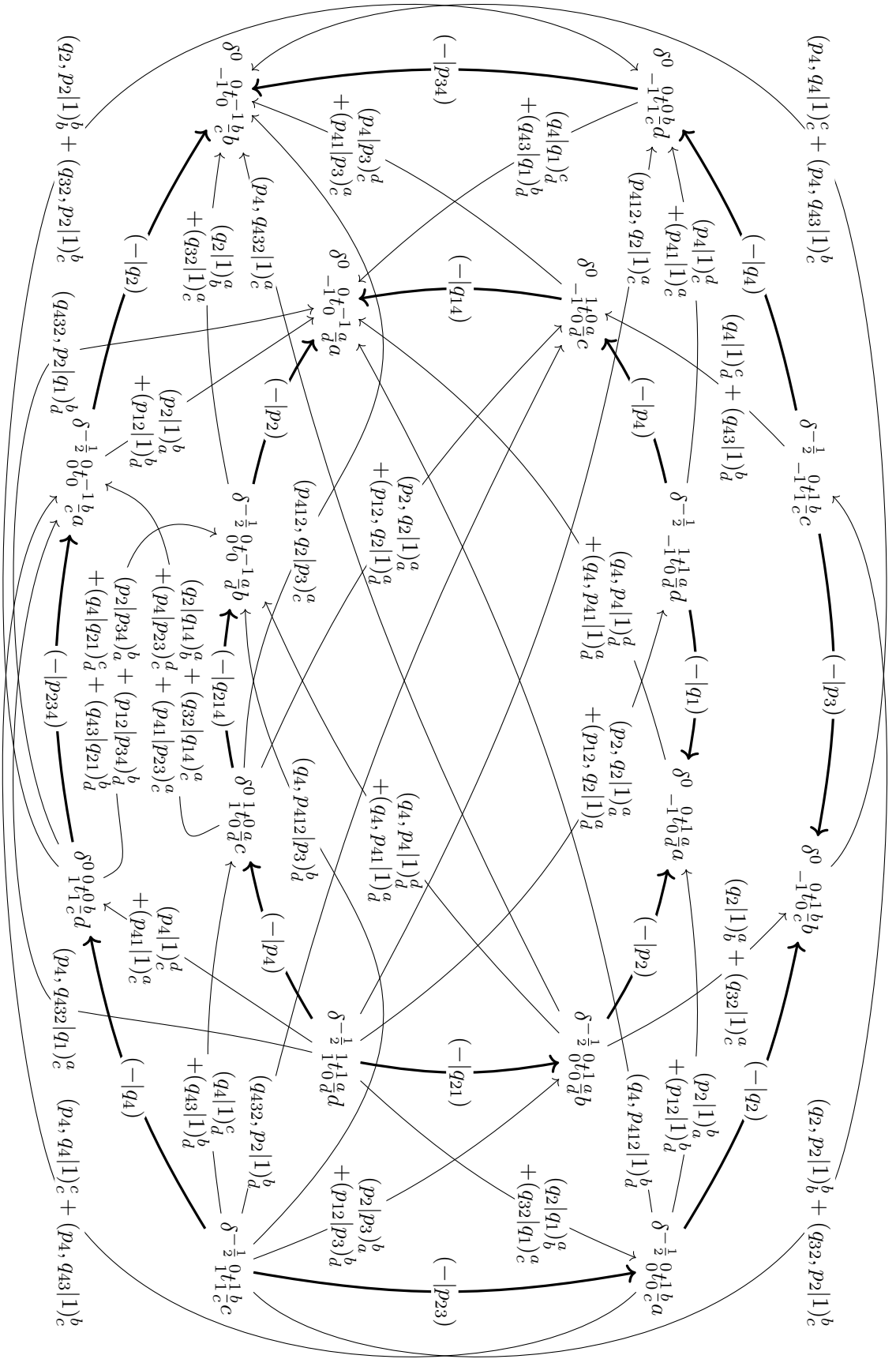
\begin{sidewaysfigure}[p]
	\vspace*{435pt}
	\centering
	\newlength{\RadAD}
	\setlength{\RadAD}{4.5cm}
	\newlength{\RadBC}
	\setlength{\RadBC}{7cm}
	
	\tikzstyle{tight}=[fill=white,inner sep=1pt,rounded corners]
	\begin{tikzpicture}
	
	\draw (20:1.5\RadBC)   node(Bar){$\delta^{-\frac{1}{2}}\Alex{0}{0}{0}{1}\tfrac{b}{c}a$};
	\draw (60:\RadBC)      node(Bbt){$\delta^{0}\Alex{0}{-1}{0}{1}\tfrac{b}{c}b$};
	\draw (120:\RadBC)     node(Bct){$\delta^{-\frac{1}{2}}\Alex{0}{-1}{1}{1}\tfrac{b}{c}c$};
	\draw (160:1.5\RadBC)  node(Bdl){$\delta^{0}\Alex{0}{-1}{1}{0}\tfrac{b}{c}d$};
	
	\draw (-160:1.5\RadBC) node(Bbl){$\delta^{0}\Alex{0}{-1}{0}{-1}\tfrac{b}{c}b$};
	\draw (-120:\RadBC)    node(Bab){$\delta^{-\frac{1}{2}}\Alex{0}{0}{0}{-1}\tfrac{b}{c}a$};
	\draw (-60:\RadBC)     node(Bdb){$\delta^{0}\Alex{0}{1}{1}{0}\tfrac{b}{c}d$};
	\draw (-20:1.5\RadBC)  node(Bcr){$\delta^{-\frac{1}{2}}\Alex{0}{1}{1}{1}\tfrac{b}{c}c$};
	
	\draw (20:1.5\RadAD)   node(Abr){$\delta^{-\frac{1}{2}}\Alex{0}{0}{0}{1}\tfrac{a}{d}b$};
	\draw (60:\RadAD)      node(Aat){$\delta^{0}\Alex{0}{-1}{0}{1}\tfrac{a}{d}a$};
	\draw (120:\RadAD)     node(Adt){$\delta^{-\frac{1}{2}}\Alex{1}{-1}{0}{1}\tfrac{a}{d}d$};
	\draw (160:1.5\RadAD)  node(Acl){$\delta^{0}\Alex{1}{-1}{0}{0}\tfrac{a}{d}c$};
	
	\draw (-160:1.5\RadAD)+(0,0.5) node(Aal){$\delta^{0}\Alex{0}{-1}{0}{-1}\tfrac{a}{d}a$};
	\draw (-120:\RadAD)+(0,0.5)    node(Abb){$\delta^{-\frac{1}{2}}\Alex{0}{0}{0}{-1}\tfrac{a}{d}b$};
	\draw (-60:\RadAD)+(0,0.5)     node(Acb){$\delta^{0}\Alex{1}{1}{0}{0}\tfrac{a}{d}c$};
	\draw (-20:1.5\RadAD)+(0,0.5)  node(Adr){$\delta^{-\frac{1}{2}}\Alex{1}{1}{0}{1}\tfrac{a}{d}d$};
	
	\draw[->] (Bar) to [bend right=10,pos=0.3] node[tight,align=center]{$(p_2|1)^b_a$\\$+(p_{12}|1)^b_d$}   (Aat);
	\draw[->] (Abr) to [bend right=10,pos=0.7] node[tight]             {$(q_2|1)^a_b+(q_{32}|1)^a_c$}       (Bbt);
	\draw[->] (Bct) to [bend right=10,pos=0.3] node[tight]             {$(q_{4}|1)^c_d+(q_{43}|1)^b_d$}     (Acl);
	\draw[->] (Adt) to [bend right=10,pos=0.7] node[tight,align=center]{$(p_{4}|1)^d_c$\\$+(p_{41}|1)^a_c$} (Bdl);
	\draw[->] (Adr) to [bend right=10,pos=0.7] node[tight,align=center]{$(q_{2}|q_1)^a_b$\\$+(q_{32}|q_1)^a_c$} (Bar);
	\draw[->] (Bcr) to [bend right=10,pos=0.3] node[tight,align=center]{$(p_{2}|p_3)^b_a$\\$+(p_{12}|p_3)^b_d$} (Abr);
	\draw[->] (Bdl) to [bend right=10,pos=0.3] node[tight,align=center]{$(q_{4}|q_1)^c_d$\\$+(q_{43}|q_1)^b_d$} (Aal);
	\draw[->] (Acl) to [bend right=10,pos=0.7] node[tight,align=center]{$(p_{4}|p_3)^d_c$\\$+(p_{41}|p_3)^a_c$} (Bbl);
	\draw[->] (Bcr) to [bend left=10,pos=0.3]  node[tight,align=center]{$(q_{4}|1)^c_d$\\$+(q_{43}|1)^b_d$} (Acb);
	\draw[->] (Adr) to [pos=0.75]              node[tight,align=center]{$(p_{4}|1)^d_c$\\$+(p_{41}|1)^a_c$} (Bdb);
	\draw[->] (Bab) to [pos=0.25]              node[tight,align=center]{$(p_{2}|1)^b_a$\\$+(p_{12}|1)^b_d$} (Aal);
	\draw[->] (Abb) to [bend left=10,pos=0.7]  node[tight,align=center]{$(q_{2}|1)^a_b$\\$+(q_{32}|1)^a_c$} (Bbl);
	\draw[->] (Bdb) .. controls  +(-0.5,1) and +(3,0) ..  (-90:5.55cm) node[tight,align=center]%
	{$(p_{2}|p_{34})^b_a+(p_{12}|p_{34})^b_d$\\$+(q_{4}|q_{21})^c_d+(q_{43}|q_{21})^b_d$}%
	.. controls  +(-3,0) and +(-0.5,-1) .. (Abb);
	\draw[->] (Acb) .. controls  +(0.5,-1) and +(3,0) ..  (-90:4.45cm) node[tight,align=center]%
	{$(q_{2}|q_{14})^a_b+(q_{32}|q_{14})^a_c$\\$+(p_{4}|p_{23})^d_c+(p_{41}|p_{23})^a_c$}%
	.. controls  +(-3,0) and +(0.5,1) .. (Bab);
	
	\draw[->] (Aat) to [bend right=20,pos=0.3] node[tight,align=center]{$(q_4,p_4|1)^d_d$\\$+(q_4,p_{41}|1)^a_d$}   (Aal);
	\draw[->] (Abr) to [bend left=10,pos=0.48]  node[tight,align=center]{$(q_4,p_4|1)^d_d$\\$+(q_4,p_{41}|1)^a_d$}    (Abb);
	\draw[->] (Acb) to [bend left=10,pos=0.52]  node[tight,align=center]{$(p_2,q_2|1)^a_a$\\$+(p_{12},q_2|1)^a_d$}    (Acl);
	\draw[->] (Adr) to [bend right=20,pos=0.7] node[tight,align=center]{$(p_2,q_2|1)^a_a$\\$+(p_{12},q_2|1)^a_d$}   (Adt);
	\draw[<-] (Bab) .. controls  +(2,-2) and +(-1.5,-1) .. (-35:10.9cm) node[tight]%
	{$(p_4,q_4|1)^c_c+(p_{4},q_{43}|1)^b_c$}  .. controls  +(3,2) and +(2,-2) .. (Bar);
	\draw[<-] (Bbl) .. controls  +(-2,2) and +(-3,-2) .. (145:10.9cm) node[tight]%
	{$(p_4,q_4|1)^c_c+(p_{4},q_{43}|1)^b_c$}  .. controls  +(1.5,1) and +(-2,2) .. (Bbt);
	\draw[->] (Bcr) .. controls  +(2,2) and +(3,-2) .. (35:10.9cm) node[tight]%
	{$(q_{2},p_2|1)^b_b+(q_{32},p_2|1)^b_c$}  .. controls  +(-1.5,1) and +(2,2) .. (Bct);
	\draw[->] (Bdb) .. controls  +(-2,-2) and +(1.5,-1) .. (-145:10.9cm) node[tight]%
	{$(q_{2},p_2|1)^b_b+(q_{32},p_2|1)^b_c$}  .. controls  +(-3,2) and +(-2,-2) .. (Bdl);
	
	\draw[->] (Adr) to [bend right=10,pos=0.92] node[tight]{$(p_{412},q_2|1)^a_c$}   (Bdl);
	\draw[<-] (Acl) to [bend right=7,pos=0.92] node[tight]{$(q_{432},p_2|1)^b_d$}   (Bcr);
	\draw[->] (Abr) to [bend left=7,pos=0.92]  node[tight]{$(p_4,q_{432}|1)^a_c$}   (Bbl);
	\draw[<-] (Aal) to [bend left=10,pos=0.92]  node[tight]{$(q_4,p_{412}|1)^b_d$}   (Bar);
	
	\draw[->] (Adr) .. controls  +(-0.5,-2) and +(0.2,1) .. (-49:8.4cm) node[tight]%
	{$(p_4,q_{432}|q_1)^a_c$} .. controls  +(-0.2,-1) and +(3,-1.5) ..   (Bab);
	\draw[->] (Bdb) .. controls  +(-3,-1.5) and +(0.2,-1) .. (-131:8.4cm) node[tight]%
	{$(q_{432},p_2|q_1)^b_d$} .. controls  +(-0.2,1) and +(0.5,-2) .. (Aal);
	
	\draw[->] (Acb) .. controls  +(-1.5,0.1) and +(2,-1) .. (-135:3.3cm) node[tight]%
	{$(p_{412},q_2|p_3)^a_c$} .. controls  +(-2,1) and +(4,4) .. (Bbl);
	\draw[->] (Bcr) .. controls  +(-4,4) and +(2,1) .. (-45:3.3cm) node[tight]%
	{$(q_{4},p_{412}|p_3)^b_d$} .. controls  +(-2,-1) and +(1.5,0.1) .. (Abb);
	
	\draw[->,very thick] (Bar) to [bend right=10] node[tight]{$(-|q_{2})$}   (Bbt);
	\draw[->,very thick] (Bct) to [bend left=10] node[tight]{$(-|p_{3})$}   (Bbt);
	\draw[->,very thick] (Bct) to [bend right=10] node[tight]{$(-|q_{4})$}   (Bdl);
	\draw[->,very thick] (Bdl) to [bend right=10] node[tight]{$(-|p_{34})$}  (Bbl);
	\draw[->,very thick] (Bab) to [bend left=10] node[tight]{$(-|q_{2})$}   (Bbl);
	\draw[->,very thick] (Bdb) to [bend left=10] node[tight]{$(-|p_{234})$} (Bab);
	\draw[->,very thick] (Bcr) to [bend left=10] node[tight]{$(-|q_{4})$}   (Bdb);
	\draw[->,very thick] (Bcr) to [bend right=10] node[tight]{$(-|p_{23})$}  (Bar);
	
	\draw[->,very thick] (Abr) to [bend right=10] node[tight]{$(-|p_{2})$}   (Aat);
	\draw[->,very thick] (Adt) to [bend left=10] node[tight]{$(-|q_{1})$}   (Aat);
	\draw[->,very thick] (Adt) to [bend right=10] node[tight]{$(-|p_{4})$}   (Acl);
	\draw[->,very thick] (Acl) to [bend right=10] node[tight]{$(-|q_{14})$}  (Aal);
	\draw[->,very thick] (Abb) to [bend left=10] node[tight]{$(-|p_{2})$}   (Aal);
	\draw[->,very thick] (Acb) to [bend left=10] node[tight]{$(-|q_{214})$} (Abb);
	\draw[->,very thick] (Adr) to [bend left=10] node[tight]{$(-|p_{4})$}   (Acb);
	\draw[->,very thick] (Adr) to [bend right=10] node[tight]{$(-|q_{21})$}  (Abr);
	
	\end{tikzpicture}
	\vspace*{-10pt}
	\caption{The bimodule~$\ConjBimod$ from Definition~\ref{def:ConjBimod}. There are also self-loops 
		$
		\delta^{0}\Alex{0}{0}{0}{0}bx
		\xrightarrow{(q_3|1)}
		\delta^{\frac{1}{2}}\Alex{0}{0}{-1}{0}cx
		$ 
		and 
		$
		\delta^{0}\Alex{0}{0}{0}{0}ax
		\xrightarrow{(p_1|1)}
		\delta^{\frac{1}{2}}\Alex{-1}{0}{0}{0}dx
		$
		for any $x\in\{a,b,c,d\}$.
	}\label{fig:ConjugationBimodule}
\end{sidewaysfigure}

%% file: sections/ConjBimod_down_raw.tex
\begin{sidewaysfigure}[p]
	\vspace*{435pt}
	\centering
	\newlength{\RadCl}
	\setlength{\RadCl}{2.5cm}
	\newlength{\RadAl}
	\setlength{\RadAl}{3.5cm}
	\newlength{\RadB}
	\setlength{\RadB}{4.5cm}
	\newlength{\RadCr}
	\setlength{\RadCr}{6cm}
	\newlength{\RadAr}
	\setlength{\RadAr}{7cm}
	\tikzstyle{tight}=[fill=white,inner sep=3pt,rounded corners]
	\begin{tikzpicture}
	
	\draw [line width=20pt,lightgray,line cap=round] (20:1.5\RadB) to (20:1.5\RadCr);
	\draw [line width=20pt,lightgray,line cap=round] (60:\RadB) to (60:\RadCr);
	\draw [line width=20pt,lightgray,line cap=round] (120:\RadB) to (120:\RadCr);
	\draw [line width=20pt,lightgray,line cap=round] (160:1.5\RadB) to (160:1.5\RadCr);
	
	\draw [line width=20pt,lightgray,line cap=round] (-20:1.5\RadB) to (-20:1.5\RadCr);
	\draw [line width=20pt,lightgray,line cap=round] (-60:\RadB) to (-60:\RadCr);
	\draw [line width=20pt,lightgray,line cap=round] (-120:\RadB) to (-120:\RadCr);
	\draw [line width=20pt,lightgray,line cap=round] (-160:1.5\RadB) to (-160:1.5\RadCr);
	
	\node (Albrnew) at (20:1.5\RadAl){};
	\node (Clbtnew) at (60:\RadCl){};
	\draw [line width=20pt,lightgray,line cap=round] (Albrnew) to[bend right=10] (Clbtnew);
	\node (Albbnew) at (-120:\RadAl){};
	\node (Clblnew) at (-160:1.5\RadCl){};
	\draw [line width=20pt,lightgray,line cap=round] (Albbnew) to[bend left=10] (Clblnew);
		
	\draw (20:1.5\RadB)   node[fill=lightgray,rounded corners](Bar){$\delta^{-1}\Alex{0}{0}{0}{1}a$};
	\draw (60:\RadB)      node[fill=lightgray,rounded corners](Bbt){$\delta^{-\frac{1}{2}}\Alex{0}{-1}{0}{1}b$};
	\draw (120:\RadB)     node[fill=lightgray,rounded corners](Bct){$\delta^{-1}\Alex{0}{-1}{1}{1}c$};
	\draw (160:1.5\RadB)  node[fill=lightgray,rounded corners](Bdl){$\delta^{-\frac{1}{2}}\Alex{0}{-1}{1}{0}d$};
	
	\draw (-160:1.5\RadB) node[fill=lightgray,rounded corners](Bbl){$\delta^{-\frac{1}{2}}\Alex{0}{-1}{0}{-1}b$};
	\draw (-120:\RadB)    node[fill=lightgray,rounded corners](Bab){$\delta^{-1}\Alex{0}{0}{0}{-1}a$};
	\draw (-60:\RadB)     node[fill=lightgray,rounded corners](Bdb){$\delta^{-\frac{1}{2}}\Alex{0}{1}{1}{0}d$};
	\draw (-20:1.5\RadB)  node[fill=lightgray,rounded corners](Bcr){$\delta^{-1}\Alex{0}{1}{1}{1}c$};
	
	\draw (20:1.5\RadCr)   node[fill=lightgray,rounded corners](Crar){$\delta^{0}\Alex{0}{0}{0}{1}a$};
	\draw (60:\RadCr)      node[fill=lightgray,rounded corners](Crbt){$\delta^{\frac{1}{2}}\Alex{0}{-1}{0}{1}b$};
	\draw (120:\RadCr)     node[fill=lightgray,rounded corners](Crct){$\delta^{0}\Alex{0}{-1}{1}{1}c$};
	\draw (160:1.5\RadCr)  node[fill=lightgray,rounded corners](Crdl){$\delta^{\frac{1}{2}}\Alex{0}{-1}{1}{0}d$};
	
	\draw (-160:1.5\RadCr) node[fill=lightgray,rounded corners](Crbl){$\delta^{\frac{1}{2}}\Alex{0}{-1}{0}{-1}b$};
	\draw (-120:\RadCr)    node[fill=lightgray,rounded corners](Crab){$\delta^{0}\Alex{0}{0}{0}{-1}a$};
	\draw (-60:\RadCr)     node[fill=lightgray,rounded corners](Crdb){$\delta^{\frac{1}{2}}\Alex{0}{1}{1}{0}d$};
	\draw (-20:1.5\RadCr)  node[fill=lightgray,rounded corners](Crcr){$\delta^{0}\Alex{0}{1}{1}{1}c$};
	
	\draw (20:1.5\RadCl)   node(Clar){$\delta^{0}\Alex{0}{2}{0}{1}a$};
	\draw (60:\RadCl)      node[fill=lightgray,rounded corners](Clbt){$\delta^{\frac{1}{2}}\Alex{0}{1}{0}{1}b$};
	\draw (120:\RadCl)     node(Clct){$\delta^{0}\Alex{0}{1}{1}{1}c$};
	\draw (160:1.5\RadCl)  node(Cldl){$\delta^{\frac{1}{2}}\Alex{0}{1}{1}{0}d$};
	
	\draw (-160:1.5\RadCl) node[fill=lightgray,rounded corners](Clbl){$\delta^{\frac{1}{2}}\Alex{0}{1}{0}{-1}b$};
	\draw (-120:\RadCl)    node(Clab){$\delta^{0}\Alex{0}{2}{0}{-1}a$};
	\draw (-60:\RadCl)     node(Cldb){$\delta^{\frac{1}{2}}\Alex{0}{3}{1}{0}d$};
	\draw (-20:1.5\RadCl)  node(Clcr){$\delta^{0}\Alex{0}{3}{1}{1}c$};
	
	\draw (20:1.5\RadAr)   node(Arbr){$\delta^{-\frac{1}{2}}\Alex{0}{1}{2}{1}b$};
	\draw (60:\RadAr)      node(Arat){$\delta^{0}\Alex{0}{0}{2}{1}a$};
	\draw (120:\RadAr)     node(Ardt){$\delta^{-\frac{1}{2}}\Alex{1}{0}{2}{1}d$};
	\draw (160:1.5\RadAr)  node(Arcl){$\delta^{0}\Alex{1}{0}{2}{0}c$};
	
	\draw (-160:1.5\RadAr) node(Aral){$\delta^{0}\Alex{0}{0}{2}{-1}a$};
	\draw (-120:\RadAr)    node(Arbb){$\delta^{-\frac{1}{2}}\Alex{0}{1}{2}{-1}b$};
	\draw (-60:\RadAr)     node(Arcb){$\delta^{0}\Alex{1}{2}{2}{0}c$};
	\draw (-20:1.5\RadAr)  node(Ardr){$\delta^{-\frac{1}{2}}\Alex{1}{2}{2}{1}d$};
	
	\draw (20:1.5\RadAl)   node[fill=lightgray,rounded corners](Albr){$\delta^{-\frac{1}{2}}\Alex{0}{1}{0}{1}b$};
	\draw (60:\RadAl)      node(Alat){$\delta^{0}\Alex{0}{0}{0}{1}a$};
	\draw (120:\RadAl)     node(Aldt){$\delta^{-\frac{1}{2}}\Alex{1}{0}{0}{1}d$};
	\draw (160:1.5\RadAl)  node(Alcl){$\delta^{0}\Alex{1}{0}{0}{0}c$};
	
	\draw (-160:1.5\RadAl) node(Alal){$\delta^{0}\Alex{0}{0}{0}{-1}a$};
	\draw (-120:\RadAl)    node[fill=lightgray,rounded corners](Albb){$\delta^{-\frac{1}{2}}\Alex{0}{1}{0}{-1}b$};
	\draw (-60:\RadAl)     node(Alcb){$\delta^{0}\Alex{1}{2}{0}{0}c$};
	\draw (-20:1.5\RadAl)  node(Aldr){$\delta^{-\frac{1}{2}}\Alex{1}{2}{0}{1}d$};
	
	\draw[->] (Albr) to [bend right=10] node[fill=lightgray,rounded corners]{$1$}      (Clbt);
	\draw[->] (Aldr) to [bend right=10] node[tight]{$q_1$}    (Clar);
	\draw[->] (Albb) to [bend left=10] node[fill=lightgray,rounded corners]{$1$}      (Clbl);
	\draw[->] (Alcb) to [bend left=10] node[tight]{$q_{14}$} (Clab);
	
	\draw[->] (Bar) to [bend right=10] node[tight]{$1$}      (Alat);
	\draw[->] (Bcr) to [bend right=10] node[tight]{$p_3$}    (Albr);
	\draw[->] (Bab) to [bend left=10] node[tight]{$1$}      (Alal);
	\draw[->] (Bdb) to [bend left=10] node[tight]{$p_{34}$} (Albb);

	\draw[->] (Bar) to node[below]{$1$} (Crar);
	\draw[->] (Bbt) to node[left]{$1$} (Crbt);
	\draw[->] (Bct) to node[right]{$1$} (Crct);
	\draw[->] (Bdl) to node[below]{$1$} (Crdl);
	\draw[->] (Bbl) to node[above]{$1$} (Crbl);
	\draw[->] (Bab) to node[right]{$1$} (Crab);
	\draw[->] (Bdb) to node[left]{$1$} (Crdb);
	\draw[->] (Bcr) to node[above]{$1$} (Crcr);
	
	\draw[->] (Ardt) to [bend right=10] node[tight]{$1$}      (Crdl);
	\draw[->] (Arcl) to [bend right=10] node[tight]{$p_3$}    (Crbl);
	\draw[->] (Ardr) to [bend left=10] node[tight]{$1$}      (Crdb);
	\draw[->] (Arcb) to [bend left=10] node[tight]{$p_{23}$} (Crab);
	
	\draw[->,dashed] (Ardr) to [bend right=10] node[tight]{$q_1$} +(1,3);
	\draw[->,dashed] (Arbr) to [bend right=10] node[tight]{$1$} +(-3,2.5);
	\draw[->,dashed] (Arcb) to [bend left=10] node[tight]{$q_{14}$} +(-3,-1);
	\draw[->,dashed] (Arbb) to [bend left=10] node[tight]{$1$} +(-3,0);
	
	\draw[<-,dashed] (Clbl) to [bend left=10] node[tight]{$p_3$} +(1,2);
	\draw[<-,dashed] (Cldl) to [bend left=10] node[tight]{$1$} +(3,0);
	\draw[<-,dashed] (Clab) to [bend right=10] node[tight]{$p_{23}$} +(2,1);
	\draw[<-,dashed] (Cldb) to [bend right=10] node[tight]{$1$} +(1.5,1.5);
	
	\draw[->,very thick] (Bar) to [bend right=10] node[tight]{$q_{2}$}   (Bbt);
	\draw[->,very thick] (Bct) to [bend left=10] node[tight]{$p_{3}$}    (Bbt);
	\draw[->,very thick] (Bct) to [bend right=10] node[tight]{$q_{4}$}   (Bdl);
	\draw[->,very thick] (Bdl) to [bend right=10] node[tight]{$p_{34}$}  (Bbl);
	\draw[->,very thick] (Bab) to [bend left=10] node[tight]{$q_{2}$}    (Bbl);
	\draw[->,very thick] (Bdb) to [bend left=10] node[tight]{$p_{234}$}  (Bab);
	\draw[->,very thick] (Bcr) to [bend left=10] node[tight]{$q_{4}$}    (Bdb);
	\draw[->,very thick] (Bcr) to [bend right=10] node[tight]{$p_{23}$}  (Bar);
	
	\draw[->,very thick] (Crar) to [bend right=10] node[tight]{$q_{2}$}   (Crbt);
	\draw[->,very thick] (Crct) to [bend left=10] node[tight]{$p_{3}$}    (Crbt);
	\draw[->,very thick] (Crct) to [bend right=10] node[tight]{$q_{4}$}   (Crdl);
	\draw[->,very thick] (Crdl) to [bend right=10] node[tight]{$p_{34}$}  (Crbl);
	\draw[->,very thick] (Crab) to [bend left=10] node[tight]{$q_{2}$}    (Crbl);
	\draw[->,very thick] (Crdb) to [bend left=10] node[tight]{$p_{234}$}  (Crab);
	\draw[->,very thick] (Crcr) to [bend left=10] node[tight]{$q_{4}$}    (Crdb);
	\draw[->,very thick] (Crcr) to [bend right=10] node[tight]{$p_{23}$}  (Crar);
	
	\draw[->,very thick] (Clar) to [bend right=10] node[below]{$q_{2}$}   (Clbt);
	\draw[->,very thick] (Clct) to [bend left=10]  node[below]{$p_{3}$}    (Clbt);
	\draw[->,very thick] (Clct) to [bend right=10] node[below]{~~$q_{4}$}   (Cldl);
	\draw[->,very thick] (Cldl) to [bend right=10] node[tight]{$p_{34}$}  (Clbl);
	\draw[->,very thick] (Clab) to [bend left=10]  node[above]{~~$q_{2}$}    (Clbl);
	\draw[->,very thick] (Cldb) to [bend left=10]  node[above]{$p_{234}$}  (Clab);
	\draw[->,very thick] (Clcr) to [bend left=10]  node[above]{$q_{4}$}    (Cldb);
	\draw[->,very thick] (Clcr) to [bend right=10] node[tight]{$p_{23}$}  (Clar);
	
	\draw[->,very thick] (Arbr) to [bend right=10] node[tight]{$p_{2}$}   (Arat);
	\draw[->,very thick] (Ardt) to [bend left=10] node[tight]{$q_{1}$}    (Arat);
	\draw[->,very thick] (Ardt) to [bend right=10] node[tight]{$p_{4}$}   (Arcl);
	\draw[->,very thick] (Arcl) to [bend right=10] node[tight]{$q_{14}$}  (Aral);
	\draw[->,very thick] (Arbb) to [bend left=10] node[tight]{$p_{2}$}    (Aral);
	\draw[->,very thick] (Arcb) to [bend left=10] node[tight]{$q_{214}$}  (Arbb);
	\draw[->,very thick] (Ardr) to [bend left=10] node[tight]{$p_{4}$}    (Arcb);
	\draw[->,very thick] (Ardr) to [bend right=10] node[tight]{$q_{21}$}  (Arbr);
	
	\draw[->,very thick] (Albr) to [bend right=10] node[tight]{$p_{2}$}   (Alat);
	\draw[->,very thick] (Aldt) to [bend left=10] node[tight]{$q_{1}$}    (Alat);
	\draw[->,very thick] (Aldt) to [bend right=10] node[tight]{$p_{4}$}   (Alcl);
	\draw[->,very thick] (Alcl) to [bend right=10] node[tight]{$q_{14}$}  (Alal);
	\draw[->,very thick] (Albb) to [bend left=10] node[tight]{$p_{2}$}    (Alal);
	\draw[->,very thick] (Alcb) to [bend left=10] node[tight]{$q_{214}$}  (Albb);
	\draw[->,very thick] (Aldr) to [bend left=10] node[tight]{$p_{4}$}    (Alcb);
	\draw[->,very thick] (Aldr) to [bend right=10] node[tight]{$q_{21}$}  (Albr);
	
	\draw[->] (Bcr) to [in=-50,out=130] node[tight]{$1$} (Clct);
	\draw[->] (Bdb) to [in=-45,out=65] node[tight]{$1$} (Cldl);
	
	\end{tikzpicture}
	\\
	\hrulefill
	$$
	\text{(inner)}\qquad
	\rr\Big(
	\begin{tikzcd}[row sep=0.5cm, column sep=0.6cm]
	\cdots
	\arrow[dashed]{r}{p_{41}}
	&
	\delta^{0}\Alex{0}{-2}{-1}{0}c
	&
	\delta^{0}\Alex{0}{-1}{0}{0}a
	\arrow[swap]{l}{q_{32}}
	&
	\delta^{-\frac{1}{2}}\Alex{0}{0}{0}{0}b
	\arrow{r}{q_{3}}
	\arrow[swap]{l}{p_{2}}
	&
	\delta^{0}\Alex{0}{0}{-1}{0}c
	&
	\delta^{0}\Alex{0}{-1}{-2}{0}a
	\arrow[swap]{l}{p_{41}}
	\arrow[dashed]{r}{q_{32}}
	&
	\cdots
	\end{tikzcd}
	\Big)\boxtimes\ConjBimod
	\qquad\text{(outer)}
	$$
	\caption{The computation of the action of the bimodule~$\ConjBimod$ on the first curve segment in the proof of Lemma~\ref{lem:ConjBimod:ActionOnHorizontalComponents}}\label{fig:ConjugationBimodule:down:raw}
\end{sidewaysfigure}

%% file: sections/ConjBimod_down_cancelled.tex
\begin{sidewaysfigure}[p]
	\vspace*{435pt}
	\centering
	\setlength{\RadCl}{3cm}
	\setlength{\RadAl}{4.3cm}
	\setlength{\RadB}{5cm}
	\setlength{\RadCr}{6cm}
	\setlength{\RadAr}{7cm}
	
	\tikzstyle{nodeTLO}=[draw,ellipse,inner sep=1pt]
	\tikzstyle{nodeTLI}=[draw,ellipse,inner sep=1pt,dotted]
	\tikzstyle{nodeBRI}=[draw,rectangle,rounded corners]
	\tikzstyle{nodeBRO}=[draw,rectangle,rounded corners,dotted]
	
	\tikzstyle{tight}=[fill=white,inner sep=3pt,rounded corners]
	
	\begin{tikzpicture}
	
	\draw (20:1.5\RadCl)   node[nodeBRI](Clar){$\delta^{0}\Alex{0}{2}{0}{1}a$};
	\draw (120:\RadCl)     node[nodeTLI](Clct){$\delta^{0}\Alex{0}{1}{1}{1}c$};
	\draw (160:1.5\RadCl)  node(Cldl){$\delta^{\frac{1}{2}}\Alex{0}{1}{1}{0}d$};
	
	\draw (-120:\RadCl)    node[nodeBRI](Clab){$\delta^{0}\Alex{0}{2}{0}{-1}a$};
	\draw (-60:\RadCl)     node(Cldb){$\delta^{\frac{1}{2}}\Alex{0}{3}{1}{0}d$};
	\draw (-20:1.5\RadCl)  node[nodeBRI](Clcr){$\delta^{0}\Alex{0}{3}{1}{1}c$};
	
	\draw (20:1.5\RadAr)   node(Arbr){$\delta^{-\frac{1}{2}}\Alex{0}{1}{2}{1}b$};
	\draw (60:\RadAr)      node[nodeTLO](Arat){$\delta^{0}\Alex{0}{0}{2}{1}a$};
	\draw (120:\RadAr)     node[nodeTLO](Ardt){$\delta^{-\frac{1}{2}}\Alex{1}{0}{2}{1}d$};
	\draw (160:1.5\RadAr)  node[nodeTLO](Arcl){$\delta^{0}\Alex{1}{0}{2}{0}c$};
	
	\draw (-160:1.5\RadAr) node[nodeTLO](Aral){$\delta^{0}\Alex{0}{0}{2}{-1}a$};
	\draw (-120:\RadAr)    node(Arbb){$\delta^{-\frac{1}{2}}\Alex{0}{1}{2}{-1}b$};
	\draw (-60:\RadAr)     node[nodeBRO](Arcb){$\delta^{0}\Alex{1}{2}{2}{0}c$};
	\draw (-20:1.5\RadAr)  node[nodeBRO](Ardr){$\delta^{-\frac{1}{2}}\Alex{1}{2}{2}{1}d$};
	
	\draw (60:\RadAl)      node[nodeTLI](Alat){$\delta^{0}\Alex{0}{0}{0}{1}a$};
	\draw (120:\RadAl)     node[nodeTLI](Aldt){$\delta^{-\frac{1}{2}}\Alex{1}{0}{0}{1}d$};
	\draw (160:1.5\RadAl)  node[nodeTLI](Alcl){$\delta^{0}\Alex{1}{0}{0}{0}c$};
	
	\draw (-160:1.5\RadAl) node[nodeTLI](Alal){$\delta^{0}\Alex{0}{0}{0}{-1}a$};
	\draw (-60:\RadAl)     node[nodeBRI](Alcb){$\delta^{0}\Alex{1}{2}{0}{0}c$};
	\draw (-20:1.5\RadAl)  node[nodeBRI](Aldr){$\delta^{-\frac{1}{2}}\Alex{1}{2}{0}{1}d$};

	\draw[->] (Aldr) to [bend right=10] node[tight]{$q_1$}    (Clar);
	\draw[->] (Alcb) to [bend left=10] node[tight]{$q_{14}$} (Clab);
	
	\draw[->,dashed] (Ardr) to [bend right=10] node[tight]{$q_1$} +(1,3);
	\draw[->,dashed] (Arbr) to [bend right=10] node[tight]{$1$} +(-3,2.5);
	\draw[->,dashed] (Arcb) to [bend left=10] node[tight]{$q_{14}$} +(-3,-1);
	\draw[->,dashed] (Arbb) to [bend left=10] node[tight]{$1$} +(-3,0);
	
	\draw[<-,dashed] (Cldl) to [bend left=10] node[tight]{$1$} +(3,0);
	\draw[<-,dashed] (Clab) to [bend right=10] node[tight]{$p_{23}$} +(2,1);
	\draw[<-,dashed] (Cldb) to [bend right=10] node[tight]{$1$} +(1.5,1.5);

	\draw[->] (Ardr) to [bend right=10] node[tight]{$1$}      (Cldl);
	
	\draw[->] (Arcb) to [bend left=20] node[tight]{$p_{23}$} (Alal);
	
	\draw[->] (Clct) to [bend right=10] node[tight]{$p_{23}$}    (Alat);

	\draw[<-,dashed] (Alal) to [bend right=10] node[tight]{$p_{23}$} +(3,2);
	\draw[->] (Cldl) to [bend left=10] node[tight]{$p_{234}$}  (Alal);
	
	%

	\draw[->] (Clct) to [bend right=10] node[tight]{$q_{4}$}   (Cldl);
	\draw[->] (Cldb) to [bend left=10] node[tight]{$p_{234}$}  (Clab);
	\draw[->] (Clcr) to [bend left=10] node[tight]{$q_{4}$}    (Cldb);
	\draw[->] (Clcr) to [bend right=10] node[tight]{$p_{23}$}  (Clar);	
	\draw[->] (Arbr) to [bend right=10] node[tight]{$p_{2}$}   (Arat);
	\draw[->] (Ardt) to [bend left=10] node[tight]{$q_{1}$}    (Arat);
	\draw[->] (Ardt) to [bend right=10] node[tight]{$p_{4}$}   (Arcl);
	\draw[->] (Arcl) to [bend right=10] node[tight]{$q_{14}$}  (Aral);
	\draw[->] (Arbb) to [bend left=10] node[tight]{$p_{2}$}    (Aral);
	\draw[->] (Arcb) to [bend left=10] node[tight]{$q_{214}$}  (Arbb);
	\draw[->] (Ardr) to [bend left=10] node[tight]{$p_{4}$}    (Arcb);
	\draw[->] (Ardr) to [bend right=10] node[tight]{$q_{21}$}  (Arbr);

	\draw[->] (Aldt) to [bend left=10] node[tight]{$q_{1}$}    (Alat);
	\draw[->] (Aldt) to [bend right=10] node[tight]{$p_{4}$}   (Alcl);
	\draw[->] (Alcl) to [bend right=10] node[tight]{$q_{14}$}  (Alal);
	\draw[->] (Aldr) to [bend left=10] node[tight]{$p_{4}$}    (Alcb);

	\end{tikzpicture}
	\\
	\hrulefill
	$$
	\text{(inner)}\qquad
	\rr\Big(
	\begin{tikzcd}[row sep=0.5cm, column sep=0.6cm]
	\cdots
	\arrow[dashed]{r}{p_{41}}
	&
	\delta^{0}\Alex{0}{-2}{-1}{0}c
	&
	\delta^{0}\Alex{0}{-1}{0}{0}a
	\arrow[swap]{l}{q_{32}}
	&
	\delta^{-\frac{1}{2}}\Alex{0}{0}{0}{0}b
	\arrow{r}{q_{3}}
	\arrow[swap]{l}{p_{2}}
	&
	\delta^{0}\Alex{0}{0}{-1}{0}c
	&
	\delta^{0}\Alex{0}{-1}{-2}{0}a
	\arrow[swap]{l}{p_{41}}
	\arrow[dashed]{r}{q_{32}}
	&
	\cdots
	\end{tikzcd}
	\Big)\boxtimes\ConjBimod
	\qquad\text{(outer)}
	$$
	\caption{The result of cancelling all shaded identity arrows in Figure~\ref{fig:ConjugationBimodule:down:raw}}\label{fig:ConjugationBimodule:down:cancelled}
\end{sidewaysfigure}
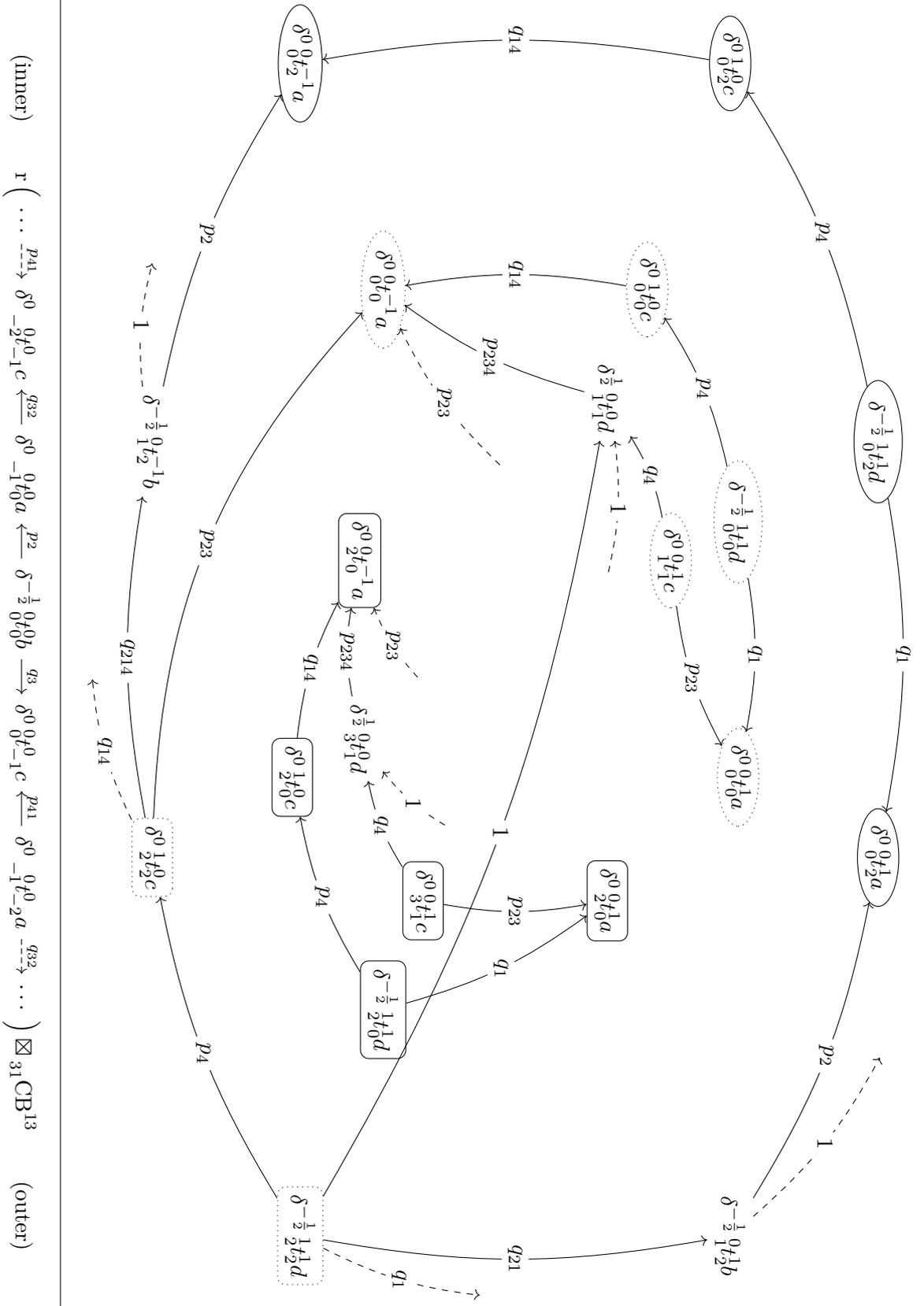

%% file: sections/ConjBimod_up_raw.tex
\begin{sidewaysfigure}[p]
	\vspace*{435pt}
	\centering
	\setlength{\RadAr}{2.5cm}
	\setlength{\RadCr}{3.5cm}
	\setlength{\RadB}{5cm}
	\setlength{\RadAl}{6cm}
	\setlength{\RadCl}{7cm}
	
	\tikzstyle{tight}=[fill=white,inner sep=3pt,rounded corners]
	
	\begin{tikzpicture}
	
	\node (Aldtnew) at (120:\RadAl){};
	\node (Cldlnew) at (160:1.5\RadCl){};
	\draw [line width=20pt,lightgray,line cap=round] (Aldtnew) to[bend right=10] (Cldlnew);
	\node (Aldrnew) at (-20:1.5\RadAl){};
	\node (Cldbnew) at (-60:\RadCl){};
	\draw [line width=20pt,lightgray,line cap=round] (Aldrnew) to[bend left=10] (Cldbnew);
	\node (Albrnew) at (20:1.5\RadAl){};
	\node (Bbtnew) at (60:\RadB){};
	\draw [line width=20pt,lightgray,line cap=round] (Albrnew) to[bend right=10] (Bbtnew);
	\node (Albbnew) at (-120:\RadAl){};
	\node (Bblnew) at (-160:1.5\RadB){};
	\draw [line width=20pt,lightgray,line cap=round] (Albbnew) to[bend left=10] (Bblnew);
	\node (Arbrnew) at (20:1.5\RadAr){};
	\node (Crbtnew) at (60:\RadCr){};
	\draw [line width=20pt,lightgray,line cap=round] (Arbrnew) to[bend right=10] (Crbtnew);
	\node (Arbbnew) at (-120:\RadAr){};
	\node (Crblnew) at (-160:1.5\RadCr){};
	\draw [line width=20pt,lightgray,line cap=round] (Arbbnew) to[bend left=10] (Crblnew);
%
	\draw (20:1.5\RadB)   node(Bar){$\delta^{0}\Alex{0}{0}{0}{1}a$};
	\draw (60:\RadB)      node[fill=lightgray,rounded corners](Bbt){$\delta^{\frac{1}{2}}\Alex{0}{-1}{0}{1}b$};
	\draw (120:\RadB)     node(Bct){$\delta^{0}\Alex{0}{-1}{1}{1}c$};
	\draw (160:1.5\RadB)  node(Bdl){$\delta^{\frac{1}{2}}\Alex{0}{-1}{1}{0}d$};
	
	\draw (-160:1.5\RadB) node[fill=lightgray,rounded corners](Bbl){$\delta^{\frac{1}{2}}\Alex{0}{-1}{0}{-1}b$};
	\draw (-120:\RadB)    node(Bab){$\delta^{0}\Alex{0}{0}{0}{-1}a$};
	\draw (-60:\RadB)     node(Bdb){$\delta^{\frac{1}{2}}\Alex{0}{1}{1}{0}d$};
	\draw (-20:1.5\RadB)  node(Bcr){$\delta^{0}\Alex{0}{1}{1}{1}c$};
	
	\draw (20:1.5\RadCr)   node(Crar){$\delta^{0}\Alex{0}{0}{-2}{1}a$};
	\draw (60:\RadCr)      node[fill=lightgray,rounded corners](Crbt){$\delta^{\frac{1}{2}}\Alex{0}{-1}{-2}{1}b$};
	\draw (120:\RadCr)     node(Crct){$\delta^{0}\Alex{0}{-1}{-1}{1}c$};
	\draw (160:1.5\RadCr)  node(Crdl){$\delta^{\frac{1}{2}}\Alex{0}{-1}{-1}{0}d$};
	
	\draw (-160:1.5\RadCr) node[fill=lightgray,rounded corners](Crbl){$\delta^{\frac{1}{2}}\Alex{0}{-1}{-2}{-1}b$};
	\draw (-120:\RadCr)    node(Crab){$\delta^{0}\Alex{0}{0}{-2}{-1}a$};
	\draw (-60:\RadCr)     node(Crdb){$\delta^{\frac{1}{2}}\Alex{0}{1}{-1}{0}d$};
	\draw (-20:1.5\RadCr)  node(Crcr){$\delta^{0}\Alex{0}{1}{-1}{1}c$};
	
	\draw (20:1.5\RadCl)   node(Clar){$\delta^{0}\Alex{0}{-2}{-2}{1}a$};
	\draw (60:\RadCl)      node(Clbt){$\delta^{\frac{1}{2}}\Alex{0}{-3}{-2}{1}b$};
	\draw (120:\RadCl)     node(Clct){$\delta^{0}\Alex{0}{-3}{-1}{1}c$};
	\draw (160:1.5\RadCl)  node[fill=lightgray,rounded corners](Cldl){$\delta^{\frac{1}{2}}\Alex{0}{-3}{-1}{0}d$};
	
	\draw (-160:1.5\RadCl) node(Clbl){$\delta^{\frac{1}{2}}\Alex{0}{-3}{-2}{-1}b$};
	\draw (-120:\RadCl)    node(Clab){$\delta^{0}\Alex{0}{-2}{-2}{-1}a$};
	\draw (-60:\RadCl)     node[fill=lightgray,rounded corners](Cldb){$\delta^{\frac{1}{2}}\Alex{0}{-1}{-1}{0}d$};
	\draw (-20:1.5\RadCl)  node(Clcr){$\delta^{0}\Alex{0}{-1}{-1}{1}c$};
	
	\draw (20:1.5\RadAr)   node[fill=lightgray,rounded corners](Arbr){$\delta^{-\frac{1}{2}}\Alex{0}{-1}{-2}{1}b$};
	\draw (60:\RadAr)      node(Arat){$\delta^{0}\Alex{0}{-2}{-2}{1}a$};
	\draw (120:\RadAr)     node(Ardt){$\delta^{-\frac{1}{2}}\Alex{1}{-2}{-2}{1}d$};
	\draw (160:1.5\RadAr)  node(Arcl){$\delta^{0}\Alex{1}{-2}{-2}{0}c$};
	
	\draw (-160:1.5\RadAr) node(Aral){$\delta^{0}\Alex{0}{-2}{-2}{-1}a$};
	\draw (-120:\RadAr)    node[fill=lightgray,rounded corners](Arbb){$\delta^{-\frac{1}{2}}\Alex{0}{-1}{-2}{-1}b$};
	\draw (-60:\RadAr)     node(Arcb){$\delta^{0}\Alex{1}{0}{-2}{0}c$};
	\draw (-20:1.5\RadAr)  node(Ardr){$\delta^{-\frac{1}{2}}\Alex{1}{0}{-2}{1}d$};
	
	\draw (20:1.5\RadAl)   node[fill=lightgray,rounded corners](Albr){$\delta^{-\frac{1}{2}}\Alex{0}{-1}{0}{1}b$};
	\draw (60:\RadAl)      node(Alat){$\delta^{0}\Alex{0}{-2}{0}{1}a$};
	\draw (120:\RadAl)     node[fill=lightgray,rounded corners](Aldt){$\delta^{-\frac{1}{2}}\Alex{1}{-2}{0}{1}d$};
	\draw (160:1.5\RadAl)  node(Alcl){$\delta^{0}\Alex{1}{-2}{0}{0}c$};
	
	\draw (-160:1.5\RadAl) node(Alal){$\delta^{0}\Alex{0}{-2}{0}{-1}a$};
	\draw (-120:\RadAl)    node[fill=lightgray,rounded corners](Albb){$\delta^{-\frac{1}{2}}\Alex{0}{-1}{0}{-1}b$};
	\draw (-60:\RadAl)     node(Alcb){$\delta^{0}\Alex{1}{0}{0}{0}c$};
	\draw (-20:1.5\RadAl)  node[fill=lightgray,rounded corners](Aldr){$\delta^{-\frac{1}{2}}\Alex{1}{0}{0}{1}d$};

	\draw[->] (Aldt) to [bend right=10] node[fill=lightgray,rounded corners]{$1$}      (Cldl);
	\draw[->] (Alcl) to [bend right=10] node[tight,align=center]{$p_3$}    (Clbl);
	\draw[->] (Aldr) to [bend left=10] node[fill=lightgray,rounded corners]{$1$}      (Cldb);
	\draw[->] (Alcb) to [bend left=10] node[tight,align=center]{$p_{23}$} (Clab);
	
	\draw[->] (Albr) to [bend right=10] node[fill=lightgray,rounded corners]{$1$}      (Bbt);
	\draw[->] (Aldr) to [bend right=10] node[tight]{$q_1$}    (Bar);
	\draw[->] (Albb) to [bend left=10]  node[fill=lightgray,rounded corners]{$1$}      (Bbl);
	\draw[->] (Alcb) to [bend left=10]  node[tight]{$q_{14}$} (Bab);
	
	
	\draw[->] (Aldr) to [in=-80,out=120,pos=0.2] node[tight]{$1$}   (Crdl);
	\draw[->] (Alcb) to [in=-50,out=155] node[tight]{$p_3$} (Crbl);
	
	\draw[->] (Arbr) to [bend right=10] node[fill=lightgray,rounded corners]{$1$}      (Crbt);
	\draw[->] (Ardr) to [bend right=10] node[tight]{$q_1$}    (Crar);
	\draw[->] (Arbb) to [bend left=10]  node[fill=lightgray,rounded corners]{$1$}      (Crbl);
	\draw[->] (Arcb) to [bend left=10]  node[tight]{$q_{14}$} (Crab);
	
	\draw[->,dashed] (Ardr) to [bend left=10] node[tight]{$1$} +(-2.5,0.4);
	\draw[->,dashed] (Arcb) to [bend left=10] node[tight]{$p_{23}$} +(-2.5,1.2);
	\draw[->,dashed] (Ardt) to [bend right=10] node[tight]{$1$} +(-1,-1.5);
	\draw[->,dashed] (Arcl) to [bend right=10] node[tight]{$p_3$} +(1,-2);
	
	\draw[<-,dashed] (Clbl) to [bend right=10] node[tight]{$1$} +(3,-2.5);
	\draw[<-,dashed] (Clab) to [bend right=10] node[tight]{$q_{14}$} +(3,-1);
	\draw[<-,dashed] (Clbt) to [bend left=10] node[tight]{$1$} +(4,0);
	\draw[<-,dashed] (Clar) to [bend left=10] node[tight]{$q_1$} +(1,-3);
	
	\draw[->,very thick] (Bar) to [bend right=10] node[tight]{$q_{2}$}   (Bbt);
	\draw[->,very thick] (Bct) to [bend left=10] node[tight]{$p_{3}$}    (Bbt);
	\draw[->,very thick] (Bct) to [bend right=10] node[tight]{$q_{4}$}   (Bdl);
	\draw[->,very thick] (Bdl) to [bend right=10] node[tight]{$p_{34}$}  (Bbl);
	\draw[->,very thick] (Bab) to [bend left=10] node[tight]{$q_{2}$}    (Bbl);
	\draw[->,very thick] (Bdb) to [bend left=10] node[tight]{$p_{234}$}  (Bab);
	\draw[->,very thick] (Bcr) to [bend left=10] node[tight]{$q_{4}$}    (Bdb);
	\draw[->,very thick] (Bcr) to [bend right=10] node[tight]{$p_{23}$}  (Bar);
	
	\draw[->,very thick] (Crar) to [bend right=10] node[tight]{$q_{2}$}   (Crbt);
	\draw[->,very thick] (Crct) to [bend left=10] node[tight]{$p_{3}$}    (Crbt);
	\draw[->,very thick] (Crct) to [bend right=10] node[tight]{$q_{4}$}   (Crdl);
	\draw[->,very thick] (Crdl) to [bend right=10] node[tight]{$p_{34}$}  (Crbl);
	\draw[->,very thick] (Crab) to [bend left=10] node[tight]{$q_{2}$}    (Crbl);
	\draw[->,very thick] (Crdb) to [bend left=10] node[tight]{$p_{234}$}  (Crab);
	\draw[->,very thick] (Crcr) to [bend left=10] node[tight]{$q_{4}$}    (Crdb);
	\draw[->,very thick] (Crcr) to [bend right=10] node[tight]{$p_{23}$}  (Crar);
	
	\draw[->,very thick] (Clar) to [bend right=10] node[tight]{$q_{2}~$}   (Clbt);
	\draw[->,very thick] (Clct) to [bend left=10]  node[tight]{$p_{3}$}    (Clbt);
	\draw[->,very thick] (Clct) to [bend right=10] node[tight]{$q_{4}$}   (Cldl);
	\draw[->,very thick] (Cldl) to [bend right=10] node[tight]{$p_{34}$}  (Clbl);
	\draw[->,very thick] (Clab) to [bend left=10]  node[tight]{$q_{2}$}    (Clbl);
	\draw[->,very thick] (Cldb) to [bend left=10]  node[tight]{$p_{234}$}  (Clab);
	\draw[->,very thick] (Clcr) to [bend left=10]  node[tight]{$q_{4}~$}    (Cldb);
	\draw[->,very thick] (Clcr) to [bend right=10] node[tight]{$p_{23}$}  (Clar);
	
	\draw[->,very thick] (Arbr) to [bend right=10] node[below]{$p_{2}$}   (Arat);
	\draw[->,very thick] (Ardt) to [bend left=10]  node[below]{$q_{1}$}    (Arat);
	\draw[->,very thick] (Ardt) to [bend right=10] node[below]{~~$p_{4}$}   (Arcl);
	\draw[->,very thick] (Arcl) to [bend right=10] node[tight]{$q_{14}$}  (Aral);
	\draw[->,very thick] (Arbb) to [bend left=10]  node[above]{~~$p_{2}$}    (Aral);
	\draw[->,very thick] (Arcb) to [bend left=10]  node[above]{$q_{214}$}  (Arbb);
	\draw[->,very thick] (Ardr) to [bend left=10]  node[above]{$p_{4}$}    (Arcb);
	\draw[->,very thick] (Ardr) to [bend right=10] node[tight]{$q_{21}$}  (Arbr);
	
	\draw[->,very thick] (Albr) to [bend right=10] node[tight]{$p_{2}$}   (Alat);
	\draw[->,very thick] (Aldt) to [bend left=10] node[tight]{$q_{1}$}    (Alat);
	\draw[->,very thick] (Aldt) to [bend right=10] node[tight]{$p_{4}$}   (Alcl);
	\draw[->,very thick] (Alcl) to [bend right=10] node[tight]{$q_{14}$}  (Alal);
	\draw[->,very thick] (Albb) to [bend left=10] node[tight]{$p_{2}$}    (Alal);
	\draw[->,very thick] (Alcb) to [bend left=10] node[tight]{$q_{214}$}  (Albb);
	\draw[->,very thick] (Aldr) to [bend left=10] node[tight]{$p_{4}$}    (Alcb);
	\draw[->,very thick] (Aldr) to [bend right=10] node[tight]{$q_{21}$}  (Albr);
	
	\end{tikzpicture}
	\\
	\hrulefill
	$$
	\text{(inner)}\qquad
	\rr\Big(
	\begin{tikzcd}[row sep=0.5cm, column sep=0.6cm]
	\cdots
	&
	\delta^{0}\Alex{0}{1}{2}{0}a
	\arrow{r}{q_{32}}
	\arrow[dashed,swap]{l}{p_{41}}
	&
	\delta^{0}\Alex{0}{0}{1}{0}c
	&
	\delta^{\frac{1}{2}}\Alex{0}{0}{0}{0}b
	\arrow[swap]{l}{p_{412}}
	&
	\delta^{0}\Alex{0}{1}{0}{0}a
	\arrow[swap]{l}{q_{2}}
	\arrow{r}{p_{41}}
	&
	\delta^{0}\Alex{0}{2}{1}{0}c
	&
	\cdots
	\arrow[dashed,swap]{l}{q_{32}}
	\end{tikzcd}
	\Big)\boxtimes\ConjBimod
	\qquad\text{(outer)}
	$$
	\caption{The computation of the action of the bimodule~$\ConjBimod$ on the second curve segment in the proof of Lemma~\ref{lem:ConjBimod:ActionOnHorizontalComponents}}\label{fig:ConjugationBimodule:up:raw}
\end{sidewaysfigure}

%% file: sections/ConjBimod_up_cancelled.tex
\begin{sidewaysfigure}[p]
	\vspace*{435pt}
	\centering
	\setlength{\RadCl}{7cm}
	\setlength{\RadAl}{6cm}
	\setlength{\RadB}{4.75cm}
	\setlength{\RadCr}{3.5cm}
	\setlength{\RadAr}{2.5cm}
	
	\tikzstyle{nodeTLO}=[draw,ellipse,inner sep=1pt]
	\tikzstyle{nodeTLI}=[draw,ellipse,inner sep=1pt,dotted]
	\tikzstyle{nodeBRI}=[draw,rectangle,rounded corners]
	\tikzstyle{nodeBRO}=[draw,rectangle,rounded corners,dotted]
	
	\tikzstyle{tight}=[fill=white,inner sep=3pt,rounded corners]
	
	\begin{tikzpicture}
	
	\draw (20:1.5\RadB)   node[nodeBRO](Bar){$\delta^{0}\Alex{0}{0}{0}{1}a$};
	\draw (120:\RadB)     node[nodeTLO](Bct){$\delta^{0}\Alex{0}{-1}{1}{1}c$};
	\draw (160:1.5\RadB)  node[nodeTLO](Bdl){$\delta^{\frac{1}{2}}\Alex{0}{-1}{1}{0}d$};
	
	\draw (-120:\RadB)    node[nodeBRO](Bab){$\delta^{0}\Alex{0}{0}{0}{-1}a$};
	\draw (-60:\RadB)     node[nodeBRO](Bdb){$\delta^{\frac{1}{2}}\Alex{0}{1}{1}{0}d$};
	\draw (-20:1.5\RadB)  node[nodeBRO](Bcr){$\delta^{0}\Alex{0}{1}{1}{1}c$};
	
	\draw (20:1.5\RadCr)   node[nodeBRI](Crar){$\delta^{0}\Alex{0}{0}{-2}{1}a$};
	\draw (120:\RadCr)     node[nodeTLI](Crct){$\delta^{0}\Alex{0}{-1}{-1}{1}c$};
	\draw (160:1.5\RadCr)  node[nodeTLI](Crdl){$\delta^{\frac{1}{2}}\Alex{0}{-1}{-1}{0}d$};
	
	\draw (-120:\RadCr)    node[nodeBRI](Crab){$\delta^{0}\Alex{0}{0}{-2}{-1}a$};
	\draw (-60:\RadCr)     node[nodeBRI](Crdb){$\delta^{\frac{1}{2}}\Alex{0}{1}{-1}{0}d$};
	\draw (-20:1.5\RadCr)  node[nodeBRI](Crcr){$\delta^{0}\Alex{0}{1}{-1}{1}c$};
	
	\draw (20:1.5\RadCl)   node[nodeBRO](Clar){$\delta^{0}\Alex{0}{-2}{-2}{1}a$};
	\draw (60:\RadCl)      node(Clbt){$\delta^{\frac{1}{2}}\Alex{0}{-3}{-2}{1}b$};
	\draw (120:\RadCl)     node[nodeTLO](Clct){$\delta^{0}\Alex{0}{-3}{-1}{1}c$};
	
	\draw (-160:1.5\RadCl) node(Clbl){$\delta^{\frac{1}{2}}\Alex{0}{-3}{-2}{-1}b$};
	\draw (-120:\RadCl)    node[nodeBRO](Clab){$\delta^{0}\Alex{0}{-2}{-2}{-1}a$};
	\draw (-20:1.5\RadCl)  node[nodeBRO](Clcr){$\delta^{0}\Alex{0}{-1}{-1}{1}c$};
	
	\draw (60:\RadAr)      node[nodeTLI](Arat){$\delta^{0}\Alex{0}{-2}{-2}{1}a$};
	\draw (120:\RadAr)     node(Ardt){$\delta^{-\frac{1}{2}}\Alex{1}{-2}{-2}{1}d$};
	\draw (160:1.5\RadAr)  node[nodeTLI](Arcl){$\delta^{0}\Alex{1}{-2}{-2}{0}c$};
	
	\draw (-160:1.5\RadAr) node[nodeTLI](Aral){$\delta^{0}\Alex{0}{-2}{-2}{-1}a$};
	\draw (-60:\RadAr)     node[nodeBRI](Arcb){$\delta^{0}\Alex{1}{0}{-2}{0}c$};
	\draw (-20:1.5\RadAr)  node(Ardr){$\delta^{-\frac{1}{2}}\Alex{1}{0}{-2}{1}d$};
	
	\draw (60:\RadAl)      node[nodeTLO](Alat){$\delta^{0}\Alex{0}{-2}{0}{1}a$};
	\draw (160:1.5\RadAl)  node[nodeTLO](Alcl){$\delta^{0}\Alex{1}{-2}{0}{0}c$};
	
	\draw (-160:1.5\RadAl) node[nodeTLO](Alal){$\delta^{0}\Alex{0}{-2}{0}{-1}a$};
	\draw (-60:\RadAl)     node[nodeBRO](Alcb){$\delta^{0}\Alex{1}{0}{0}{0}c$};
	\draw[->] (Alcl) to [bend right=10] node[tight,align=center]{$p_3$}    (Clbl);
	\draw[->] (Alcb) to [bend left=10] node[tight,align=center]{$p_{23}$} (Clab);
	
	\draw[->] (Alcb) to [bend left=10]  node[tight]{$q_{14}$} (Bab);
	
	
	
	\draw[->] (Ardr) to [bend right=10] node[tight]{$q_1$}    (Crar);
	\draw[->] (Arcb) to [bend left=10]  node[tight]{$q_{14}$} (Crab);
	
	\draw[->,dashed] (Ardr) to [bend left=10] node[tight]{$1$} +(-2.5,0.4);
	\draw[->,dashed] (Arcb) to [bend left=10] node[tight]{$p_{23}$} +(-2.5,1.2);
	\draw[->,dashed] (Ardt) to [bend right=10] node[tight]{$1$} +(-1,-1.5);
	\draw[->,dashed] (Arcl) to [bend right=10] node[tight]{$p_3$} +(1,-2);
	
	\draw[<-,dashed] (Clbl) to [bend right=10] node[tight]{$1$} +(3,-2.5);
	\draw[<-,dashed] (Clab) to [bend right=10] node[tight]{$q_{14}$} +(3,-1);
	\draw[<-,dashed] (Clbt) to [bend left=10] node[tight]{$1$} +(4,0);
	\draw[<-,dashed] (Clar) to [bend left=10] node[tight]{$q_1$} +(1,-3);
	
	\draw[->] (Clct) to [bend left=10] node[tight]{$q_{14}$}    (Alat);
	
	\draw[->] (Bct) to [bend left=10] node[tight]{$p_{23}$}    (Alat);
		
	\draw[->] (Crct) to [bend left=10] node[tight]{$p_{23}$}    (Arat);		
	
	\draw[->] (Bdl) to [bend right=10] node[tight]{$p_{234}$}  (Alal);
	
	\draw[->] (Alcb) .. controls  +(-2,-0.2) and +(0,-1) .. (-135:6cm) node[tight]	{$p_{23}$}  .. controls  +(0,1) and +(0,-2) .. (Aral);
	\draw[->] (Crdl) to [bend right=10] node[tight]{$p_{234}$}    (Aral);
	
	\draw[->] (Clcr) to [bend right=10]  node[tight]{$q_{14}$}    (Bar);
	\draw[->] (Clcr) .. controls  +(-1,2) and +(4,0) .. (0,-0.5) node[tight]	{$q_4$}  .. controls  +(-4,0) and +(-0.2,-3) .. (Crdl);
		
	
	\draw[->] (Bct) to [bend right=10] node[tight]{$q_{4}$}   (Bdl);
	\draw[->] (Bdb) to [bend left=10] node[tight]{$p_{234}$}  (Bab);
	\draw[->] (Bcr) to [bend left=10] node[tight]{$q_{4}$}    (Bdb);
	\draw[->] (Bcr) to [bend right=10] node[tight]{$p_{23}$}  (Bar);

	\draw[->] (Crct) to [bend right=10] node[tight]{$q_{4}$}   (Crdl);
	\draw[->] (Crdb) to [bend left=10] node[tight]{$p_{234}$}  (Crab);
	\draw[->] (Crcr) to [bend left=10] node[tight]{$q_{4}$}    (Crdb);
	\draw[->] (Crcr) to [bend right=10] node[tight]{$p_{23}$}  (Crar);
	
	\draw[->] (Clar) to [bend right=10] node[tight]{$q_{2}~$}   (Clbt);
	\draw[->] (Clct) to [bend left=10]  node[tight]{$p_{3}$}    (Clbt);
	\draw[->] (Clab) to [bend left=10]  node[tight]{$q_{2}$}    (Clbl);
	\draw[->] (Clcr) to [bend right=10] node[tight]{$p_{23}$}  (Clar);

	\draw[->] (Ardt) to [bend left=10] node[below]{$q_{1}$}    (Arat);
	\draw[->] (Ardt) to [bend right=10] node[below]{~~$p_{4}$}   (Arcl);
	\draw[->] (Arcl) to [bend right=10] node[tight]{$q_{14}$}  (Aral);
	\draw[->] (Ardr) to [bend left=10] node[above]{$p_{4}$}    (Arcb);

	\draw[->] (Alcl) to [bend right=10] node[tight]{$q_{14}$}  (Alal);
	
	\end{tikzpicture}
	\\
	\hrulefill
	$$
	\text{(inner)}\qquad
	\rr\Big(
	\begin{tikzcd}[row sep=0.5cm, column sep=0.6cm]
	\cdots
	&
	\delta^{0}\Alex{0}{1}{2}{0}a
	\arrow{r}{q_{32}}
	\arrow[dashed,swap]{l}{p_{41}}
	&
	\delta^{0}\Alex{0}{0}{1}{0}c
	&
	\delta^{\frac{1}{2}}\Alex{0}{0}{0}{0}b
	\arrow[swap]{l}{p_{412}}
	&
	\delta^{0}\Alex{0}{1}{0}{0}a
	\arrow[swap]{l}{q_{2}}
	\arrow{r}{p_{41}}
	&
	\delta^{0}\Alex{0}{2}{1}{0}c
	&
	\cdots
	\arrow[dashed,swap]{l}{q_{32}}
	\end{tikzcd}
	\Big)\boxtimes\ConjBimod
	\qquad\text{(outer)}
	$$
	\caption{The result of cancelling all shaded identity arrows in Figure~\ref{fig:ConjugationBimodule:up:raw}}\label{fig:ConjugationBimodule:up:cancelled}
\end{sidewaysfigure}
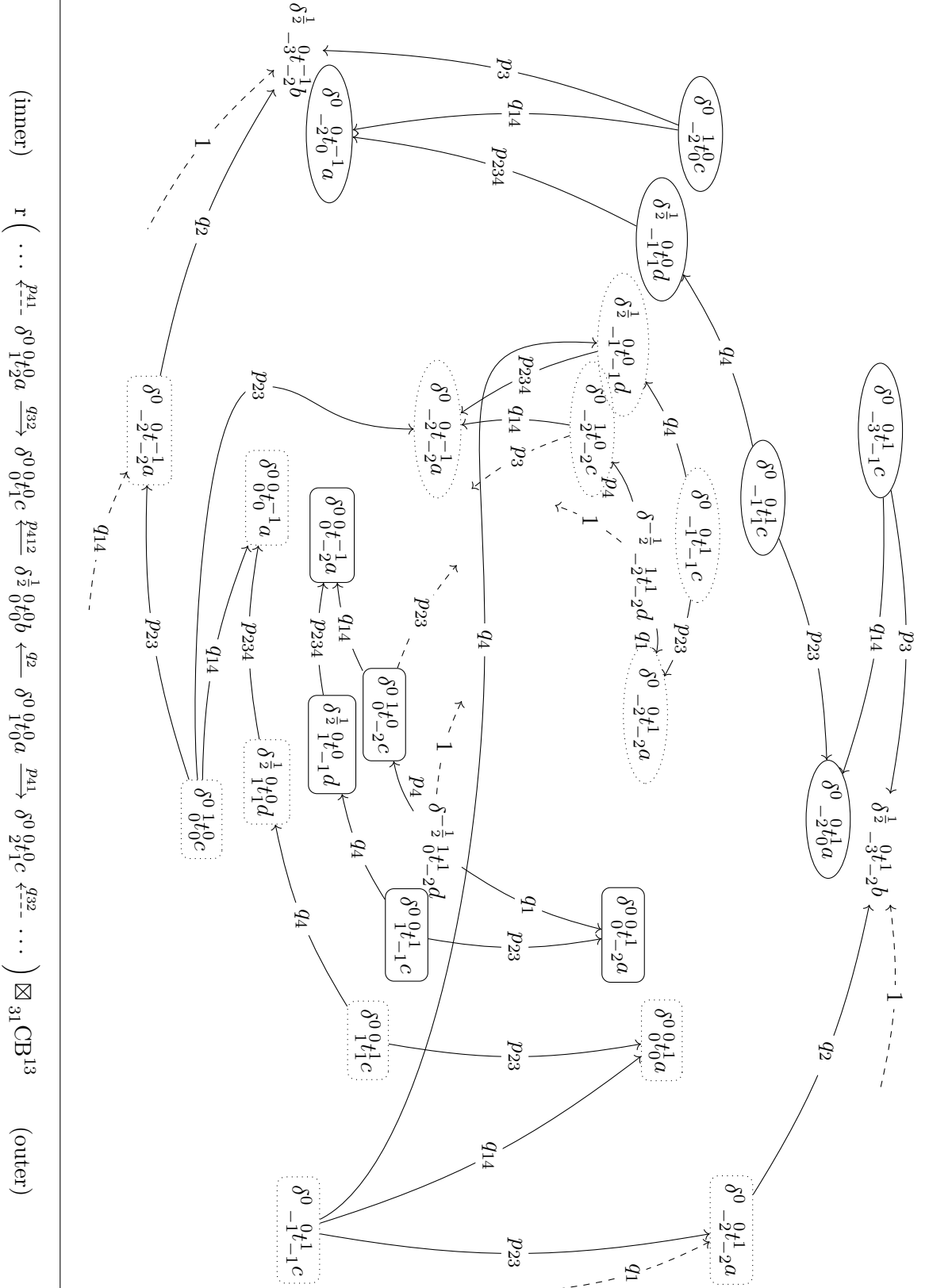

%% file: sections/ConjBimod_ac_raw.tex
\begin{sidewaysfigure}[p]
	\vspace*{435pt}
	\centering
	\setlength{\RadAl}{3cm}
	\setlength{\RadCl}{4cm}
	\setlength{\RadAr}{6cm}
	\setlength{\RadCr}{7cm}
	
	\tikzstyle{tight}=[fill=white,inner sep=3pt,rounded corners]
	
	\begin{tikzpicture}
	\node (Ardtnew) at (120:\RadAr){};
	\node (Crdlnew) at (160:1.5\RadCr){};
	\draw [line width=20pt,lightgray,line cap=round] (Ardtnew) to[bend right=10] (Crdlnew);
	\node (Ardrnew) at (-20:1.5\RadAr){};
	\node (Crdbnew) at (-60:\RadCr){};
	\draw [line width=20pt,lightgray,line cap=round] (Ardrnew) to[bend left=10] (Crdbnew);
	\node (Albrnew) at (20:1.5\RadAl){};
	\node (Clbtnew) at (60:\RadCl){};
	\draw [line width=20pt,lightgray,line cap=round] (Albrnew) to[bend right=10] (Clbtnew);
	\node (Albbnew) at (-120:\RadAl){};
	\node (Clblnew) at (-160:1.5\RadCl){};
	\draw [line width=20pt,lightgray,line cap=round] (Albbnew) to[bend left=10] (Clblnew);
	
	\draw (20:1.5\RadCr)   node(Crar){$\delta^{0}\Alex{0}{-1}{-2}{1}a$};
	\draw (60:\RadCr)      node(Crbt){$\delta^{\frac{1}{2}}\Alex{0}{-2}{-2}{1}b$};
	\draw (120:\RadCr)     node(Crct){$\delta^{0}\Alex{0}{-2}{-1}{1}c$};
	\draw (160:1.5\RadCr)  node[fill=lightgray,rounded corners](Crdl){$\delta^{\frac{1}{2}}\Alex{0}{-2}{-1}{0}d$};
	
	\draw (-160:1.5\RadCr) node(Crbl){$\delta^{\frac{1}{2}}\Alex{0}{-2}{-2}{-1}b$};
	\draw (-120:\RadCr)    node(Crab){$\delta^{0}\Alex{0}{-1}{-2}{-1}a$};
	\draw (-60:\RadCr)     node[fill=lightgray,rounded corners](Crdb){$\delta^{\frac{1}{2}}\Alex{0}{0}{-1}{0}d$};
	\draw (-20:1.5\RadCr)  node(Crcr){$\delta^{0}\Alex{0}{0}{-1}{1}c$};
	
	\draw (20:1.5\RadCl)   node(Clar){$\delta^{0}\Alex{0}{0}{-1}{1}a$};
	\draw (60:\RadCl)      node[fill=lightgray,rounded corners](Clbt){$\delta^{\frac{1}{2}}\Alex{0}{-1}{-1}{1}b$};
	\draw (120:\RadCl)     node(Clct){$\delta^{0}\Alex{0}{-1}{0}{1}c$};
	\draw (160:1.5\RadCl)  node(Cldl){$\delta^{\frac{1}{2}}\Alex{0}{-1}{0}{0}d$};
	
	\draw (-160:1.5\RadCl) node[fill=lightgray,rounded corners](Clbl){$\delta^{\frac{1}{2}}\Alex{0}{-1}{-1}{-1}b$};
	\draw (-120:\RadCl)    node(Clab){$\delta^{0}\Alex{0}{0}{-1}{-1}a$};
	\draw (-60:\RadCl)     node(Cldb){$\delta^{\frac{1}{2}}\Alex{0}{1}{0}{0}d$};
	\draw (-20:1.5\RadCl)  node(Clcr){$\delta^{0}\Alex{0}{1}{0}{1}c$};
	
	\draw (20:1.5\RadAr)   node(Arbr){$\delta^{-\frac{1}{2}}\Alex{0}{0}{0}{1}b$};
	\draw (60:\RadAr)      node(Arat){$\delta^{0}\Alex{0}{-1}{0}{1}a$};
	\draw (120:\RadAr)     node[fill=lightgray,rounded corners](Ardt){$\delta^{-\frac{1}{2}}\Alex{1}{-1}{0}{1}d$};
	\draw (160:1.5\RadAr)  node(Arcl){$\delta^{0}\Alex{1}{-1}{0}{0}c$};
	
	\draw (-160:1.5\RadAr) node(Aral){$\delta^{0}\Alex{0}{-1}{0}{-1}a$};
	\draw (-120:\RadAr)    node(Arbb){$\delta^{-\frac{1}{2}}\Alex{0}{0}{0}{-1}b$};
	\draw (-60:\RadAr)     node(Arcb){$\delta^{0}\Alex{1}{1}{0}{0}c$};
	\draw (-20:1.5\RadAr)  node[fill=lightgray,rounded corners](Ardr){$\delta^{-\frac{1}{2}}\Alex{1}{1}{0}{1}d$};
	
	\draw (20:1.5\RadAl)   node[fill=lightgray,rounded corners](Albr){$\delta^{-\frac{1}{2}}\Alex{0}{-1}{-1}{1}b$};
	\draw (60:\RadAl)      node(Alat){$\delta^{0}\Alex{0}{-2}{-1}{1}a$};
	\draw (120:\RadAl)     node(Aldt){$\delta^{-\frac{1}{2}}\Alex{1}{-2}{-1}{1}d$};
	\draw (160:1.5\RadAl)  node(Alcl){$\delta^{0}\Alex{1}{-2}{-1}{0}c$};
	
	\draw (-160:1.5\RadAl) node(Alal){$\delta^{0}\Alex{0}{-2}{-1}{-1}a$};
	\draw (-120:\RadAl)    node[fill=lightgray,rounded corners](Albb){$\delta^{-\frac{1}{2}}\Alex{0}{-1}{-1}{-1}b$};
	\draw (-60:\RadAl)     node(Alcb){$\delta^{0}\Alex{1}{0}{-1}{0}c$};
	\draw (-20:1.5\RadAl)  node(Aldr){$\delta^{-\frac{1}{2}}\Alex{1}{0}{-1}{1}d$};

	\draw[->] (Ardt) to [bend right=10] node[fill=lightgray,rounded corners]{$1$}      (Crdl);
	\draw[->] (Arcl) to [bend right=10] node[tight]{$p_3$}    (Crbl);
	\draw[->] (Ardr) to [bend left=10] node[fill=lightgray,rounded corners]{$1$}       (Crdb);
	\draw[->] (Arcb) to [bend left=10] node[tight]{$p_{23}$}  (Crab);
	
	\draw[->] (Albr) to [bend right=10] node[fill=lightgray,rounded corners]{$1$}      (Clbt);
	\draw[->] (Aldr) to [bend right=10] node[tight]{$q_1$}    (Clar);
	\draw[->] (Albb) to [bend left=10]  node[fill=lightgray,rounded corners]{$1$}      (Clbl);
	\draw[->] (Alcb) to [bend left=10]  node[tight]{$q_{14}$} (Clab);

	\draw[->,dashed] (Aldt) to [bend right=10] node[tight]{$1$} +(-1.5,-2);
	\draw[->,dashed] (Alcl) to [bend right=10] node[tight]{$p_3$} +(1.25,-2);
	\draw[->,dashed] (Alcb) to [bend left=10] node[tight]{$p_{32}$} +(-3,1);
	\draw[->,dashed] (Aldr) to [bend left=10] node[tight]{$1$} +(-3,0);
	
	\draw[<-,dashed] (Cldl) to [bend left=10] node[tight]{$1$} +(1.5,2);
	\draw[<-,dashed] (Clbl) to [bend left=10] node[tight]{$p_3$} +(-1.25,2);
	\draw[<-,dashed] (Clab) to [bend right=10] node[tight]{$p_{32}$} +(3,-1);
	\draw[<-,dashed] (Cldb) to [bend right=10] node[tight]{$1$} +(3,0);
	
	\draw[->,dashed] (Arbb) to [bend left=10] node[tight]{$1$} +(-3,1.5);
	\draw[->,dashed] (Arcb) to [bend left=10] node[tight]{$q_{14}$} +(-5,0.5);
	\draw[->,dashed] (Ardr) to [bend right=10] node[tight]{$q_1$} +(-1,4);
	\draw[->,dashed] (Arbr) to [bend right=10] node[tight]{$1$} +(-4,1);
	
	\draw[<-,dashed] (Crbl) to [bend right=10] node[tight]{$1$} +(3,-2.5);
	\draw[<-,dashed] (Crab) to [bend right=10] node[tight]{$q_{14}$} +(5,-1);
	\draw[<-,dashed] (Crar) to [bend left=10] node[tight]{$q_1$} +(1,-4);
	\draw[<-,dashed] (Crbt) to [bend left=10] node[tight]{$1$} +(4,-0.5);

	\draw[->,very thick] (Crar) to [bend right=10] node[tight]{$q_{2}$}   (Crbt);
	\draw[->,very thick] (Crct) to [bend left=10] node[tight]{$p_{3}$}    (Crbt);
	\draw[->,very thick] (Crct) to [bend right=10] node[tight]{$q_{4}$}   (Crdl);
	\draw[->,very thick] (Crdl) to [bend right=10] node[tight]{$p_{34}$}  (Crbl);
	\draw[->,very thick] (Crab) to [bend left=10] node[tight]{$q_{2}$}    (Crbl);
	\draw[->,very thick] (Crdb) to [bend left=10] node[tight]{$p_{234}$}  (Crab);
	\draw[->,very thick] (Crcr) to [bend left=10] node[tight]{$q_{4}$}    (Crdb);
	\draw[->,very thick] (Crcr) to [bend right=10] node[tight]{$p_{23}$}  (Crar);
	
	\draw[->,very thick] (Clar) to [bend right=10] node[tight]{$q_{2}~$}   (Clbt);
	\draw[->,very thick] (Clct) to [bend left=10]  node[tight]{$p_{3}$}    (Clbt);
	\draw[->,very thick] (Clct) to [bend right=10] node[tight]{$q_{4}$}   (Cldl);
	\draw[->,very thick] (Cldl) to [bend right=10] node[tight]{$p_{34}$}  (Clbl);
	\draw[->,very thick] (Clab) to [bend left=10]  node[tight]{$q_{2}$}    (Clbl);
	\draw[->,very thick] (Cldb) to [bend left=10]  node[tight]{$p_{234}$}  (Clab);
	\draw[->,very thick] (Clcr) to [bend left=10]  node[tight]{$q_{4}~$}    (Cldb);
	\draw[->,very thick] (Clcr) to [bend right=10] node[tight]{$p_{23}$}  (Clar);
	
	\draw[->,very thick] (Arbr) to [bend right=10] node[tight]{$p_{2}$}   (Arat);
	\draw[->,very thick] (Ardt) to [bend left=10]  node[tight]{$q_{1}$}    (Arat);
	\draw[->,very thick] (Ardt) to [bend right=10] node[tight]{$p_{4}$}   (Arcl);
	\draw[->,very thick] (Arcl) to [bend right=10] node[tight]{$q_{14}$}  (Aral);
	\draw[->,very thick] (Arbb) to [bend left=10]  node[tight]{$p_{2}$}    (Aral);
	\draw[->,very thick] (Arcb) to [bend left=10]  node[tight]{$q_{214}$}  (Arbb);
	\draw[->,very thick] (Ardr) to [bend left=10]  node[tight]{$p_{4}$}    (Arcb);
	\draw[->,very thick] (Ardr) to [bend right=10] node[tight]{$q_{21}$}  (Arbr);
	
	\draw[->,very thick] (Albr) to [bend right=10] node[below]{$p_{2}$}   (Alat);
	\draw[->,very thick] (Aldt) to [bend left=10]  node[below]{$q_{1}$}    (Alat);
	\draw[->,very thick] (Aldt) to [bend right=10] node[below]{~~$p_{4}$}   (Alcl);
	\draw[->,very thick] (Alcl) to [bend right=10] node[right]{$q_{14}$}  (Alal);
	\draw[->,very thick] (Albb) to [bend left=10]  node[above]{~~$p_{2}$}    (Alal);
	\draw[->,very thick] (Alcb) to [bend left=10]  node[above]{$q_{214}$}  (Albb);
	\draw[->,very thick] (Aldr) to [bend left=10]  node[above]{$p_{4}$}    (Alcb);
	\draw[->,very thick] (Aldr) to [bend right=10] node[left]{$q_{21}$}  (Albr);
	
	\end{tikzpicture}
	\\
	\hrulefill
	$$
	\text{(inner)}\qquad
	\rr\Big(
	\begin{tikzcd}[row sep=0.5cm, column sep=0.6cm]
	\cdots
	&
	\delta^{0}\Alex{0}{1}{1}{0}a
	\arrow{r}{q_{32}}
	\arrow[dashed,swap]{l}{p_{41}}
	&
	\delta^{0}\Alex{0}{0}{0}{0}c
	&
	\cdots
	\arrow[swap,dashed]{l}{p_{41}}
	&
	\cdots
	&
	\delta^{0}\Alex{0}{0}{0}{0}a
	\arrow[dashed,swap]{l}{q_{32}}
	\arrow{r}{p_{41}}
	&
	\delta^{0}\Alex{0}{1}{1}{0}c
	&
	\cdots
	\arrow[dashed,swap]{l}{q_{32}}
	\end{tikzcd}
	\Big)\boxtimes\ConjBimod
	\qquad\text{(outer)}
	$$
	\caption{The computation of the action of the bimodule~$\ConjBimod$ on the last two curve segments in the proof of Lemma~\ref{lem:ConjBimod:ActionOnHorizontalComponents}}\label{fig:ConjugationBimodule:ac:raw}
\end{sidewaysfigure}
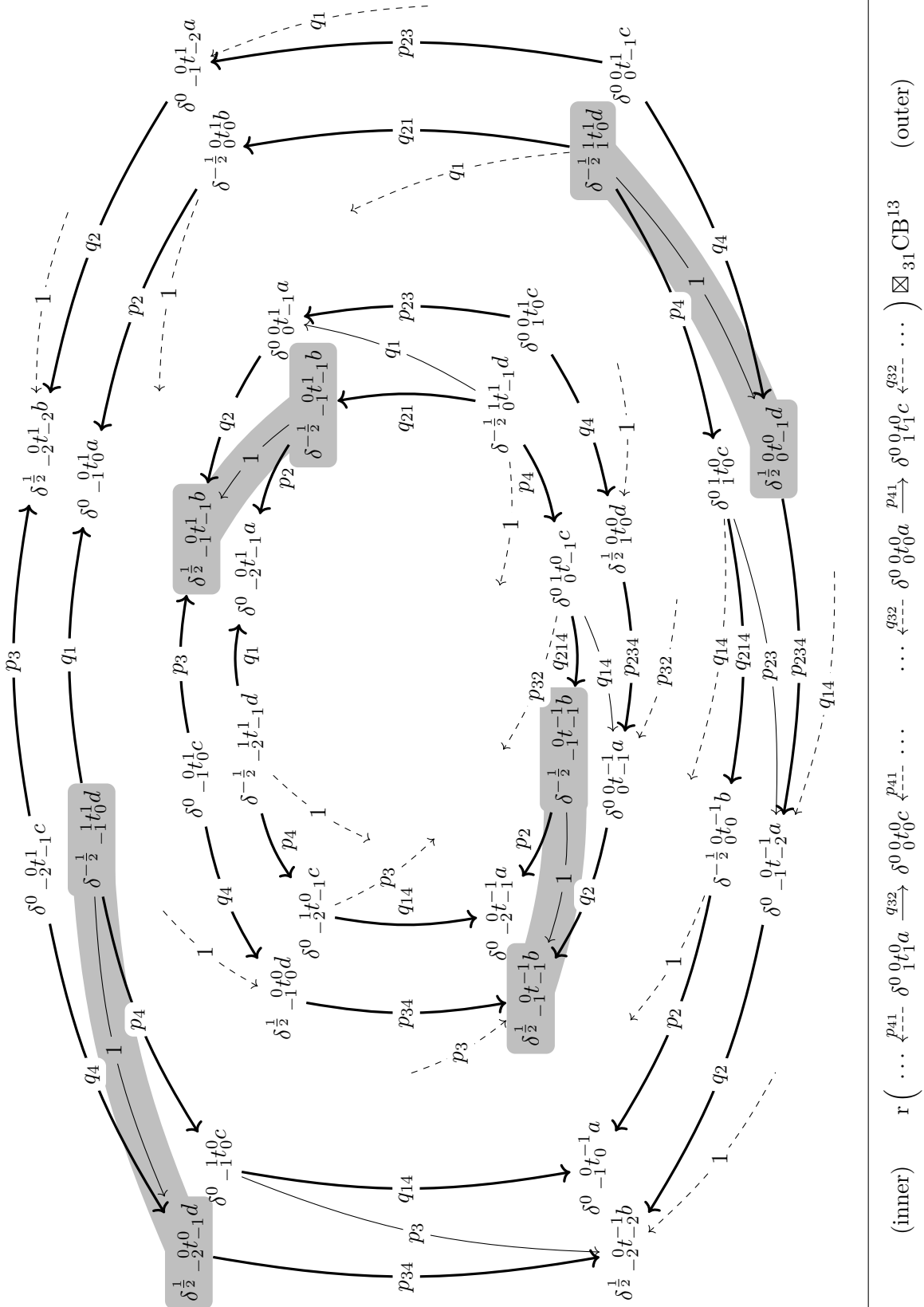

%% file: sections/ConjBimod_ac_cancelled.tex
\begin{sidewaysfigure}[p]
	\vspace*{435pt}
	\centering
	\setlength{\RadAl}{3cm}
	\setlength{\RadCl}{4cm}
	\setlength{\RadAr}{6cm}
	\setlength{\RadCr}{7cm}
	
	 \tikzstyle{nodeRac}=[draw,rectangle,rounded corners,inner sep=4pt]
	\tikzstyle{nodeRacS}=[draw,rectangle,rounded corners,inner sep=3pt,fill=white,dotted]
	 \tikzstyle{nodeBac}=[draw,rectangle,rounded corners,inner sep=4pt,dotted]
	\tikzstyle{nodeBacS}=[draw,rectangle,rounded corners,inner sep=3pt,fill=white]
	
	 \tikzstyle{nodeTac}=[draw,ellipse,rounded corners,inner sep=4pt]
	\tikzstyle{nodeTacS}=[draw,ellipse,rounded corners,inner sep=3pt,fill=white,dotted]
	 \tikzstyle{nodeLac}=[draw,ellipse,rounded corners,inner sep=4pt,dotted]
	\tikzstyle{nodeLacS}=[draw,ellipse,rounded corners,inner sep=3pt,fill=white]
	
	\tikzstyle{tight}=[fill=white,inner sep=3pt,rounded corners]
	
	\begin{tikzpicture}
	
	\draw (20:1.5\RadCr)   node[nodeRac](Crar){$\delta^{0}\Alex{0}{-1}{-2}{1}a$};
	\draw (20:1.5\RadCr)   node[nodeRacS](CrarS){$\delta^{0}\Alex{0}{-1}{-2}{1}a$};
	
	\draw (60:\RadCr)      node(Crbt){$\delta^{\frac{1}{2}}\Alex{0}{-2}{-2}{1}b$};
	
	\draw (120:\RadCr)     node[nodeTac](Crct){$\delta^{0}\Alex{0}{-2}{-1}{1}c$};
	\draw (120:\RadCr)     node[nodeTacS](CrctS){$\delta^{0}\Alex{0}{-2}{-1}{1}c$};
	
	\draw (-160:1.5\RadCr) node(Crbl){$\delta^{\frac{1}{2}}\Alex{0}{-2}{-2}{-1}b$};
	
	\draw (-120:\RadCr)    node[nodeBac](Crab){$\delta^{0}\Alex{0}{-1}{-2}{-1}a$};
	\draw (-120:\RadCr)    node[nodeBacS](CrabS){$\delta^{0}\Alex{0}{-1}{-2}{-1}a$};

	\draw (-20:1.5\RadCr)  node[nodeRac](Crcr){$\delta^{0}\Alex{0}{0}{-1}{1}c$};
	\draw (-20:1.5\RadCr)  node[nodeRacS](CrcrS){$\delta^{0}\Alex{0}{0}{-1}{1}c$};

	\draw (20:1.5\RadCl)   node[nodeRac](Clar){$\delta^{0}\Alex{0}{0}{-1}{1}a$};
	\draw (20:1.5\RadCl)   node[nodeRacS](ClarS){$\delta^{0}\Alex{0}{0}{-1}{1}a$};
	
	\draw (120:\RadCl)     node[nodeTac](Clct){$\delta^{0}\Alex{0}{-1}{0}{1}c$};
	\draw (120:\RadCl)     node[nodeTacS](ClctS){$\delta^{0}\Alex{0}{-1}{0}{1}c$};
	
	\draw (160:1.5\RadCl)  node(Cldl){$\delta^{\frac{1}{2}}\Alex{0}{-1}{0}{0}d$};
	
	\draw (-120:\RadCl)    node[nodeBac](Clab){$\delta^{0}\Alex{0}{0}{-1}{-1}a$};
	\draw (-120:\RadCl)    node[nodeBacS](ClabS){$\delta^{0}\Alex{0}{0}{-1}{-1}a$};
	
	\draw (-60:\RadCl)     node(Cldb){$\delta^{\frac{1}{2}}\Alex{0}{1}{0}{0}d$};
	
	\draw (-20:1.5\RadCl)  node[nodeRac](Clcr){$\delta^{0}\Alex{0}{1}{0}{1}c$};
	\draw (-20:1.5\RadCl)  node[nodeRacS](ClcrS){$\delta^{0}\Alex{0}{1}{0}{1}c$};

	\draw (20:1.5\RadAr)   node(Arbr){$\delta^{-\frac{1}{2}}\Alex{0}{0}{0}{1}b$};
	
	\draw (60:\RadAr)      node[nodeTac](Arat){$\delta^{0}\Alex{0}{-1}{0}{1}a$};
	\draw (60:\RadAr)      node[nodeTacS](AratS){$\delta^{0}\Alex{0}{-1}{0}{1}a$};
	
	\draw (160:1.5\RadAr)  node[nodeLac](Arcl){$\delta^{0}\Alex{1}{-1}{0}{0}c$};
	\draw (160:1.5\RadAr)  node[nodeLacS](ArclS){$\delta^{0}\Alex{1}{-1}{0}{0}c$};
		
	\draw (-160:1.5\RadAr) node[nodeLac](Aral){$\delta^{0}\Alex{0}{-1}{0}{-1}a$};
	\draw (-160:1.5\RadAr) node[nodeLacS](AralS){$\delta^{0}\Alex{0}{-1}{0}{-1}a$};
	
	\draw (-120:\RadAr)    node(Arbb){$\delta^{-\frac{1}{2}}\Alex{0}{0}{0}{-1}b$};
	
	\draw (-60:\RadAr)     node[nodeBac](Arcb){$\delta^{0}\Alex{1}{1}{0}{0}c$};
	\draw (-60:\RadAr)     node[nodeBacS](ArcbS){$\delta^{0}\Alex{1}{1}{0}{0}c$};

	\draw (60:\RadAl)      node[nodeTac](Alat){$\delta^{0}\Alex{0}{-2}{-1}{1}a$};
	\draw (60:\RadAl)      node[nodeTacS](AlatS){$\delta^{0}\Alex{0}{-2}{-1}{1}a$};
	
	\draw (120:\RadAl)     node(Aldt){$\delta^{-\frac{1}{2}}\Alex{1}{-2}{-1}{1}d$};
	
	\draw (160:1.5\RadAl)  node[nodeLac](Alcl){$\delta^{0}\Alex{1}{-2}{-1}{0}c$};
	\draw (160:1.5\RadAl)  node[nodeLacS](AlclS){$\delta^{0}\Alex{1}{-2}{-1}{0}c$};
	
	\draw (-160:1.5\RadAl) node[nodeLac](Alal){$\delta^{0}\Alex{0}{-2}{-1}{-1}a$};
	\draw (-160:1.5\RadAl) node[nodeLacS](AlalS){$\delta^{0}\Alex{0}{-2}{-1}{-1}a$};
	
	\draw (-60:\RadAl)     node[nodeBac](Alcb){$\delta^{0}\Alex{1}{0}{-1}{0}c$};
	\draw (-60:\RadAl)     node[nodeBacS](AlcbS){$\delta^{0}\Alex{1}{0}{-1}{0}c$};
	\draw (-20:1.5\RadAl)  node(Aldr){$\delta^{-\frac{1}{2}}\Alex{1}{0}{-1}{1}d$};

	\draw[->] (Arcl) to [bend right=10] node[tight]{$p_3$}    (Crbl);
	\draw[->] (Arcb) to [bend left=10] node[tight]{$p_{23}$}  (Crab);
	
	\draw[->] (Aldr) to [bend right=10] node[tight]{$q_1$}    (Clar);
	\draw[->] (Alcb) to [bend left=10]  node[tight]{$q_{14}$} (Clab);

	\draw[->,dashed] (Aldt) to [bend right=10] node[tight]{$1$} +(-1.5,-2);
	\draw[->,dashed] (Alcl) to [bend right=10] node[tight]{$p_3$} +(1.25,-2);
	\draw[->,dashed] (Alcb) to [bend left=10] node[tight]{$p_{32}$} +(-3,1);
	\draw[->,dashed] (Aldr) to [bend left=10] node[tight]{$1$} +(-3,0);
	
	\draw[<-,dashed] (Cldl) to [bend left=10] node[tight]{$1$} +(1.5,2);
	\draw[<-,dashed] (Clab) to [bend right=10] node[tight]{$p_{32}$} +(3,-1);
	\draw[<-,dashed] (Cldb) to [bend right=10] node[tight]{$1$} +(3,0);
	
	\draw[->,dashed] (Arbb) to [bend left=10] node[tight]{$1$} +(-3,1.5);
	\draw[->,dashed] (Arcb) to [bend left=10] node[tight]{$q_{14}$} +(-5,0.5);
	\draw[->,dashed] (Arbr) to [bend right=10] node[tight]{$1$} +(-4,1);
	
	\draw[<-,dashed] (Crbl) to [bend right=10] node[tight]{$1$} +(3,-2.5);
	\draw[<-,dashed] (Crab) to [bend right=10] node[tight]{$q_{14}$} +(5,-1);
	\draw[<-,dashed] (Crar) to [bend left=10] node[tight]{$q_1$} +(1,-4);
	\draw[<-,dashed] (Crbt) to [bend left=10] node[tight]{$1$} +(4,-0.5);
	
	\draw[->] (Clct) to [bend left=10]  node[tight]{$p_{23}$}    (Alat);
		
	\draw[->] (Cldl) to [bend right=10]  node[tight]{$p_{234}$}    (Alal);
	\draw[<-,dashed] (Alal) to [bend left=10] node[tight]{$p_{23}$} +(-3.25,2);
		
	\draw[->] (Crct) to [bend right=10] node[tight]{$q_{14}$}   (Arat);
	
	
	\draw[->] (Crcr) to [bend left=10] node[tight]{$q_{214}$}    (Arbr);
	\draw[->,dashed] (Crcr) to [bend left=10] node[tight]{$q_{14}$} +(-2.5,4);
	
	\draw[->] (Crar) to [bend right=10] node[tight]{$q_{2}$}   (Crbt);
	\draw[->] (Crct) to [bend left=10] node[tight]{$p_{3}$}    (Crbt);
	\draw[->] (Crab) to [bend left=10] node[tight]{$q_{2}$}    (Crbl);
	
	\draw[->] (Crcr) to [bend right=10] node[tight]{$p_{23}$}  (Crar);
	\draw[->] (Clct) to [bend right=10] node[tight]{$q_{4}$}   (Cldl);
	\draw[->] (Cldb) to [bend left=10]  node[tight]{$p_{234}$}  (Clab);
	\draw[->] (Clcr) to [bend left=10]  node[tight]{$q_{4}~$}    (Cldb);
	\draw[->] (Clcr) to [bend right=10] node[tight]{$p_{23}$}  (Clar);
	
	\draw[->] (Arbr) to [bend right=10] node[tight]{$p_{2}$}   (Arat);
	\draw[->] (Arcl) to [bend right=10] node[tight]{$q_{14}$}  (Aral);
	\draw[->] (Arbb) to [bend left=10]  node[tight]{$p_{2}$}    (Aral);
	\draw[->] (Arcb) to [bend left=10]  node[tight]{$q_{214}$}  (Arbb);	

	\draw[->] (Aldt) to [bend left=10]  node[below]{$q_{1}$}    (Alat);
	\draw[->] (Aldt) to [bend right=10] node[below]{~~$p_{4}$}   (Alcl);
	\draw[->] (Alcl) to [bend right=10] node[right]{$q_{14}$}  (Alal);
	\draw[->] (Aldr) to [bend left=10]  node[above]{$p_{4}$}    (Alcb);
	
	\end{tikzpicture}
	\\
	\hrulefill
	$$
	\text{(inner)}\qquad
	\rr\Big(
	\begin{tikzcd}[row sep=0.5cm, column sep=0.6cm]
	\cdots
	&
	\delta^{0}\Alex{0}{1}{1}{0}a
	\arrow{r}{q_{32}}
	\arrow[dashed,swap]{l}{p_{41}}
	&
	\delta^{0}\Alex{0}{0}{0}{0}c
	&
	\cdots
	\arrow[swap,dashed]{l}{p_{41}}
	&
	\cdots
	&
	\delta^{0}\Alex{0}{0}{0}{0}a
	\arrow[dashed,swap]{l}{q_{32}}
	\arrow{r}{p_{41}}
	&
	\delta^{0}\Alex{0}{1}{1}{0}c
	&
	\cdots
	\arrow[dashed,swap]{l}{q_{32}}
	\end{tikzcd}
	\Big)\boxtimes\ConjBimod
	\qquad\text{(outer)}
	$$
	\caption{The result of cancelling all shaded identity arrows in Figure~\ref{fig:ConjugationBimodule:ac:raw}
	}\label{fig:ConjugationBimodule:ac:cancelled}
\end{sidewaysfigure}
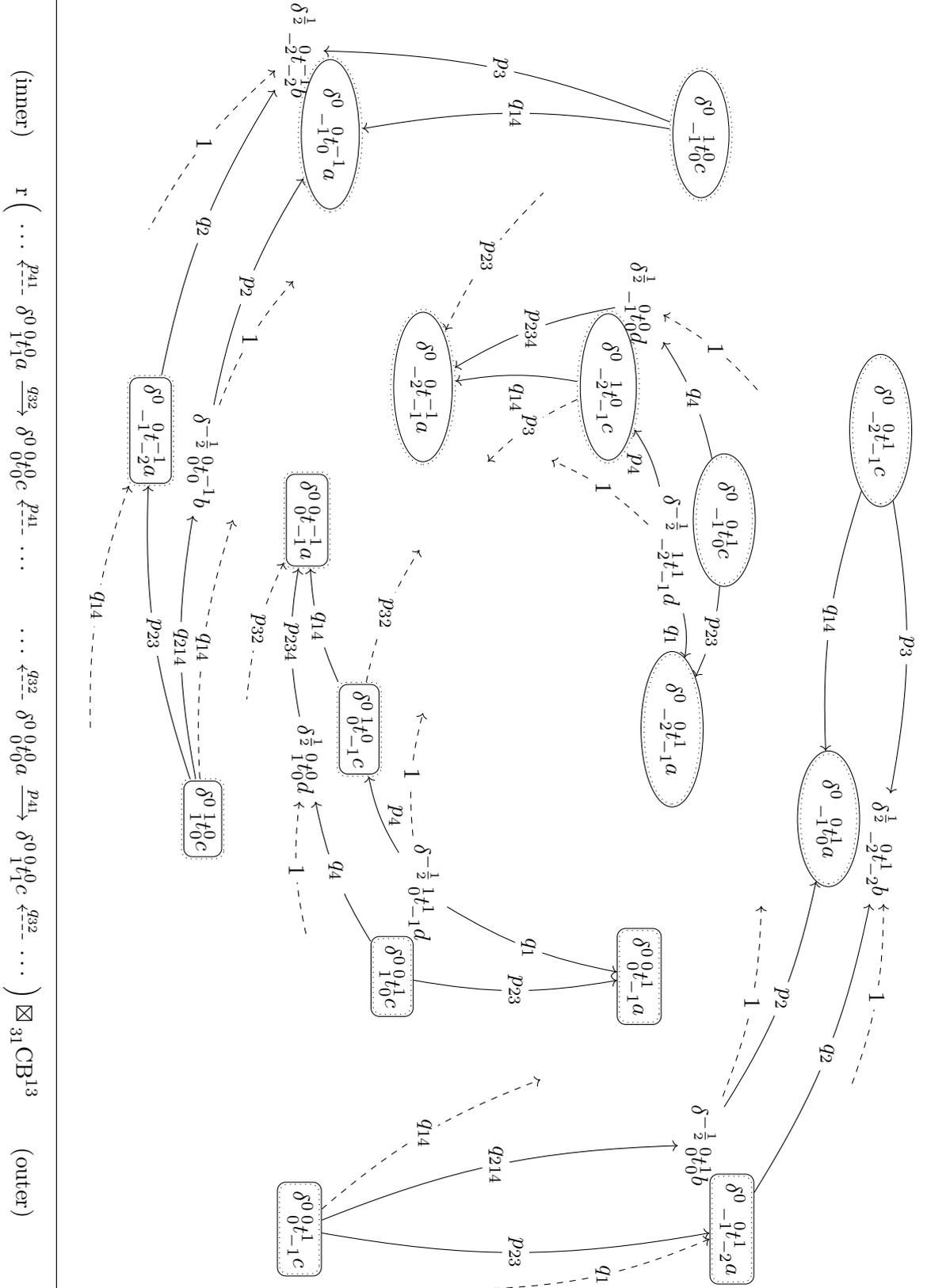